\documentclass{amsart}

\usepackage{amssymb, amsmath,  amsfonts }
\usepackage{color }

\newtheorem{thm}{Theorem}[section]
\newtheorem{prop}[thm]{Proposition}
\newtheorem{lemma}[thm]{Lemma}
\newtheorem{cor}[thm]{Corollary}
\newtheorem{defn}[thm]{Definition}

\newtheorem{remark}[thm]{Remark}
\newtheorem{notation}[thm]{Notation}

\newcommand{\typeb}{ \circ\!\!\;\!\!-\!\!\!-\!\circ\!\!\!\Rightarrow\!\!=\!\!\!\circ}

\newcommand{\cordin}{\scriptstyle{\quad\alpha_1\!\!\!\quad\alpha_2\!\!\!\quad\alpha_3}}
\numberwithin{equation}{section}

\begin{document}{\allowdisplaybreaks[4]

\title[Gromov-Witten invariants for  $G/B$ and Pontryagin product for $\Omega K$]{Gromov-Witten invariants for  $G/B$ and \\Pontryagin product for $\Omega K$}

\author{Naichung Conan Leung}
\address{The Institute of Mathematical Sciences and Department of Mathematics,
           The Chinese University of Hong Kong, Shatin, Hong Kong}
\email{leung@math.cuhk.edu.hk}
\thanks{  }

\author{Changzheng Li}
\address{The Institute of Mathematical Sciences and Department of Mathematics,
           The Chinese University of Hong Kong, Shatin, Hong Kong}

\curraddr{School of Mathematics, Korea Institute for Advanced Study, 87,  Hoegiro, Dongdaemun-gu, Seoul, 130-722, Korea}
\email{czli@kias.re.kr}
\date{
      }




\begin{abstract}
  We give an explicit formula for  ($T$-equivariant) 3-pointed genus zero Gromov-Witten invariants for $G/B$.
   We derive it by finding an  explicit formula for the  Pontryagin product on the equivariant homology of the based loop group $\Omega K$.

\end{abstract}

\maketitle


\section{Introduction }

 A flag variety $G/B$ is the quotient of a simply-connected  simple complex Lie group by its  Borel subgroup and
    it plays very important roles in many different branches of mathematics. There are natural   Schubert cycles inside $G/B$.
     The corresponding Schubert cocycles $\sigma^u$'s
   form a basis of the cohomology ring $H^*(G/B)$. In terms of this basis, the structure coefficients $N_{u, v}^w$ of the intersection
  product,            $$\sigma^u\cdot\sigma^v=\sum_wN_{u, v}^w\sigma^w,$$
         are called   \textit{Schubert structure constants}, which is a direct generalization of the Littlewood-Richardson coefficients for complex Grassmannians.
 When $G=SL(n+1, \mathbb{C})$, the  coefficients $N_{u, v}^{w}$ count suitable Young tableaus (see e.g. \cite{fu22}) or honeycombs \cite{kntao}, \cite{kntao22}. An explicit  formula for $N_{u,v}^w$ in all
    cases are given by Kostant and Kumar \cite{koku}  by considering Kac-Moody groups and an effective algorithm
              is obtained by Duan \cite{duan} via topological methods. Note that a ring presentation of $H^*(G/B, \mathbb{C})$ is given much earlier by Borel \cite{borel} in
              terms of Chern classes of universal bundles over $G/B=K/T$, where $K$ is a maximal compact Lie subgroup of $G$ and $T=K\cap B$ is
                a maximal torus of $K$.

  The (small) quantum cohomology ring of $G/B$, or more generally of any symplectic manifold,     is  introduced by the physicist Vafa \cite{vafa} and
   it is a deformation of the ring structure on  $H^*(G/B)$ by incorporating genus zero Gromov-Witten invariants of $G/B$ into the intersection product.
      As complex vector spaces, the quantum cohomology ring    $QH^*(G/B)$ is isomorphic
      to  $H^*(G/B)\otimes \mathbb{C}[\mathbf{q}]$ with  $\mathbf{q}_\lambda=q_1^{a_1}\cdots q_n^{a_n}$
        for $\lambda=(a_1, \cdots, a_n)\in H_2(G/B, \mathbb{Z})$. The structure coefficients $N_{u, v}^{w, \lambda}$ of the quantum product,
  $$\sigma^u\star \sigma^v=\sum_{w, \lambda}N_{u, v}^{w, \lambda}\mathbf{q}_\lambda\sigma^w,$$
  are called \textit{quantum Schubert structure constants}. As we will see in section \ref{appendequiquantum}, $N_{u, v}^{w, \lambda}=I_{0, 3, \lambda}(\sigma^u, \sigma^v, \sigma^{\omega_0w})$
   is the 3-pointed genus zero Gromov-Witten invariant for $\sigma^u, \sigma^v, \sigma^{\omega_0w}\in H^*(G/B)$, by the definition of the quantum product $\sigma^u\star \sigma^v$.
     We will use the terminology ``quantum Schubert structure constants" instead of ``Gromov-Witten invariants" for $G/B$ throughout this paper, in analog with the classical
    Schubert structure constants.

  Because of the lack of functoriality, the study of the quantum cohomology ring of $G/B$, or more generally partial flag varieties $G/P$, is a   challenging problem.
     A presentation of the ring structure on  $QH^*(G/B)$ is given by Kim \cite{kim} in terms of   Toda lattice
             for the Langlands dual Lie group. There have been a lot of studies of $QH^*(G/P)$ in special cases including
   complex Grassmannians,  partial flag varieties
             of type $A$, isotropic Grassmannians and  two exceptional minuscule homogeneous varieties
               (see e.g. \cite{bu0}, \cite{bu}, \cite{kt1}, \cite{kt2}  and  \cite{cmn} respectively  and the excellent survey \cite{fu11}).
             Nevertheless, the quantum  Schubert structure constants had only been computed explicitly for very few cases, such as
               complex Grassmannians and complete flag varieties of type $A$.

         In this article, we give
            an explicit formula for the (equivariant) quantum Schubert structure constants of the quantum cohomology ring $QH^*(G/B)$ (for partial flag varieties $G/P$,
            see \cite{czli}). We should note that an algorithm to determine the equivariant quantum Schubert structure constants\footnote{
             Explicitly, the equivariant (quantum) Schubert structure constants are homogeneous polynomials.}  was obtained earlier by Mihalcea \cite{mih}
            and he used it to find a characterization of the
             torus-equivariant quantum cohomology $QH^*_T(G/P)$.
            To describe the formula, we consider the  affine Weyl group $W_{\scriptsize\mbox{af}} = W\ltimes Q^\vee$, which is the semidirect product of the Weyl group $W$
               and the coroot lattice $Q^\vee$.
            For any      $x, y\in W_{\scriptsize\mbox{af}}$,  we define $c_{x, [y]}$ and $d_{x, [y]}$ combinatorially,
            which  are  rational functions in simple roots $\alpha_i$. In particular for
             $x=ut_A, y=vt_A$ and $z=wt_{2A+\lambda}$ with  $A=-12n(n+1)\sum_{i=1}^nw_i^\vee$ a sum of fundamental coweights $w_i^\vee$'s, the
                       rational function  $ \sum_{\lambda_1, \lambda_2\in  Q^\vee}
                                c_{x, [ t_{\lambda_1}]}c_{y, [ t_{\lambda_2}]}d_{z, [t_{\lambda_1+\lambda_2}]} 
                                 $  will be shown to be a constant,   provided that  $\langle \lambda, 2\rho\rangle = \ell(u)+\ell(v)-\ell(w)$ where
                                 $\rho$ is the summation
                                  of fundamental weights $w_i$'s. Furthermore, this number coincides with $N_{u, v}^{w, \lambda}$ as stated in
             our main theorem.
 \bigskip

  \noindent \textbf{Main Theorem }  {\itshape
  Let $u, v, w\in W$, $\lambda\in Q^\vee, \lambda\succcurlyeq 0. $ Let $A=-12n(n+1)\sum_{i=1}^nw_i^\vee$. The quantum Schubert structure constant   $N_{u, v}^{w, \lambda}$ for
               $G/B$ is given   by
                       $$  N_{u, v}^{w, \lambda}=\sum_{\lambda_1, \lambda_2\in  Q^\vee}
                                c_{ut_A, [ t_{\lambda_1}]}c_{vt_A, [ t_{\lambda_2}]}d_{wt_{2A+\lambda}, [t_{\lambda_1+\lambda_2}]}, 
                                 $$
          provided that  $\langle \lambda, 2\rho\rangle = \ell(u)+\ell(v)-\ell(w)$ and zero otherwise.
 }
\bigskip

\noindent The above summation does make sense, since there are in fact  only finitely many nonzero terms involved.
 Indeed, the summation over the infinite set $Q^\vee$ can be simplified to the finite set
       $\Gamma\times W$ with $\Gamma=\{(\lambda_1, \lambda_2)~|~ \lambda_1, \lambda_2\succcurlyeq A, \lambda_1+\lambda_2\preccurlyeq 2A+\lambda,
                  \lambda_1 \mbox{ and } \lambda_2 \mbox{ are anti-dominant elements in }  Q^\vee\}$
 and we obtain  $$ N_{u, v}^{w, \lambda}=
                     \!\sum\limits_{(\lambda_1, \lambda_2, v_1)\in \Gamma\times W}
                      \!\!\!\!\!\!\!\!\! c_{ut_A, [ v_1t_{\lambda_1}]}c_{vt_A, [ v_1t_{\lambda_2}]}d_{wt_{2A+\lambda}, [v_1t_{\lambda_1+\lambda_2}]}  
                             .$$

Quantum Schubert structure constants  for $G/P$ can be identified
with certain quantum Schubert structure constants
  for $G/B$ via Peterson-Woodward comparison formula \cite{wo}, the corresponding formula and its applications  are   discussed in     \cite{czli}.

When $v$ is a simple reflection, the equivariant quantum product
$\sigma ^{u}\star\sigma^{v}$ can be given explicitly by the
equivariant quantum Chevalley formula. This formula was originally
stated by Peterson in his unpublished lecture notes \cite{peterson} and has
been proved recently by Mihalcea \cite{mih}. In \cite{mih}, Mihalcea
also showed that the multiplication  in $QH_{T}^{\ast }(G/B)$ is
determined by the equivariant quantum Chevalley formula together
with a few other natural properties (see e.g. Proposition
\ref{crit}). As a consequence, a recursive  algorithm to determine
$N_{u,v}^{w,\lambda}$ was given in \cite{mih}. However, an explicit
formula is still lacking.

In \cite{peterson} Peterson already stated that
$QH_{T}^{\ast}(G/B)$ is ring isomorphic to
$H_{\ast}^{T}(\Omega K)$ after localization, which is called Peterson's Theorem. Here $\Omega K$ is the
based loop group of the maximal compact subgroup $K$ of $G$ and the
\textit{Pontryagin product} defines a ring structure on its
(Borel-Moore) homology group $H_{\ast}^{T}(\Omega K)$. In Peterson's
notes \cite{peterson}, the powerful tool of nil-Hecke ring of
Kostant-Kumar  \cite{koku} was used heavily. Lam and his co-authors
had done many important works along this direction, such as
\cite{lam}, \cite{lamschshi}, \cite{lamshi} and \cite{lamlmsh}. The proofs of Peterson's Theorem in \cite{peterson} are incomplete, and in
\cite{lamshi}, Lam and Shimozono  proved this result, with the help of Peterson's $j$-isomorphism.

The homology $H_{\ast}^{T}(\Omega K)$ is an associative algebra over
$S=H_{T}^{\ast}\left( \mbox{pt} \right)  $ and it has an additive
$S$-basis given by Schubert homology classes
$\{\mathfrak{S}_{x}~|~x\in W_{{\scriptsize \mbox{af}}}^{-}\} $,
where $W_{{\scriptsize \mbox{af}}}^{-}$ is the set of
minimal length representatives of cosets in $W_{{\scriptsize
\mbox{af}}}/W$. We obtain the following explicit formula for the Pontryagin
product of Schubert classes in $H_{\ast }^{T}(\Omega K)$, based on well-known localization formulas due to Arabia \cite{ara} for affine flag manifolds.
\bigskip

\noindent\textbf{Theorem  \!\ref{str}} {\itshape For any Schubert
classes $\mathfrak{S}_{x}$ and $\mathfrak{S}_{y}$ in
$H_{\ast}^{T}(\Omega K),$ the structure coefficients for their
Pontryagin product}
\[
\mathfrak{S}_{x}\mathfrak{S}_{y}=\sum_{z\in W_{{\scriptsize \mbox{af}}}^{-}%
}b_{x,y}^{\,z}\mathfrak{S}_{z}%
\]
{\itshape are given by%
\[
b_{x,y}^{\,z}=\sum_{\lambda,\mu\in Q^{\vee}}c_{x,[t_{\lambda}%
]}c_{y,[t_{\mu}]}d_{z,[t_{\lambda+\mu}]}.
\]
}

In the present paper, we give an alternative proof
 of Peterson's Theorem, in the sense that we find elementary proofs of the following two formulas
of Peterson-Lam-Shimozono \cite{lamshi} on the Pontryagin product of
certain Schubert classes,  by analyzing the combinatorial nature of the summation in the
formula of $b_{x, y}^{\, z}$.

(i) {\itshape For any
}$wt_{\lambda},t_{\mu}\in W_{{\scriptsize
\mbox{af}}}^{-}$,\thinspace
{\itshape one has $\mathfrak{S}_{wt_{\lambda}}\mathfrak{S}_{t_{\mu}%
}=\mathfrak{S}_{wt_{\lambda+\mu}}$};

(ii) {\itshape For any}
$\sigma_{i}t_{\lambda },ut_{\mu}\in W_{{\scriptsize \mbox{af}}}^{-}$
{\itshape with $\sigma _{i}=\sigma_{\alpha_{i}}$, $i\in I$, one has
\[
\mathfrak{S}_{\sigma_{i}t_{\lambda}}\mathfrak{S}_{ut_{\mu}}=(u(w_{i}%
)-w_{i})\mathfrak{S}_{ut_{\lambda+\mu}}\,+\,\sum_{\gamma\in\Gamma_{1}}%
\langle\gamma^{\vee},w_{i}\rangle\mathfrak{S}_{u\sigma_{\gamma}t_{\lambda+\mu
}}\,+\,\sum_{\gamma\in\Gamma_{2}}\langle\gamma^{\vee},w_{i}\rangle
\mathfrak{S}_{u\sigma_{\gamma}t_{\lambda+\mu+\gamma^{\vee}}},
\]
where }  $\Gamma_1=\{\gamma\in R^+ ~|~
\ell(u\sigma_\gamma)=\ell(u)+1\}$ and
            $\Gamma_2=\{\gamma\in R^+ ~|~  \ell(u\sigma_\gamma)=\ell(u)+1-\langle\gamma^\vee, 2\rho\rangle\}$.

Indeed,  Lam and Shimozono noticed that combining the above formulas
   with the criterion of Mihalcea gives a proof of Peterson's theorem.
This in turn shows that any structure  constant of  $QH^*_T(G/B)$  coincides with
certain structure constants  of    $H_*^T(\Omega K)$, which yields our Main Theorem. In particular, there is a choice of
  certain  $b_{x, y}^{\, z}$'s which coincide with the same   $N_{u, v}^{w, \lambda}$. For instance, we can choose one such $A$ as in
  Main theorem  to make  a certain choice $(x, y, z)=(ut_A, vt_A, wt_{2A+\lambda})$.
    In many cases,
           we can replace it by  a  \textit{smaller} one  (see section \ref{secproofofmain} for more details on the choices).
  As a consequence, there are  only a few nonzero terms in the summation for $N_{u, v}^{w, \lambda}$ in many cases.  For instance for $G=SL(3, \mathbb{C})$ with
   $u=v=s_1s_2s_1,     w=s_1s_2$ and $\lambda=\theta^\vee$, where $\theta$ is the highest root (see section \ref{ex11} for more details on the notations),
   it suffices to take $A=-\theta^\vee$ and
         the   summation for $N_{u, v}^{w, \lambda}$
        in fact contains one term only, namely $N_{u, v}^{w, \lambda}=c_{s_0, [s_0]}^2d_{s_2s_0, [s_1s_2s_1t_{-2\theta^\vee}]},$
  where $c_{s_0, [s_0]}=(-1)^1{1\over s_0(\alpha_0)}\big|_{\alpha_0=-\theta}=-{1\over \theta}$ and
      $d_{s_2s_0, [s_1s_2s_1t_{-2\theta^\vee}]}=d_{s_2s_0, [s_0s_1s_2s_1s_0]}=s_0s_1(\alpha_2)s_0s_1s_2s_1(\alpha_0)\big|_{\alpha_0=-\theta}$ $=\theta^2$ by definition. Hence,
      $N_{u, v}^{w, \lambda}=(-{1\over \theta})^2 \cdot \theta^2=1$.
      This coefficient can  also be determined by Mihalcea's algorithm but our formula is more effective. To show the computational power of our formula, we will compute some nontrivial
      coefficients for the higher rank group $Spin(7, \mathbb{C})$.

There could be an alternative way to determine our structure
coefficients by finding polynomial representatives for Schubert
classes. For instance, this approach has been used by  Fomin,
Gelfand and Postnikov for  complete flag varieties of type $A$
\cite{fominGP}. The work of Magyar \cite{mag}  could be relevant for
general cases. See also \cite{fomin}.

This paper is organized as follows. In section 2, we set up the
notations that will be used throughout this article and review some
well-known facts on the theory of Kac-Moody algebras and groups. In
section 3, we define the important quantities $c_{x,[y]}$,
$d_{x,\left[  y\right]  }$ and derive an explicit formula for the
Pontryagin product on $H_{\ast}^{T}(\Omega K)$. In section 4, we
analyze our formula and  prove our main theorem.  In section 5, we give
examples to demonstrate the effectiveness of our formula. The proofs of some propositions stated in
$\mbox{section 4}$ are given in the
the appendix.

\section{Notations}

\subsection{Notations}
We introduce the notations that are used throughout  the paper.

\begin{enumerate}
    \item[$G$:]   a simply-connected simple complex Lie group of rank $n$.
    \item[$B, H$:] $B$ is a Borel subgroup of $G$; $H$ is a   maximal torus of $G$ contained in $B$.
    \item[$K$:]    a maximal compact subgroup of $G$.
     \item[$T$:] $T=K\cap H$ is a maximal torus in $K$.
    \item[$\mathfrak{g},\mathfrak{h}$:]  $\mathfrak{g}=\mbox{Lie}(G)$; $\mathfrak{h}= \mbox{Lie}(H)$.
     \item[$I, I_{\scriptsize\mbox{af}}$:]  $I=\{1, \cdots, n\}$;    $I_{\scriptsize\mbox{af}}=\{0, 1, \cdots, n\}$.
    \item[$R, \Delta$:]  $R$ is the root system of $(\mathfrak{g}, \mathfrak{h})$;  $\Delta=\{\alpha_i~|~i\in I\}$ is a basis of simple roots.
     \item[$R^+$:] $R^+=R\cap \bigoplus_{i\in I}\mathbb{Z}_{\geq 0}\alpha_i$ is the set of the positive roots; $R=\big(-R^+)\bigsqcup R^+$.
     \item[$\alpha_i^\vee\!,\! Q^\vee$:] $\{\alpha_i^\vee~|~i\in I\}$ are the simple coroots; $Q^\vee=\bigoplus_{i\in I} \mathbb{Z}\alpha_i^\vee$ is the coroot lattice.
     \item[$\tilde Q^\vee$:]     $\tilde Q^\vee=\{\mu\in Q^\vee~|~  \langle\mu, \alpha_i\rangle\leq 0,\,\, i\in I\}$ is the set of anti-dominant elements.

     \item[$w_i, \rho$:] $\{w_i~|~i\in I\}$ are the fundamental weights;    $\rho={1\over 2}\sum_{\beta\in R^+}\beta \,\,\big(=\sum_{i\in I}w_i\big)$.
    \item[$w_i^\vee$:]     $\{w_i^\vee~|~i\in I\}$ are the fundamental coweights.
     \item[$W$:] $W=\langle\sigma_{\alpha_i}~:~ i\in I\rangle$ is the Weyl group  of ($\mathfrak{g}, \mathfrak{h})$.

     \item[$\theta,  \omega_0$:] $\theta$ is the highest (long) root of $R$; $\omega_0$ is the longest element in $W$.

       \item[$\mathfrak{g}_{\scriptsize \mbox{af}} $:]      the (untwisted) affine Kac-Moody algebra associated to $\mathfrak{g}$.
        \item[$ \mathfrak{h}_{\scriptsize \mbox{af}}$:]     Cartan subalgebra of $\mathfrak{g}_{\scriptsize \mbox{af}} $.
        \item[$\alpha_0,  \delta$:] $\alpha_0$ is the affine simple root; $\delta=\alpha_0+\theta$ is the null root.
         \item[$ R^+_{\scriptsize \mbox{re}}$:] $R^+_{\scriptsize \mbox{re}}=\{\alpha+m\delta~|~ \alpha\in R, m\in\mathbb{Z}^+\}\cup R^+$ is the set of positive real roots.

       \item[ $\mathcal{S}, Y$:]  $\mathcal{S}=\{\sigma_{\alpha_i}~|~i\in I_{\scriptsize \mbox{af}}\}$;  $Y\subset \Delta$ is a subset.

       \item[ $ W_{\scriptsize \mbox{af}}$:] the Weyl group of $\mathfrak{g}_{\scriptsize \mbox{af}}$;  $ W_{\scriptsize \mbox{af}}=\langle \sigma_{\alpha_i}: i\in I_{\scriptsize \mbox{af}}\rangle$.

       \item[$W_{\scriptsize \mbox{af}, Y} $:]  the subgroup of $W_{{\scriptsize \mbox{af}}}$ generated by $ \{\sigma_{\alpha}~|~ \alpha\in Y\}$.

       \item[$ W_{\scriptsize \mbox{af}}^{Y}$:]
                                          the subset $\{x\in W_{\scriptsize \mbox{af}}~|~ \ell(x)\leq \ell(y), \forall y\in xW_{\scriptsize \mbox{af}, Y}\}$ of $W_{\scriptsize \mbox{af}}$.

     \item[$\mathcal{G} $:]   the Kac-Moody group associated to the Kac-Moody algebra $\mathfrak{g}_{\scriptsize \mbox{af}}$.
           \item[$ \mathcal{B} $:]  the standard Borel subgroup of $\mathcal{G}$.

    \item[$ \mathcal{P}_Y$:]   the standard parabolic subgroup of $\mathcal{G}$ associated to
                                          $Y$;  $\mathcal{P}_Y\supset \mathcal{B}$.

      \item[$W_{\scriptsize \mbox{af}}^-, \mathcal{P}_0 $:]    $W_{\scriptsize \mbox{af}}^-=W_{\scriptsize \mbox{af}}^\Delta$;\quad $\mathcal{P}_0=\mathcal{P}_\Delta$.

    \item[$LK, \Omega K$:]    $LK =\{f:\mathbb{S}^1\to K~|~ f \mbox{ is smooth } \}$;  $\Omega K=\{f\in LK~|~ f(1_{\mathbb{S}^1})=1_K\}$.

      \item[$S, \hat S$:] $S=\mathbb{Q}[\alpha_1, \cdots,\alpha_n]$;\quad $\hat S=\mathbb{Q}[\alpha_0, \alpha_1, \cdots,\alpha_n]$.
      \item[$q_\lambda$:] $q_\lambda=q_1^{a_1}\cdots q_n^{a_n}$ for $\lambda=\sum_{i=1}^n a_i\alpha_i^\vee\in Q^\vee$.
          \item[$\sigma_i, \sigma_\beta$:] $\sigma_i=\sigma_{\alpha_i}$ is a simple reflection. $\sigma_\beta$ is a reflection for $\beta\in R^+_{\scriptsize\mbox{re}}\bigsqcup \big(-R^+_{\scriptsize\mbox{re}}\big)$.
      \item[$\sigma_u, \sigma^u$:] Schubert classes for $G/B$, where $u\in W$. $\sigma_u\in H_{2\ell(u)}(G/B, \mathbb{Z})$ and $\sigma^u\in H^{2\ell(u)}(G/B, \mathbb{Z})$
                                            are defined in section 6.3.
     \item[$\mathfrak{S}_x, \mathfrak{S}^x$:] Schubert classes for $\mathcal{G}/\mathcal{P}_Y$.
                                 $\mathfrak{S}_x\in H_{2\ell(x)}(\mathcal{G}/\mathcal{P}_Y, \mathbb{Z})$ and $\mathfrak{S}^x\in H^{2\ell(x)}(\mathcal{G}/\mathcal{P}_Y, \mathbb{Z})$
                                            are defined in section 6.4.

       \item[$c_{x, [y]} $:] defined in section 3.1 for any $x, y\in  W_{\scriptsize \mbox{af}}^-$.
       \item[$d_{x, [y]} $:] defined in section 3.1 for any $x, y\in  W_{\scriptsize \mbox{af}}^-$.
       \item[$c_{x, y}'$:] $c_{x, y}'=c_{x, y}\big|_{\alpha_0=-\theta}$ with $c_{x, y}$ defined in section 3.1.

       \item[$\Gamma_1$:] $\Gamma_1(u)=\{\gamma\in R^+ ~|~   \ell(u\sigma_\gamma)=\ell(u)+1\}$, or  simply  $\Gamma_1=\Gamma_1(u)$.

         \item[$ \Gamma_2$:]  $\Gamma_2(u)=\{\gamma\in R^+ ~|~  \ell(u\sigma_\gamma)=\ell(u)+1-\langle\gamma^\vee, 2\rho\rangle\}$, or simply  $\Gamma_2=\Gamma_2(u)$.

\end{enumerate}

\subsection{Some more explanations}

See \cite{kac} and \cite{kumar}
    for the meaning of   the notations as in section 2.1 as well as  the theory of Kac-Moody algebras and groups.

The   fundamental weights $\{w_i~|~i\in I\}$   are the dual basis to
the simple coroots $\{\alpha_i^\vee~|~i\in I\}$
     with respect to the natural pairing $\langle\cdot, \cdot\rangle
       :\mathfrak{h}\times\mathfrak{h}^*\rightarrow \mathbb{C}$. The simple reflections  $\{\sigma_i=\sigma_{\alpha_i}~|~i\in I\}$ act on $\mathfrak{h}$ by
        $\sigma_i(\lambda)=\lambda-\langle\lambda, \alpha_i\rangle\alpha_i^\vee$ for $\lambda\in\mathfrak{h}$.   Therefore the Weyl group $W$,
         which is generated by the simple reflections,  acts on $\mathfrak{h}$ and $\mathfrak{h}^*$ naturally. Note that $R=W\cdot \Delta$. For any $\gamma\in R$,
          $\gamma=w(\alpha_i)$ for some $w\in W$ and $i\in I$.  We can well define $\gamma^\vee=w(\alpha_i^\vee)$, which is independent of the expressions of
           $\gamma$.

     The Weyl group $W_{\scriptsize \mbox{af}}$ of $\mathfrak{g}_{\scriptsize \mbox{af}}$ is in fact an affine group,
      $W_{\scriptsize \mbox{af}}=W\ltimes Q^\vee$,   where we denote
  $t_\lambda$\footnote{ The notation $t_\lambda$ is used instead of $t_{\nu(\lambda)}$ as in chapter 6 of \cite{kac}.}
   the image
  of $\lambda\in Q^\vee$ in $W_{\scriptsize \mbox{af}}$ (by abusing notations). To be more precise, one has $\sigma_\beta=\sigma_\alpha t_{m\alpha^\vee}$ for
   $\beta=\alpha+m\delta\in R_{\scriptsize \mbox{re}}=\big(-R^{+}_{\scriptsize \mbox{re}}\big)\bigsqcup R^+_{\scriptsize \mbox{re}}$. In particular,
    $\sigma_{\alpha_0}=\sigma_\theta t_{-\theta^\vee}$.
    Given $w\in W, \lambda\in Q^\vee, \gamma\in \bigoplus_{i\in I}\mathbb{Z}\alpha_i$
   and $m\in\mathbb{Z}$, we have $t_{w\cdot \lambda}=wt_\lambda w^{-1}$ and  the following action
             $$wt_\lambda\cdot(\gamma+m\delta)=w\cdot \gamma+(m-\langle \lambda, \gamma\rangle)\delta.$$
     Since  $\big(W_{\scriptsize \mbox{af}}, \mathcal{S}\big)$ is a Coxeter system,  we can define the length function
     $\ell: W_{\scriptsize \mbox{af}}\rightarrow \mathbb{Z}_{\geq 0}$ and the Bruhat order $(W_{\scriptsize \mbox{af}}, \preccurlyeq)$ (see e.g. \cite{hump}). We use the following notation
         $$x=[\sigma_{\beta_1}\cdots\sigma_{\beta_r}]_{\scriptsize\mbox{red}} ,$$
      whenever $(\sigma_{\beta_1},\cdots, \sigma_{\beta_r})$ is a reduced decomposition of $x\in W_{\scriptsize \mbox{af}}$; that is,  $r=\ell(x)$, $x=\sigma_{\beta_1}\cdots\sigma_{\beta_r}$ and
       $\beta_i$'s are simple roots. (It is possible that $\beta_i=\beta_j$ for $i\neq j$.) This notation will  also be used throughout this article.

  Explicitly, the affine Kac-Moody group $\mathcal{G}$ is realized as a central extension by $\mathbb{C}^*$
      of the loop group consisting of the $\mathbb{C}((t))$-rational points $G(\mathbb{C}((t)))$ of $G$ extended by one dimensional complex torus.
   For each subset $Y\subset \Delta$, there is a standard parabolic subgroup $P_Y\subset \mathcal{G}$ corresponding to $Y$.
   In particular, $\mathcal{B}=\mathcal{P}_\emptyset$ and we denote  $\mathcal{P}_0=\mathcal{P}_\Delta$.
     For our purpose of studying
      the generalized flag varieties $\mathcal{G}/\mathcal{B}$ and $\mathcal{G}/\mathcal{P}_0$,
       the  group $\mathcal{G}$  can be taken simply to be $\mathcal{G}=G(\mathbb{C}((t)))$.
      That is, $\mathcal{G}=\mbox{Mor}(\mathbb{C}^*, G)$. As a consequence, $\mathcal{P}_0=G(\mathbb{C}[[t]])=\mbox{Mor}(\mathbb{C}, G)$ and $\mathcal{B}=\{f\in \mathcal{P}_0~|~ f(0)\in B\}$.

     In the present paper, we only consider the following two cases: $Y=\emptyset$  and $Y=\Delta$.
     Note that   $W_{\scriptsize \mbox{af}, \emptyset}=\{1\}$,
              $W_{\scriptsize \mbox{af}}^\emptyset= W_{\scriptsize \mbox{af}}$ and $W_{\scriptsize \mbox{af}, \Delta}=W$. We denote
               $W_{\scriptsize \mbox{af}}^-=W_{\scriptsize \mbox{af}}^\Delta$.

\section{Pontryagin product on equivariant homology of $\Omega K$}
  The $T$-equivariant (Borel-Moore) homology $H^T_*(\Omega K)$ of based loop group $\Omega K$ is a module
   over  $S=H^*_T(\mbox{pt})=\mathbb{Q}[\alpha_1, \cdots, \alpha_n]$ with an $S$-basis of Schubert classes
   $\{\mathfrak{S}_x~|~ x\in W_{\scriptsize\mbox{af}}^-\}$, where   $W_{\scriptsize\mbox{af}}^-$ is the set of
     minimal length representatives of cosets in $W_{\scriptsize\mbox{af}}/W$.  $T$ acts on $\Omega K$ by pointwise conjugation.
    The Pontryagin product $\Omega K\times \Omega K\rightarrow \Omega K$,
        given by $(f\cdot g)(t)=f(t)\cdot g(t)$,
          is associative and $T$-equivariant. Therefore, it induces a product map
            $H^T_*(\Omega K)\otimes H^T_*(\Omega K)\rightarrow H^T_*(\Omega K)$,  making $H^T_*(\Omega K)$   an associative $S$-algebra.
              The structure  constants $b_{x, y}^{\,z}\in S$ are defined by
                $$\mathfrak{S}_x\mathfrak{S}_y=\sum\nolimits_{z\in W_{\scriptsize\mbox{af}}^-}b_{x, y}^{\,z}\mathfrak{S}_z$$
           for  $x,y\in W_{\scriptsize\mbox{af}}^-.$
       The main result of this section is Theorem \ref{str},  giving an explicit formula for the Pontryagin product as
         $b_{x,y}^{\,z}=\sum_{\lambda, \mu\in Q^\vee}c_{x, [t_\lambda]}c_{y, [t_\mu]}d_{z, [t_{\lambda+\mu}]}$
       in which the summation is in fact only over finitely many non-zero terms   and  $c_{x, [y]}, d_{x, [y]}$ are defined combinatorially  as below.
       Due to   Peterson's Theorem which was proved by Lam and Shimozono,  these structure coefficients $b_{x, y}^{\, z}$
          correspond
          to quantum Schubert structure constants  for the equivariant quantum cohomology $QH^*_T(G/B)$.

\subsection{Definitions and properties of  $c_{x, [y]}$ and $d_{x, [y]}$}

 \begin{defn}\label{de11}For any {\upshape $x, y\in  W_{\scriptsize\mbox{af} }$}, we define the homogeneous rational  function  $c_{x, y}=c_{x, y}(\alpha_0, \cdots, \alpha_n)
   \in \mathbb{Q}[\alpha_0^{\pm1},\alpha_1^{\pm1},\cdots, \alpha_n^{\pm1}]$ as follows. Let  {\upshape $x=[\sigma_{\beta_1}\cdots\sigma_{\beta_m}]_{ \scriptsize\mbox{red}}$}.
 If $y\not\preccurlyeq x$,  then $c_{x, y}=0$; if $y\preccurlyeq x$, then
      $$c_{x, y}:=(-1)^m\sum \big(\sigma_{\beta_1}^{\varepsilon_1}(\beta_1)\sigma_{\beta_1}^{\varepsilon_1}\sigma_{\beta_2}^{\varepsilon_2}(\beta_2)\cdots
             \sigma_{\beta_1}^{\varepsilon_1}\cdots\sigma_{\beta_m}^{\varepsilon_m}(\beta_m)\big)^{-1},$$
      where the summation runs over all   $(\varepsilon_1, \cdots,  \varepsilon_m)\!\in\!\! \{0, 1\}^m$ satisfying
       $\sigma_{\beta_1}^{\varepsilon_1}\cdots\sigma_{\beta_m}^{\varepsilon_m}=y.$

     We  define $c_{x, [y]}\in\mathbb{Q}[ \alpha_1^{\pm1}, \cdots, \alpha_n^{\pm1}]$ as follows.
      $$c_{x, [y]}:=\sum\nolimits_{z\in y W}c_{x, z}\big|_{\alpha_0=-\theta}=\sum\nolimits_{z\in y W}c_{x, z}(-\theta, \alpha_1, \cdots, \alpha_n).$$

    Let
     $\gamma_k$ denote  the (positive real) root $\sigma_{\beta_1}\cdots \sigma_{\beta_{k-1}}(\beta_k)$.   We define the homogeneous polynomial
     $d_{y, x}= d_{y, x}(\alpha_0, \alpha_1, \cdots, \alpha_n)\in \mathbb{Q}[\alpha_0, \cdots, \alpha_n]$ as follows.\\
             If $y\not \preccurlyeq x$ then  $d_{y, x}= 0$; if $y=1$, then $d_{y, x}= 1$; if $y\preccurlyeq x$ and  $y\neq 1$, then
          $$d_{y, x}:=\sum
                     \gamma_{i_1}\cdots\gamma_{i_r},$$
    where the summation runs
     over all subsequences $(i_1, \cdots, i_r)$  of $(1, \cdots, m)$ such that
     {\upshape $ y=[\sigma_{\beta_{i_1}}\cdots\sigma_{\beta_{i_r}}]_{\scriptsize \mbox{red}}$}.

       We define $d_{y, [x]}\in \mathbb{Q}[\alpha_1,\cdots,\alpha_n]$ as follows.
             $$d_{y, [x]}:= d_{y, x}{\big|_{\alpha_0=-\theta}}=d_{y, x}(-\theta, \alpha_1, \cdots, \alpha_n).$$
  \end{defn}
 Note that for any {\upshape $y, y'\in W_{\scriptsize\mbox{af}}$} with $yW= y'W$, one has $c_{x, [y]}=c_{x, [y']}$.
  In addition, one  has  $d_{x, [y]}=d_{x, [y']}$ provided  $x\in W_{\scriptsize\mbox{af}}^-$  (following from Lemma \ref{coe99}).

 \begin{prop}[\cite{koku}; see also chapter 11 of \cite{kumar}]\label{basis} $c_{x, y}$ and $d_{y, x}$ are well-defined,
                              independent of the choices of   reduced decompositions of $x$.
    The transpose of $\big(c_{x, y}\big)$ is the inverse of the matrix $\big( d_{x, y}\big)$ in the following sense
                   {\upshape $$\sum\nolimits_{z\in W_{\scriptsize \mbox{af}}}c_{x, z}d_{y, z}=\delta_{x, y}=\sum\nolimits_{z\in W_{\scriptsize \mbox{af}}}c_{z, x}d_{z, y}, \quad \mbox{for any } x, y\in W_{\scriptsize \mbox{af}}.$$}
 \end{prop}
Note that both   summations in the above proposition contain only
finitely many nonzero terms.

  \subsection{Explicit formula for Pontryagin product on $H_*^T(\Omega K)$}
 Because of the homotopy-equivalence between
   $\mathcal{G}/\mathcal{P}_0$ and $\Omega K$, we interchange the notations  $\mathcal{G}/\mathcal{P}_0$ and $\Omega K$ freely.
  Let $\hat T_\mathbb{C}$ denote  the standard maximal torus of $\mathcal{G}$  with maximal compact
  sub-torus $\hat T$.
   The $\hat T$-equivariant cohomology $H^*_{\hat T}(\mathcal{G}/\mathcal{P}_0)$ is an $\hat S$-algebra  with a basis of Schubert classes
   $ \{\hat{\mathfrak{S}}^{x}~|~ x\in W^-_{\scriptsize\mbox{af}}\}$, where  $\hat S=H^*_{\hat T}(\mbox{pt})= \mathbb{Q}[\alpha_0, \alpha_1, \cdots, \alpha_n]$.
   Note that $T\subset \hat T$ is a sub-torus. The $T$-equivariant cohomology $H^*_{T}(\mathcal{G}/\mathcal{P}_0)$ is an $S$-algebra  with a basis of Schubert classes
   $ \{ {\mathfrak{S}}^{x}~|~ x\in W^-_{\scriptsize\mbox{af}}\}$, where  $S=H^*_{T}(\mbox{pt})= \mathbb{Q}[\alpha_1, \cdots, \alpha_n]$. Furthermore, one has
    the following     evaluation maps $\mbox{ev}: H^*_{\hat T}(\mathcal{G}/\mathcal{P}_0)\rightarrow H^*_T(\mathcal{G}/\mathcal{P}_0)$ and
           $ {ev}: \hat S \rightarrow  S$ such that  $\mbox{ev}(f\hat{\mathfrak{S}}^{x})= {ev}( f)\mathfrak{S}^{x}$, where
            $f=f(\alpha_0, \alpha_1,\cdots, \alpha_n)\in \hat S$ and  $ {ev}(f)=f(-\theta, \alpha_1, \cdots, \alpha_n)\in S$. (See appendix \ref{equi}
            for more details on the above descriptions.)

   The $T$-equivariant homology $H^T_*(\mathcal{G}/\mathcal{P}_0)$ is the submodule of $\mbox{Hom}_S(H^*_T(\mathcal{G}/\mathcal{P}_0), S)$ spanned by
      the equivariant Schubert homology classes  $\{\mathfrak{S}_{x}
    ~|~ x\in  W_{\scriptsize\mbox{af}}^- \}$\footnote{We should note that we are using the equivariant Borel-Moore homology (see  e.g.  \cite{grah}).},
       which for any $x, y\in W_{\scriptsize\mbox{af}}^-$ satisfy $\langle \mathfrak{S}_{x}, \mathfrak{S}^{y}\rangle= \delta_{x, y}$  with respect to the natural pairing.

    The adjoint action of $T$ on $K$ induces a canonical action on $\Omega K$ by pointwise conjugation; that is, $(t\cdot f)(s):= t\cdot f(s)\cdot t^{-1}$ for
    any $t\in T$ and $f\in \Omega K$.    The group multiplication $K\times K\to K$ induces a
     so-called Pontryagin  product $\Omega K\times \Omega K\rightarrow \Omega K$ by pointwise multiplication; that is, $(f\cdot g)(s)=f(s)\cdot g(s)$
         for any $f, g\in \Omega K$.
    The Pontryagin product
          is obviously associative and  $T$-equivariant. Therefore  it induces
            $H^T_*(\Omega K)\otimes H^T_*(\Omega K)\rightarrow H^T_*(\Omega K)$, which is also called the  Pontryagin product.
            As a consequence, $H^T_*(\Omega K)$ is an associative $S$-algebra (see \cite{peterson}, \cite{lam}), and therefore
            the   structure  coefficients $b_{x, y}^{\,z}$ for the Pontryagin product
                $$\mathfrak{S}_x\mathfrak{S}_y=\sum\nolimits_{z\in W_{\scriptsize\mbox{af}}^-}b_{x, y}^{\,z}\mathfrak{S}_z$$  for  $x,y\in W_{\scriptsize\mbox{af}}^-$
         are   polynomials in $S$. Now we  state the
        main result of this section as follows.
 \begin{thm}\label{str}  For any   {\upshape $x, y, z\in W_{\scriptsize\mbox{af}}^-$},  the structure coefficient  $b_{x, y}^{\,z}$  for $H^T_*(\Omega K)$ is
      given by 
             {\upshape $$b_{x,y}^{\,z}=\sum_{\lambda, \mu\in Q^\vee}c_{x, [t_\lambda]}c_{y, [t_\mu]}d_{z, [t_{\lambda+\mu}]}.$$}
  \end{thm}
   In particular, $b_{x,y}^{\,z}=b_{y, x}^{\, z}$ for all $z$, which implies $\mathfrak{S}_x\mathfrak{S}_y=\mathfrak{S}_y\mathfrak{S}_x$.
  Furthermore, $c_{x, [t_\lambda]}\neq 0$ only if there exists  $z\in t_\lambda W$ such that $c_{x, z}\neq 0$, which holds only if $z \preccurlyeq x$.
  In particular, there are only finitely many nonzero terms of $c_{x, [t_\lambda]}$'s  and $c_{y, [t_\mu]}$'s once $x$ and $y$ are fixed. Hence, the above summation
   for $b_{x, y}^{\, z}$ does make sense.

  Note that $\mathcal{P}_0=\mathcal{P}_\Delta$ and $ W_{\scriptsize\mbox{af}}^-=W_{\scriptsize\mbox{af}}^\Delta$. By replacing $\Delta$ with a general subset $Y\subset \Delta$,
     $H_*^T(\mathcal{G}/\mathcal{P}_Y)$ and $\mathfrak{S}^x_Y$ can be defined in  a similar manner.
     To distinguish these with the case of our main interest $\mathcal{G}/\mathcal{P}_0$,
     we denote $\mathfrak{S}^x_0, \mathfrak{S}_x^0$ for the case $Y=\emptyset$ (note that $\mathcal{P}_\emptyset=\mathcal{B}$).
      These notions can be extended to $\hat T$-equivariant (co)homology for $\mathcal{G}/\mathcal{P}_Y$ for a larger $\hat T$-action. The corresponding
      Schubert classes are denoted by  ``$\hat {\mathfrak{S}}$" instead of       ``$\mathfrak{S}$".

    \begin{defn} Given
                 {\upshape $x\in W_{\scriptsize\mbox{af}}$}, we define the
                   element $\psi_x^Y$ in {\upshape$\mbox{Hom}_S(H^*_T(\mathcal{G}/\mathcal{P}_Y), S)$} to be
                   the canonical morphism  {\upshape$\psi_x^Y:=(\iota_x^Y)^*: H^*_T(\mathcal{G}/\mathcal{P}_Y)\rightarrow H^*_T(\mbox{pt})= S$,}
                     where  $\iota_x^Y$ is the $T$-equivariant map  {\upshape $\iota_x^Y:\mbox{pt}\rightarrow \mathcal{G}/\mathcal{P}_Y$} given by
                      {\upshape $\mbox{pt}\mapsto x\mathcal{P}_Y$}.

              We define $\hat{\psi}_x^Y$ to be the element $\hat{\psi}_x^Y=(\iota_x^Y)^*$ in {\upshape $\mbox{Hom}_{\hat S}(H^*_{\hat T}(\mathcal{G}/\mathcal{P}_Y), \hat S)$,}
          by considering  the action by the larger group $\hat T$.

         Since we consider the cases $Y=\Delta$ and $Y=\emptyset$ only, we simply denote
            $$\psi_x=\psi^\Delta_x;\quad \hat\psi_x=\hat \psi_x^\Delta;\quad \psi_x^0=\psi^\emptyset_x;\quad \hat\psi_x^0=\hat \psi_x^\emptyset .$$
    \end{defn}
     The relation between $\hat{\mathfrak{S}}^0_{x}$ and
      $\hat\psi_y^0$  is expressed in the following lemma.
       \begin{lemma}[Proposition 3.3.1 of \cite{ara}\footnote{The terminologies  used in \cite{ara} and the present paper can be identified as
         follows:$\mathcal{L}_x= \hat{\mathfrak{S}}^0_{x}$
          and
      $\Theta(\mu)(y)=\hat\psi_y^0(\mu)$ for $\mu\in H^*_{\hat T}(\mathcal{G}/\mathcal{B})$.}]\label{ara22}\label{basis2}

 For any     {
      \upshape  $ {x \in W_{\scriptsize\mbox{af}}},\,\,\,\,  \hat{\mathfrak{S}}^0_{x}=\sum_{y\in W_{\scriptsize\mbox{af}}}c_{x, y}\hat\psi_y^0$.
         }  
    \end{lemma}

     }
 \bigskip

     Note that one has $ \psi_x=\psi_y$ whenever $xW=yW$ (following the definition), and  that
    the canonical map $\pi_*:  H^T_*(\mathcal{G}/\mathcal{B})\rightarrow  H^T_*(\mathcal{G}/\mathcal{P}_0)$, induced by
     the natural projection  $\pi:  \mathcal{G}/\mathcal{B}\rightarrow  \mathcal{G}/\mathcal{P}_0$, is given by (see e.g. Lemma 11.3.3 of \cite{kumar})
           $$\pi_*(\mathfrak{S}^0_{x})=\begin{cases}0,&\mbox{if } x\not\in W_{\scriptsize\mbox{af}}^-\\ \mathfrak{S}_{x},&\mbox{if }x\in W_{\scriptsize\mbox{af}}^-\end{cases}.$$

   \begin{prop}\label{basis3}
       \begin{enumerate}
          \item[(i)] For any {\upshape $x\in W_{\scriptsize\mbox{af}}$}, {\upshape $\psi_x^0=\sum_{y\in W_{\scriptsize\mbox{af}}}d_{y, [x]} {\mathfrak{S}}^0_{y}$}  in $H^T_*(\mathcal{G}/\mathcal{B}).$
          \item[(ii)] For any {\upshape $x\in W_{\scriptsize\mbox{af}}^-$}, {\upshape $\psi_x=\sum_{y\in W_{\scriptsize\mbox{af}}^-}d_{y, [x]} {\mathfrak{S}}_{y}$} in  $ H^T_*(\mathcal{G}/\mathcal{P}_0)$.
          \item[(iii)] For any {\upshape $x\in W_{\scriptsize\mbox{af}}^-$,  ${\mathfrak{S}}_{x}=
                                \sum_{t\in Q^\vee} c_{x, [t]} \psi_{t}$}.
       \end{enumerate}
    \end{prop}
 \begin{proof}
   \begin{enumerate}
      \item[(i)] It follows from Proposition \ref{basis} and  Lemma \ref{basis2} that
                          $$\hat\psi_x^0=\sum_{z\in W_{\scriptsize\mbox{af}}}\delta_{x, z}\hat\psi_z^0
                          =\sum_{z\in W_{\scriptsize\mbox{af}}}\sum_{y\in W_{\scriptsize\mbox{af}}} d_{y, x}c_{y, z}\hat\psi_z^0
                           =\sum_{y\in W_{\scriptsize\mbox{af}}}d_{y, x} \hat{\mathfrak{S}}^0_{y}.$$
               Note that $\iota_x^\emptyset$ is both $T$-equivariant and $\hat T$-equivariant  and   $\mbox{id}_{\mbox{pt}}$, $\mbox{id}_{\mathcal{G}/\mathcal{B}}$
               are morphisms preserving the $T$-action. Hence,
                the identity  $\mbox{id}_{\mathcal{G}/\mathcal{B}}\circ \iota_x^\emptyset=\iota_x^\emptyset\circ \mbox{id}_{\mbox{pt}}$
              induces $\psi_x^0\circ \mbox{ev}= {ev}\circ \hat\psi_x^0: H^*_{\hat T}(\mathcal{G}/\mathcal{B})\rightarrow S$.
                  Therefore,
                  $$\psi_x^0(\mathfrak{S}^{y}_0)=\psi_x^0\circ \mbox{ev}(\hat{\mathfrak{S}}^{y}_0)
                            = {ev}\circ\hat\psi_x^0(\hat{\mathfrak{S}}^{y}_0)= {ev}(d_{y, x})=d_{y, [x]}.$$
                   Hence,  $ \psi_x^0=\sum_{y\in W_{\scriptsize\mbox{af}}} d_{y, [x]}\mathfrak{S}_{y}^0$,  where    the summation  contains only finitely
                   many nonzero terms since $d_{y, [x]}=0$ whenever $y\not\preccurlyeq x$. Thus
            $\psi_x^0\in   H_*^T(\mathcal{G}/\mathcal{B})$.

      \item[(ii)]   Note that $\iota_x^\Delta$, $\iota_x^\emptyset$ and $\pi$ are all $T$-equivariant maps.
                    Hence, the identity $\iota_x^\Delta=\pi\circ\iota_x^\emptyset$
                   induces    $\psi_x=\psi_x^0\circ\pi^*: H^*_{ T}(\mathcal{G}/\mathcal{P}_0)\rightarrow S$.
                  Therefore, $$\psi_x=\pi_*(\psi_x^0)=\pi_*(\sum\nolimits_{y\in W_{\scriptsize\mbox{af}}} d_{y, [x]}\mathfrak{S}_{y}^0)=
                                 \sum\nolimits_{y\in W_{\scriptsize\mbox{af}}^-} d_{y, [x]}\mathfrak{S}_{y} \in   H_*^T(\mathcal{G}/\mathcal{P}_0)
                                  ,$$
      noting that   the summation  contains only finitely many nonzero terms.
  \item[(iii)]
       Denote $c_{x, y}'=c_{x, y}\big|_{\alpha_0=-\theta}$.   It follows from (ii) and Proposition \ref{basis} that
                  $${\mathfrak{S}}_{x}=
                            \sum_{y\in W_{\scriptsize\mbox{af}}}\delta_{x, y}{\mathfrak{S}}_{y}
                          =\sum_{y\in W_{\scriptsize\mbox{af}}}\sum_{z\in W_{\scriptsize\mbox{af}}} c_{x, z}' d_{y, [z]}{\mathfrak{S}}_{y}
                           =\sum_{z\in W_{\scriptsize\mbox{af}}}c_{x, z}'\psi_z.$$
                   Note that $\psi_z=\psi_y$ whenever $z\in yW$, and
                    each coset $yW$ has a unique representative of translation in $Q^\vee\cong W_{\scriptsize\mbox{af}}/W$.
                     Hence for any $x\in W_{\scriptsize\mbox{af}}^-$,
                     $$ \mathfrak{S}_{x}
                           =\sum_{z\in W_{\scriptsize\mbox{af}}}c_{x, z}'  \psi_z
                           =\sum_{t\in Q^\vee}\big(\sum_{z\in tW}c_{x, z}'\big)\psi_t=\sum_{t\in Q^\vee}c_{x, [t]}\psi_t.$$
   \end{enumerate}
\vspace{-0.2cm}
 \end{proof}

  The above proposition was essentially contained in Peterson's notes \cite{peterson}.

 The Pontryagin product gives an associative $S$-algebra structure on  the equivariant homology $H^T_*(\Omega K)$.
     The
     following proposition  was
                 stated in \cite{peterson}  by Peterson. We learned  the following  proof from Thomas Lam.
\begin{prop}\label{mul}
         For any $\lambda, \mu\in Q^\vee$,\, the Pontryagin product of $\psi_{t_\lambda}$ and $\psi_{t_\mu}$ in $H^T_*(\Omega K)$ is given by
                $\psi_{t_\lambda}\psi_{t_\mu}=\psi_{t_\lambda t_\mu}=\psi_{t_{\lambda+\mu}}.$
       \end{prop}

\begin{proof}

\bigskip
 Identify $\lambda\in Q^\vee$ with the co-character $\lambda: \mathbb{S}^1 \to T$,  which gives a point $\lambda:\mathbb{S}^1\to K$ in $\Omega K$.
    These are the $T$-fixed points of $\Omega K$.
  Note that $\psi_t$ is the map $H^*_T(\Omega K)\rightarrow H^*_T(\mbox{pt})$ induced by the map
   $\mbox{pt}\to \Omega K$ with image $t$. Thus
     $\psi_{t_\lambda} \psi_{t_\mu}$ is the map $H^*_T(\Omega K) \longrightarrow H^*_T(\mbox{pt})$ induced by the following composition of maps:
        $$\mbox{pt} \longrightarrow \Omega K \times \Omega K \longrightarrow \Omega K,\quad \mbox{which is given by }
         \mbox{pt} \mapsto (t_{\lambda}, t_\mu) \mapsto t_{\lambda}t_\mu.$$

 Note that pointwise multiplication on the group takes the loops $\lambda: \mathbb{S}^1 \to T$, $\mu: \mathbb{S}^1  \to T$ to the loop $(\lambda+\mu):\mathbb{S}^1 \to T.$
 Thus  $\psi_{t_\lambda} \psi_{t_\mu} = \psi_{t_\lambda t_\mu}=\psi_{t_{\lambda+\mu}}$.

\end{proof}

    Now we can derive the proof of Theorem \ref{str} easily.
    \begin{proof}[Proof of Theorem \ref{str}]

  It follows from Proposition \ref{basis3} and Proposition \ref{mul} that
              \begin{align*}
               \mathfrak{S}_x\mathfrak{S}_y&=\big(\sum_{\lambda\in Q^\vee} c_{x, [t_{\lambda}]} \psi_{t_\lambda}\big)\big(\sum_{\mu\in Q^\vee} c_{y, [t_{\mu}]} \psi_{t_\mu}\big)\\
                         &= \sum\nolimits_{\lambda, \mu\in Q^\vee} c_{x, [t_{\lambda}]}  c_{y, [t_{\mu}]} \psi_{t_\lambda}\psi_{t_\mu}\\
                         &= \sum\nolimits_{\lambda, \mu\in Q^\vee} c_{x, [t_{\lambda}]}  c_{y, [t_{\mu}]} \psi_{t_{\lambda+\mu}}\\
                         &= \sum_{\lambda, \mu\in Q^\vee, z\in W_{\scriptsize\mbox{af}}^-} c_{x, [t_{\lambda}]}  c_{y, [t_{\mu}]} d_{z, [t_{\lambda+\mu}]}\mathfrak{S}_z.
              \end{align*}
           Hence, $b_{x,y}^{\,z}=\sum_{\lambda, \mu\in Q^\vee}c_{x, [t_\lambda]}c_{y, [t_\mu]}d_{z, [t_{\lambda+\mu}]}$.
  \end{proof}

 \begin{remark}
   The equivariant Schubert structure constants
    for the equivariant cohomology of based loop group $\Omega K$ can also be expressed in terms of $c_{x, [y]}$ and $d_{x, [y]}$. (See appendix \ref{equi} for more details.)

 \end{remark}

    Note that $c_{x, [t_{\lambda}]},  c_{y, [t_{\mu}]}$ and $ d_{z, [t_{\lambda+\mu}]} $ are homogeneous rational functions of degree $-\ell(x), -\ell(y)$ and $\ell(z)$
    respectively. Since $b_{x, y}^{\, z}$ is a polynomial, we obtain the following corollary.
    \begin{cor}\label{coe}
            Let {\upshape $x, y, z\in W_{\scriptsize\mbox{af}}^-$}, 
               one has $b_{x, y}^{\, z}=0$ unless $\ell(z)\geq\ell(x)+\ell(y)$. Furthermore
             if $\ell(z)=\ell(x)+\ell(y)$, then  the rational function  $b_{x, y}^{\, z}$ is a constant.
    \end{cor}

\section{Explicit formula for (equivariant) quantum cohomology of $G/B$}\label{secequivquan}
 In this section, we give an alternative proof of Peterson's Theorem proved by Lam and Shimozono, and
    derive the following combinatorial formula for quantum Schubert structure constants $N_{u, v}^{w, \lambda}$'s,
                   which covers  Main Theorem as stated in the introduction.
\begin{thm}\label{th22}
  For any $u, v, w\in W$, $\lambda\in Q^\vee$  with $\lambda\succcurlyeq 0 $, the quantum Schubert structure constant   $N_{u, v}^{w, \lambda}$ for
               $G/B$ is given as follows.

                Denote $A=-12n(n+1)\sum_{i\in I}w_i^\vee$ (which is in fact a regular and anti-dominant element in $Q^\vee$).
               Let  $x=ut_A, y=vt_A$ and $z=wt_{2A+\lambda}$.
       \begin{enumerate}
         \item If $\langle \lambda, 2\rho\rangle \neq \ell(u)+\ell(v)-\ell(w)$, then  $N_{u, v}^{w, \lambda}=0$.

         \item If $\langle \lambda, 2\rho\rangle = \ell(u)+\ell(v)-\ell(w)$, then
              \begin{enumerate}

              \item  The   rational function
                                 $ \sum_{\lambda_1, \lambda_2\in  Q^\vee}
                                c_{x, [ t_{\lambda_1}]}c_{y, [ t_{\lambda_2}]}d_{z, [t_{\lambda_1+\lambda_2}]} 
                                 $,
                            which belongs to $\mathbb{Q}[\alpha_1^{\pm1}, \cdots, \alpha_n^{\pm1}]$,  
                                 is in fact
                                  a constant.
              \item The quantum Schubert structure constant $  N_{u, v}^{w, \lambda}$ is given by
                                       $$  N_{u, v}^{w, \lambda}=\sum_{\lambda_1, \lambda_2\in  Q^\vee}
                                c_{x, [ t_{\lambda_1}]}c_{y, [ t_{\lambda_2}]}d_{z, [t_{\lambda_1+\lambda_2}]} 
                                 .$$
                           \end{enumerate}
                Furthermore, one has the following (by simplifying the summation)
                    $$ N_{u, v}^{w, \lambda}=
                     \sum_{(\lambda_1, \lambda_2, v_1)\in \Gamma\times W}
                       c_{x, [ v_1t_{\lambda_1}]}c_{y, [ v_1t_{\lambda_2}]}d_{z, [v_1t_{\lambda_1+\lambda_2}]}
                            ,$$
                   where  $\Gamma=\{(\lambda_1, \lambda_2)~|~ \lambda_1, \lambda_2\succcurlyeq A, \lambda_1+\lambda_2\preccurlyeq 2A+\lambda, \lambda_1, \lambda_2\in\tilde Q^\vee\}$.

       \end{enumerate}
\end{thm}

The above theorem is in fact a direct consequence of Peterson's Theorem and Theorem \ref{str}, except for the last statement of it.  As is shown in the above theorem,   our formula is a combinatorial formula, other than a recursive algorithm.
   Although rational functions in equivariant parameters are involved in our formula, we can take their valuations at a special point, so that
  the formula can be easily programmed. 

 In section 4.1, we  calculate certain  structure coefficients $b_{x, y}^{\, z}$'s for the Pontryagin product on $H_*^T(\Omega K)$.
 These calculations allow us to  re-establish the equivalence   between  the torus-equivariant quantum cohomology of $G/B$ and  the torus-equivariant homology of $\Omega K$
 after localization, which is explained in section 4.2.
  Finally in section 4.3, we prove Theorem \ref{th22}. 

 \subsection{Calculations for   structure coefficients $b_{x, y}^{\, z}$'s}\label{seccalcustr}

  In this subsection, we analyze the summation in the formula for the structure coefficients $b_{x, y}^{\, z}$'s.
  We obtained several useful formulas, including   the following two as the main results of this subsection. The proofs are elementary and  combinatorial in nature.
  These two formulas have been stated in \cite{peterson} and proved  by
   Lam-Shimozono  \cite{lamshi} by a different method using nil-Hecke ring and  Peterson $j$-isomorphism.

\begin{prop}
 \label{prop22} For any {\upshape $ wt_\lambda, t_\mu\in W_{\scriptsize\mbox{af}}^-$},\,one has
            $\mathfrak{S}_{wt_\lambda}\mathfrak{S}_{t_\mu}=  \mathfrak{S}_{wt_{\lambda+\mu}}.$
\end{prop}

\begin{prop}\label{prop33} Let {\upshape $\sigma_it_\lambda,  ut_\mu\in W_{\scriptsize\mbox{af}}^-$}, where
          $\sigma_i=\sigma_{\alpha_i}$ with $i\in I$. Then one has
        $$\mathfrak{S}_{\sigma_it_\lambda}\mathfrak{S}_{ut_\mu}= (u(w_i)-w_i)\mathfrak{S}_{ut_{\lambda+\mu}}
                      \,+\, \sum_{\gamma\in\Gamma_1}\langle\gamma^\vee, w_i\rangle\mathfrak{S}_{u\sigma_\gamma t_{\lambda+\mu}}
                      \,+\, \sum_{\gamma\in\Gamma_2}\langle\gamma^\vee, w_i\rangle\mathfrak{S}_{u\sigma_\gamma t_{\lambda+\mu+\gamma^\vee}}\footnote{
                     The coefficient    $\langle\gamma^\vee, w_i\rangle$ is always equal to zero if either
                           $\gamma\in \Gamma_1$ and $u\sigma_\gamma t_{\lambda+\mu}\notin W_{\scriptsize\mbox{af}}^- $ or
                           $\gamma\in\Gamma_2$ and $u\sigma_\gamma t_{\lambda+\mu+\gamma^\vee}\notin W_{\scriptsize\mbox{af}}^- $.},$$
     where $\Gamma_1$ and $\Gamma_2$ are as defined in section 2.1.
\end{prop}

Since the proofs of all the lemmas in this subsection only involves simple arguments, we leave them in the appendix (section \ref{appelemmaproof}).

 Each element $x\in W_{\scriptsize\mbox{af}}=W\ltimes Q^\vee$ can be written as $x=wt_\lambda$ for  unique $w\in W$ and $t_\lambda\in Q^\vee$.
  Recall that for any $\mu\in Q^\vee$, $\mu$ is called regular  if and only if $\langle\lambda, \alpha_i\rangle\neq 0$ for all $i\in I$; $\mu$ is called
   anti-dominant if and only if  $\langle\mu, \alpha_i\rangle\leq 0$ for all $i\in  I$. We denote by
     $\tilde Q^\vee$  the set of anti-dominant elements in $Q^\vee$.
 The minimal length representatives in $W_{\scriptsize\mbox{af}}^-$ are characterized as follows.

 \begin{lemma}[see  e.g.  \cite{lamshi}]\label{len22} Let $\lambda\in Q^\vee$ and $w\in W$.
  Then {\upshape $wt_\lambda\in W_{\scriptsize\mbox{af}}^- $}   if and only if {$\lambda\in\tilde Q^\vee$} and if
        $\langle\lambda, \alpha_i\rangle=0$ then $w(\alpha_i)\succcurlyeq0$.
    When this happens,   $\ell(wt_\lambda)=\ell(t_\lambda)-\ell(w)$. 
   Furthermore,  $\ell(t_{\mu})=\langle w(\mu), -2\rho\rangle$ whenever    $\mu\in Q^\vee$ satisfies   $w(\mu)\in\tilde Q^\vee$.
 \end{lemma}

  Recall that for any    $x, y\in W_{\scriptsize\mbox{af}}$,
         $y\preccurlyeq x$ with respect to the  Bruhat order  $(W_{\scriptsize\mbox{af}}, \preccurlyeq)$
   if and only if  $y$ has an induced expression from a given reduced decomposition of $x$. That is,
     if    {\upshape $x=[\sigma_{\beta_1}\cdots\sigma_{\beta_r}]_{\scriptsize\mbox{red}}$}, then
              {\upshape $y= \sigma_{\beta_{k_1}}\cdots\sigma_{\beta_{k_s}}$} for some subsequence  $(k_1,\cdots, k_s)$.
     As shown in the next proposition, the Bruhat order among certain elements becomes quite simple. This result, which
      was   mentioned explicitly by Lusztig in section 2 of \cite{lus},    is proved by Stembridge as a special case of
 Theorem 4.10 of \cite{stem}.
\begin{prop}\label{stem1} For any $\lambda,\mu\in\tilde Q^\vee$,
           $t_\mu\preccurlyeq t_\lambda$ if and only if $\lambda\preccurlyeq\mu$; that is, $ \mu-\lambda=\sum\nolimits_{i\in I}a_i\alpha_i^\vee$\,\,
           $\mbox{with } a_i\geq 0                \mbox{ for each } i\in I.$
\end{prop}

\begin{cor}\label{sum1} Let {\upshape $t_\lambda,wt_\mu\in W_{\scriptsize\mbox{af}}^-$}. Then
            {\upshape $ wt_\mu\preccurlyeq t_\lambda$} if and only if {$ \lambda\preccurlyeq\mu.$}
\end{cor}
\begin{proof}
  We use induction on $\ell(w)$ and leave the details in section \ref{appelemmaproof}.
\end{proof}

\begin{lemma}\label{sum3}
Suppose $\lambda\in Q^\vee$ satisfies $\langle \lambda, \alpha_i\rangle\leq -2$ for each $i\in I$.
Then {\upshape $t_\lambda\in W_{\scriptsize\mbox{af}}^-$} admits
  a reduced decomposition of the form $t_\lambda=\omega_0\sigma_0u_1\sigma_0\cdots u_r\sigma_0$, where $u_j\in W \mbox{ for all } j.$
  That is, we can write $t_\lambda=[\sigma_{\beta_1}\cdots \sigma_{\beta_m}]_{\scriptsize\mbox{red}}$ such that $\beta_{\ell+1}=\beta_m=\alpha_0$
       and $\omega_0=[\sigma_{\beta_1}\cdots \sigma_{\beta_\ell}]_{\scriptsize\mbox{red}}$.
 Furthermore for any  $w\in W$, there is  a subsequence $1\leq i_1<\cdots<i_k\leq m$ such that
         $wt_\lambda=\sigma_{\beta_{i_1}}\cdots \sigma_{\beta_{i_k}}$, and
            any such a subsequence must satisfy $(i_{a+1}, \cdots, i_k)=(\ell+1,  \cdots, m)$ for some $0\leq a<k$.
\end{lemma}

Denote $[z]$ the coset $zW$ for  $z\in W_{\scriptsize\mbox{af}}$. Each coset $[z]$ contains a unique element   $m_{[z]}\in W_{\scriptsize\mbox{af}}^-$ of minimum length
   and a unique element of translation  $t_{[z]}\in Q^\vee$. Note that if
  $m_{[z]}=v_1t_{\lambda_1}$, then $t_{[z]}=t_{v_1(\lambda_1)}=v_1t_{\lambda_1}v_1^{-1}$.

\begin{defn}
   The length of a coset $[z]$ is defined to be  $\ell([z])=\ell(m_{[z]})$.
   Let {\upshape $x\in W_{\scriptsize\mbox{af}}^-$}.   We define {\upshape (i)}  $x\preccurlyeq [z]$ if $x\preccurlyeq m_{[z]}$
     and {\upshape (ii)} $[z]\preccurlyeq x$ if  $m_{[z]}\preccurlyeq x$.
\end{defn}

\begin{lemma}\label{sum88}
       Let {\upshape $x=wt_\lambda\in W_{\scriptsize\mbox{af}}^-$} and {\upshape $z\in W_{\scriptsize\mbox{af}}$}.  If  $x\preccurlyeq z$, then  $x\preccurlyeq [z]$.
      If {\upshape $y=ut_\mu\in W_{\scriptsize\mbox{af}}^-$},
            then  $ c_{x, [y]}=c_{x, y}'$,  whenever  either $\lambda=\mu$ with $\lambda$ regular or  $\ell(x)=\ell(y)+1$ holds.
\end{lemma}

For any $x, y\in W_{\scriptsize\mbox{af}}$, we denote
     $c_{x, y}'=c_{x, y}\big|_{\alpha_0=-\theta}$. To prove Proposition \ref{prop22}, we analyze the effective  summation first. In this case, it
        turns out that the summation for the product contains at most one nonzero term as follows.
  \begin{prop}\label{coe33} Let {\upshape $ wt_\lambda, t_\mu\!\in\! W_{\scriptsize\mbox{af}}^-$}.
       If  $\langle \mu, \alpha_i\rangle< -\ell(\omega_0)$  for all $i\in  I $, then
       $$\mathfrak{S}_{wt_\lambda}\mathfrak{S}_{t_\mu}=  c_{wt_\lambda, wt_\lambda}' c_{t_\mu, wt_\mu}' d_{wt_{\lambda+\mu}, [wt_{\lambda+\mu}w^{-1}]}\mathfrak{S}_{wt_{\lambda+\mu}}.$$
    \end{prop}
  \begin{proof}
    By Theorem \ref{str}, $\mathfrak{S}_x\mathfrak{S}_y= \sum_z b_{x, y}^{\, z}\mathfrak{S}_z=\sum_z\big(\sum_{t_1, t_2}c_{x, [t_1]}c_{y, [t_2]}d_{z, [t_1t_2]}\big)\mathfrak{S}_z$.
   Note that    $c_{x, [t_1]}=0$ unless  $x\succcurlyeq [t_1]$; $c_{y, [t_2]}=0$ unless $y\succcurlyeq [t_2]$;
                         $d_{z, [t_1t_2]}=0$ unless $z\preccurlyeq t_1t_2$, which implies  $z\preccurlyeq [t_1t_2]$ by Lemma \ref{sum88}.
      Combining with  Corollary \ref{coe}, we conclude that
           the effective  summation for  $\mathfrak{S}_x\mathfrak{S}_y$ is only  over those $(z, t_1, t_2)$'s in $W_{\scriptsize\mbox{af}}^-\times Q^\vee\times Q^\vee$
       satisfying $\ell(z)\geq \ell(x)+\ell(y)$,
            $x\succcurlyeq [t_1]$,  $y\succcurlyeq [t_2]$ and $z\preccurlyeq t_1t_2$.
       Then we claim   $[t_1]=[wt_\lambda]$, $[t_2]=[wt_\mu]$ and $z=wt_{\lambda+\mu}$. Consequently, $[t_1t_2]=  [wt_{\lambda+\mu}w^{-1}]$. Thus the statement follows by
       using   Lemma \ref{sum88} again.

       It remains to show our claim. Indeed, we note  $\ell(z)\leq \ell([t_1t_2])$ and denote  $v_jt_{\lambda_j}=m_{[t_j]}$ for $j=1, 2$.
  Then  $v_2t_{\lambda_2}=m_{[t_2]}\preccurlyeq t_\mu$ and consequently   $\mu\preccurlyeq \lambda_2$ by Corollary \ref{sum1}.
   Therefore, $\lambda_2=\mu+\kappa$ with $\kappa=\sum_{i\in I}  c_i\alpha_i^\vee$, in which $c_i\geq 0$ for each $i\in I=\{1,\cdots, n\}$.
   If $\sum_{i\in I} 2c_i>\ell(\omega_0)$, then  $\ell(t_{\lambda_2})
           =\langle \mu+\kappa, -2\rho\rangle=
               \ell(t_\mu)-\sum_{i\in I} 2c_i<\ell(t_\mu)-\ell(\omega_0)$,
     which deduces a contradiction as follows.
                  \begin{align*} \ell(x)+\ell(y)\leq\ell([t_1t_2]) =\ell([v_1t_{\lambda_1}v_1^{-1}v_2t_{\lambda_2}])
                                                             &\leq \ell(v_1t_{\lambda_1}v_1^{-1}v_2t_{\lambda_2})\\
                                                             &\leq \ell(v_1t_{\lambda_1})+\ell(v_1^{-1}v_2)+\ell(t_{\lambda_2})\\
                                                             &\leq \ell(x)+\ell(\omega_0)+\ell(t_{\lambda_2})\\
                                                             &<\ell(x)+ \ell(t_{\mu})
                                                             =\ell(x)+ \ell(y).
             \end{align*}
   If $\sum_{i\in I} 2c_i\leq \ell(\omega_0)$, then  for each $j\in I$ one has
    $\langle \lambda_2, \alpha_j\rangle=\langle \mu+\sum_{i\in I}  c_i\alpha_i^\vee, \alpha_j\rangle < -\ell(\omega_0)+2c_j\leq 0$.
     Hence, $\lambda_2\in\tilde Q^\vee$ is regular. Therefore $v_1^{-1}v_2t_{\lambda_2}\in W_{\scriptsize\mbox{af}}^-$ and
         \begin{align*} \ell(x)+\ell(y)\leq \ell(z)\leq\ell([t_1t_2]) =\ell([v_1t_{\lambda_1}v_1^{-1}v_2t_{\lambda_2}])
                                                             &\leq \ell(v_1t_{\lambda_1}v_1^{-1}v_2t_{\lambda_2})\\
                                                             &\leq \ell(v_1t_{\lambda_1})+\ell(v_1^{-1}v_2t_{\lambda_2})\\
                                                             &\leq \ell(x)+\ell(t_{\lambda_2})-\ell(v_1^{-1}v_2)\\
                                                             &\leq \ell(x)+ \ell(t_{\mu})-0
                                                             =\ell(x)+ \ell(y).
         \end{align*}
    Hence,  all inequalities are indeed equalities. Thus
           $\ell(z)=\ell([t_1t_2])$,  $\ell(v_1^{-1}v_2)=0$, $v_1t_{\lambda_1}=x=wt_\lambda$ and  $\lambda_2=\mu.$ Hence,
             $z=m_{[t_1t_2]}=wt_{\lambda+\mu}$ and $v_1=v_2=w$.
  \end{proof}

  To prove Proposition \ref{prop22}, we need to compute the only coefficient as in the above proposition.
  The following lemma is useful for   calculations of  certain structure constants.
\begin{lemma}[Lemma 11.1.22 of \cite{kumar}]\label{coe55}
                   For  $v, u\in W$,   $\displaystyle {d_{v, u}\over  \prod_{\beta\in R^+}\beta}=        u\big(c_{v^{-1}\omega_0, u^{-1}\omega_0}\big).$
\end{lemma}
\begin{notation}\label{notation}    Suppose that  $\langle \lambda, \alpha_i\rangle\leq -2$  and $\langle \mu, \alpha_i\rangle\leq -2$  for all $i\in  I $.
  Let $m=\ell(t_\lambda)$ and $p=\ell(t_\mu)$. Because of  Lemma \ref{sum3}, we can take reduced expressions
   {\upshape $t_\lambda=[\sigma_{\beta_1} \cdots \sigma_{\beta_m}]_{\scriptsize\mbox{red}}$ and
     $t_\mu=[\sigma_{\beta_{m+1}}\cdots  \sigma_{\beta_{m+p}}]_{\scriptsize\mbox{red}}$ }  such that
      $\beta_{r+1}=\beta_{m+r+1}=\alpha_0$ and  {\upshape $\omega_0=[\sigma_{\beta_1}\cdots \sigma_{\beta_r}]_{\scriptsize\mbox{red}}
                  =[\sigma_{\beta_{m+1}}\cdots \sigma_{\beta_{m+r}}]_{\scriptsize\mbox{red}}$}, where $r=\ell(\omega_0) $.
        Note that  {\upshape $t_{\lambda+\mu}=[\sigma_{\beta_1}\sigma_{\beta_2}\cdots \sigma_{\beta_{m+p}}]_{\scriptsize\mbox{red}}$ }. Denote
       $H_i:=  \sigma_{\beta_1}\cdots \sigma_{\beta_{i-1}}(\beta_j)$ for $1\leq i\leq m+p$,  and  denote
       $\tilde H_j:=\sigma_{\beta_{m+1}}\cdots \sigma_{\beta_{m+j-1}}(\beta_{m+j})$ for $1\leq j\leq p$. Clearly, $H_{m+j}=t_\lambda(\tilde H_j)$.
\end{notation}

\begin{prop}\label{coe77}
  Suppose   $\langle \lambda, \alpha_i\rangle\leq -2$ for all $i\in  I $.  Using Notation \ref{notation}, we have
                      $\displaystyle  c_{vt_\lambda, ut_\lambda}={u(d_{v^{-1}, u^{-1}})\over \prod_{j=1}^mu(H_j)}$   for any $v, u\in W$.

\end{prop}

 \begin{proof}
 Due to Lemma \ref{sum3}, we can write  $vt_\lambda=[\sigma_{\beta_{i_1}}\cdots \sigma_{\beta_{i_k}}\sigma_{\beta_{r+1}}\cdots \sigma_{\beta_m}]_{\scriptsize\mbox{red}}$
                with $(i_1, \cdots, i_k)$  a subsequence of $(1,\cdots, r)$ and $v\omega_0=[\sigma_{\beta_{i_1}}\cdots \sigma_{\beta_{i_k}}]_{\scriptsize\mbox{red}}$.
        Furthermore if
          $ut_\lambda=\sigma_{\beta_{i_1}}^{\varepsilon_{i_1}}\cdots \sigma_{\beta_{i_k}}^{\varepsilon_{i_k}}\sigma_{\beta_{r+1}}^{\varepsilon_{r+1}}\cdots \sigma_{\beta_m}^{\varepsilon_{m}}$ for
                           $(\varepsilon_{i_1}, \cdots, \varepsilon_{i_k},\varepsilon_{r+1},\cdots, \varepsilon_m)$ $\in\{0, 1\}^{m-r+k}$,
            then   $\varepsilon_{r+1}=\cdots=\varepsilon_{m}=1$ and $\sigma_{\beta_{i_1}}^{\varepsilon_{i_1}}\cdots \sigma_{\beta_{i_k}}^{\varepsilon_{i_k}}=u\omega_0$.
               As a consequence, we have
                      $c_{vt_\lambda, ut_\lambda}=c_{v\omega_0, u\omega_0}\cdot {1\over \prod_{j=r+1}^mu(H_j)}$ by   definition.
                      Note that  $\{H_1, \cdots, H_r\}$ are all the positive roots which are mapped to negative roots under $\omega^{-1}_0$ (see e.g. Lemma \ref{rword}).
        Thus         $\prod_{i=1}^r H_i= \prod_{\beta\in R^+}\beta$.
                      Hence, the statement follows by          Lemma \ref{coe55}.
        \end{proof}
  \begin{lemma}\label{coe99}
     For any {\upshape $x, y\in W_{\scriptsize\mbox{af}}^-$} and any $w\in W$, one has $d_{x, yw}=d_{x, y}$.
\end{lemma}
\bigskip

 \begin{proof}[Proof of Proposition \ref{prop22}] We first assume    $\langle\mu, \alpha_i\rangle<-\ell(\omega_0)$ for each $i\in I$.
   Note   $t_\lambda(\tilde H_j)\big|_{\alpha_0=-\theta}=\tilde H_j\big|_{\alpha_0=-\theta}$. By definition,
             $d_{wt_{\lambda+\mu},  wt_{\lambda+\mu}}=d_{wt_{\lambda}, wt_{\lambda}}  \prod_{j=1}^pwt_\lambda(\tilde H_j)$,
     $d_{1, w^{-1}}=1,$ and
       $c_{wt_\lambda, wt_\lambda}d_{wt_\lambda, wt_\lambda}=1$.
  \,
   Thus by Lemma \ref{coe99} and Proposition  \ref{coe77},
    \begin{align*} c_{wt_\lambda, wt_\lambda}' c_{t_\mu, wt_\mu}' d_{wt_{\lambda+\mu}, [wt_{\lambda+\mu}w^{-1}]}
                               =&  c_{wt_\lambda, wt_\lambda} \cdot {w(d_{1, w^{-1}})\over \prod_{j=1}^pw(\tilde H_j)}\cdot
                                           d_{wt_{\lambda  }, wt_{\lambda }}\prod_{j=1}^pwt_\lambda(\tilde H_j) \big|_{\alpha_0=-\theta}\\
                            = & { \prod_{j=1}^{p}wt_\lambda(\tilde H_j)\over  \prod_{j=1}^pw(\tilde H_j) } {\big|_{\scriptsize{\alpha_0=-\theta}}}=1.
                   \end{align*}
           Hence, we have  $\mathfrak{S}_{wt_\lambda}\mathfrak{S}_{t_\mu}=\mathfrak{S}_{wt_{\lambda+\mu}}$, by Proposition \ref{coe33}.

     In general, we take   $\kappa\in \tilde Q^\vee$ such that   $\langle \kappa, \alpha_i\rangle<- \ell(\omega_0)$ 
                     and    $\langle \kappa+\mu, \alpha_i\rangle<- \ell(\omega_0)$
               for each $i\in I$.  Denote $x=wt_\lambda, y=t_\mu$. Because of the associativity of the 
               product,
          $$\mathfrak{S}_{xt_{\mu+\kappa}}=\mathfrak{S}_x\mathfrak{S}_{t_{\mu+\kappa}}=\mathfrak{S}_x(\mathfrak{S}_{y}\mathfrak{S}_{t_\kappa})=
           (\mathfrak{S}_x\mathfrak{S}_{y})\mathfrak{S}_{t_\kappa}=\sum\nolimits_{z}b_{x, y}^{\, z}\mathfrak{S}_z\mathfrak{S}_{t_\kappa}=\sum\nolimits_{z}b_{x, y}^{\, z}\mathfrak{S}_{zt_\kappa}.$$
           Hence, $b_{x, y}^{\, z}=0$ if $zt_\kappa\neq xt_{\mu+\kappa}$;  $b_{x, y}^{\, z}=1$ if $zt_\kappa= xt_{\mu+\kappa}$, that is, $z=wt_{\lambda+\mu}$.
     \end{proof}

 With similar arguments, we obtain  the following two propositions.
 Recall that for $u\in W$,
  $$\Gamma_1=\{\gamma\in R^+ ~|~   \ell(u\sigma_\gamma)=\ell(u)+1\},
          \Gamma_2 =\{\gamma\in R^+ ~|~  \ell(u\sigma_\gamma)=\ell(u)+1-\langle\gamma^\vee, 2\rho\rangle\}.$$
 \begin{prop}\label{coe44}
  Let {\upshape $x, y\in W_{\scriptsize\mbox{af}}^-$} with $x=\sigma_it_\lambda$ and $y=ut_\mu$, where   $\sigma_i=\sigma_{\alpha_i}$
      for some $i\in I$.   Suppose   $\langle \lambda, \alpha_j\rangle< -\ell(\omega_0)$  and $\langle \mu, \alpha_j\rangle< -\ell(\omega_0)$  for all $j\in  I $,
      then
      $$\mathfrak{S}_x\mathfrak{S}_y=c_{x, ut_\lambda}' c_{y,y}' d_{ut_{\lambda+\mu}, [ut_{\lambda+\mu}u^{-1}]}\mathfrak{S}_{ut_{\lambda+\mu}}
                      \,+\! \sum_{\gamma\in\Gamma_1}A_\gamma\mathfrak{S}_{u\sigma_\gamma t_{\lambda+\mu}}
                      \,+\! \sum_{\gamma\in\Gamma_2}B_\gamma\mathfrak{S}_{u\sigma_\gamma t_{\lambda+\mu+\gamma^\vee}},  $$
 \begin{align*} \mbox{where}\,\,\, \,\,\,A_\gamma &=c_{x, ut_\lambda}' c_{y, y}' d_{u\sigma_\gamma t_{\lambda+\mu}, [ut_{\lambda+\mu}u^{-1}]}
                                \!\! + c_{x, u\sigma_\gamma t_\lambda}' c_{y, u\sigma_\gamma t_\mu}' d_{u\sigma_\gamma t_{\lambda+\mu},
                                                                   [u\sigma_\gamma t_{\lambda+\mu}\sigma_\gamma u^{-1}]},\\
                     B_\gamma &= c_{x, ut_\lambda}' c_{y, y}' d_{u\sigma_\gamma t_{\lambda+\mu+\gamma^\vee}, [ut_{\lambda+\mu}u^{-1}]}\,+\\
                          &\quad   \,   c_{x, u\sigma_\gamma t_\lambda}' c_{y, u\sigma_\gamma t_{\mu+\gamma^\vee}}' d_{u\sigma_\gamma t_{\lambda+\mu+\gamma^\vee},
                                                       [u\sigma_\gamma t_{\lambda+\mu+\gamma^\vee}\sigma_\gamma u^{-1}]}.
        \end{align*}
 \end{prop}
\begin{proof}
  See the appendix (section \ref{appeproposproof}).
\end{proof}
\begin{prop}\label{coe88}
   Suppose that  $\langle \lambda, \alpha_j\rangle < -\ell(\omega_0)$  and $\langle \mu, \alpha_j\rangle < -\ell(\omega_0)$  for all $j\in  I .$ Let $u, w\in W$ and $\sigma_i=\sigma_{\alpha_i}$ with $i\in I$.
           Then we have  $ d_{\sigma_i, u}=w_i-u(w_i)$. Using Notation \ref{notation}, we have
            $w(d_{w^{-1}, w^{-1}})\cdot d_{wt_{\lambda+\mu}, wt_{\lambda+\mu}}=  \prod_{j=1}^{m+p}w(H_j)$. Furthermore,
     \begin{align*}(1)&\quad d_{u\sigma_\gamma t_{\lambda+\mu}, ut_{\lambda+\mu}}={1\over u(\gamma)} \cdot d_{u t_{\lambda+\mu}, ut_{\lambda+\mu}},  \mbox{ for any }\gamma\in \Gamma_1.\\
                   (2)&\quad d_{u\sigma_\gamma t_{\lambda+\mu+\gamma^\vee}, ut_{\lambda+\mu}}={1\over u(\gamma+\delta)}\cdot d_{u t_{\lambda+\mu}, ut_{\lambda+\mu}}, \mbox{ for any }\gamma\in \Gamma_2.\\
                   (3)&\quad  c_{ut_{\mu}, u\sigma_\gamma t_{\mu+\gamma^\vee}}  d_{u\sigma_\gamma t_{\lambda+\mu+\gamma^\vee},
                                  u\sigma_\gamma t_{\lambda+\mu+\gamma^\vee}}=-{\prod_{j=1}^m u\sigma_\gamma t_{\mu+\gamma^\vee}(H_j)\over u(\gamma+\delta)}, \mbox{ for any }\gamma\in \Gamma_2 .
      \end{align*}
\end{prop}
\begin{proof}
  See the appendix (section \ref{appeproposproof}).
\end{proof}

\bigskip

\begin{proof}[Proof of Proposition \ref{prop33}]
 Denote $x=\sigma_it_\lambda$ and $y=ut_\mu$.
We first assume $\langle\lambda,\alpha_k\rangle<-\ell(\omega_0)$ and
 $\langle\mu,\alpha_k\rangle<-\ell(\omega_0)$ and for each $k\in I$.
    Note that  $t_\lambda(\tilde H_j)\big|_{\alpha_0=-\theta}=\tilde H_j\big|_{\alpha_0=-\theta}$. It follows from Lemma \ref{coe99},
     Proposition  \ref{coe77} and Proposition  \ref{coe88} that
   \begin{align*}
             c_{x, ut_\lambda}'  c_{y,y}'  d_{ut_{\lambda+\mu}, [ut_{\lambda+\mu}u^{-1}]}
          &= {u(d_{\sigma_i, u^{-1}})\over \prod_{j=1}^mu(H_j)}\cdot {u(d_{u^{-1}, u^{-1}})\over \prod_{j=1}^pu(\tilde H_j)}
                                       \cdot    d_{ut_{\lambda}, ut_{\lambda}}\prod_{j=1}^put_\lambda(\tilde H_j)\big|_{\alpha_0=-\theta}\\
          &= {u(w_i- u^{-1}(w_i))\over \prod_{j=1}^mu(H_j)}\cdot {\prod_{j=1}^mu( H_j) \over \prod_{j=1}^pu(\tilde H_j)}
                                       \cdot     \prod_{j=1}^pu(\tilde H_j)\big|_{\alpha_0=-\theta}\\
          &=u(w_i)-w_i.
   \end{align*}

   For $\gamma\in \Gamma_1$, one has
 \begin{align*}
   &\,\,\, c_{x, ut_\lambda}'  c_{y, y}'  d_{u\sigma_\gamma t_{\lambda+\mu}, [ut_{\lambda+\mu}u^{-1}]}
                                 + c_{x, u\sigma_\gamma t_\lambda}'  c_{y, u\sigma_\gamma t_\mu}'  d_{u\sigma_\gamma t_{\lambda+\mu},
                                                                   [u\sigma_\gamma t_{\lambda+\mu}\sigma_\gamma u^{-1}]}\\
  =&c_{x, ut_\lambda}  c_{y, y}  {d_{ut_{\lambda+\mu}, ut_{\lambda+\mu} }\over u(\gamma)}+
                              {u\sigma_\gamma(d_{\sigma_i, \sigma_\gamma u^{-1}})\over \prod_{j=1}^mu\sigma_\gamma(H_j)}
                           \cdot  {u\sigma_\gamma(d_{u^{-1}, \sigma_\gamma u^{-1}})\over \prod_{j=1}^pu\sigma_\gamma(\tilde H_j)}
                            d_{u\sigma_\gamma t_{\lambda+\mu}, u\sigma_\gamma t_{\lambda+\mu} }\big|_{\alpha_0=-\theta}\\
    =&  { u(w_i)-w_i\over u(\gamma)}+
                              {u\sigma_\gamma(w_i)-w_i\over \prod_{j=1}^{m+p}u\sigma_\gamma(H_j)}
                           \cdot  u\sigma_\gamma({d_{\sigma_\gamma u^{-1}, \sigma_\gamma u^{-1}}\over \  \gamma})
                            d_{u\sigma_\gamma t_{\lambda+\mu}, u\sigma_\gamma t_{\lambda+\mu} }\big|_{\alpha_0=-\theta}\\
     =&  { u(w_i)-w_i\over u(\gamma)}-
                              {u\sigma_\gamma(w_i)-w_i\over \prod_{j=1}^{m+p}u\sigma_\gamma(H_j)}
                           \cdot  {\prod_{j=1}^{m+p}u\sigma_\gamma(H_j)\over  u(\gamma)}\big|_{\alpha_0=-\theta}\\
       =&  { u(w_i -\sigma_\gamma(w_i))\over    u(\gamma)}=\langle\gamma^\vee, w_i\rangle.
 \end{align*}
 For $\gamma\in \Gamma_2$, one has
\begin{align*}
   &\,\,  c_{x, ut_\lambda}'  c_{y, y}'  d_{u\sigma_\gamma t_{\lambda+\mu+\gamma^\vee}, [ut_{\lambda+\mu}u^{-1}]}
                                 + c_{x, u\sigma_\gamma t_\lambda}'  c_{y, u\sigma_\gamma t_{\mu+\gamma^\vee}}'  d_{u\sigma_\gamma t_{\lambda+\mu+\gamma^\vee},
                                                                   [u\sigma_\gamma t_{\lambda+\mu+\gamma^\vee}\sigma_\gamma u^{-1}]}\\
  =&c_{x, ut_\lambda}  c_{y, y}  {d_{ut_{\lambda+\mu}, ut_{\lambda+\mu} }\over u(\gamma+\delta)}+
                              {u\sigma_\gamma(d_{\sigma_i, \sigma_\gamma u^{-1}})\over \prod_{j=1}^mu\sigma_\gamma(H_j)}
                           \cdot  {-\prod_{j=1}^m u\sigma_\gamma t_{\mu+\gamma^\vee}(H_j)\over u(\gamma+\delta)}\big|_{\alpha_0=-\theta}\\
        =&  { u(w_i)-w_i\over u(\gamma)}-
                              {u\sigma_\gamma(w_i)-w_i\over \prod_{j=1}^{m }u\sigma_\gamma(H_j)}
                           \cdot  {\prod_{j=1}^{m }u\sigma_\gamma(H_j)\over  u(\gamma)} \big|_{\alpha_0=-\theta}\\
         =&  { u(w_i -\sigma_\gamma(w_i))\over    u(\gamma)}=\langle\gamma^\vee, w_i\rangle.
 \end{align*}
 Hence, the statement holds.

 For general cases, the statement follows from Proposition \ref{prop22} and the associativity and the commutativity of the Pontryagin product.
\end{proof}

\subsection{Equivalence between $QH^*_T(G/B)$ and $H_*^T(\Omega K)$}
In his lecture notes \cite{peterson} D. Peterson stated
that there is an isomorphism between the torus-equivariant quantum cohomology of $G/B$ and  the torus-equivariant
 homology of $\Omega K$ after localization. However, the proofs in these notes
are incomplete, and in \cite{lamshi}, Lam and Shimozono proved this result, using some of
Peterson's original approach together with   Mihalcea's criterion.
     For completeness, we  describe the literature of this equivalence.

The $T$-equivariant quantum cohomology $QH^*_T(G/B)$ is a torus-equivariant extension of the quantum cohomology ring $QH^*(G/B)$. (See section \ref{appendequiquantum} for more
 details.)
 It is      a commutative ring and   has an $S[\mathbf{q}]$-basis of Schubert classes $\sigma^u$'s
     with $S[\mathbf{q}]=\mathbb{Q}[\alpha_1, \cdots, \alpha_n, q_1, \cdots, q_n]$.
      $$\sigma^u\star_T \sigma^v = \sum_{w\in W, \lambda\in Q^\vee}     \tilde N_{u,v}^{w, \lambda} q_{\lambda}\sigma^w,\quad \mbox{where }    \tilde N_{u,v}^{w, \lambda}= \tilde N_{u,v}^{w, \lambda}(\alpha)\in S=\mathbb{Q}[\alpha_1,\cdots,\alpha_n].$$
 When $\lambda=0$, $\tilde N_{u,v}^{w, \lambda}$ is equivalent to
   the corresponding equivariant Schubert structure constant. The evaluation $\tilde N_{u,v}^{w, \lambda}|_{\alpha_1=\cdots=\alpha_n=0}$ equals the quantum
  Schubert structure constant $N_{u,v}^{w, \lambda}$.
 A direct calculation of a general  $N_{u,v}^{w, \lambda}$ can be rather  difficult. However if  $v$ is a simple reflection, then
  the following  equivariant quantum  Chevalley formula holds, which was originally stated by Peterson in \cite{peterson} and has
  been proved by Mihalcea in \cite{mih}.
 \begin{prop}[Equivariant quantum Chevalley formula]\label{quch}
  Let $u\in W$ and $s_i=\sigma_{\alpha_i}$ with $i\in I$. Then  in $QH_T^*(G/B)$ one has
    $$\sigma^{s_i}\star_T \sigma^u=(w_i-u(w_i))\sigma^u+ \sum_{\gamma\in \Gamma_1}\langle  \gamma^\vee, w_i\rangle \sigma^{u\sigma_\gamma}+
             \sum_{\gamma\in \Gamma_2}\langle \gamma^\vee, w_i\rangle q_{\gamma^\vee}\sigma^{u\sigma_\gamma},
              $$
  where  $\Gamma_1=\{\gamma\in R^+ ~|~   \ell(u\sigma_\gamma)=\ell(u)+1\}$ and        $\Gamma_2=\{\gamma\in R^+ ~|~  \ell(u\sigma_\gamma)=\ell(u)+1-\langle\gamma^\vee, 2\rho\rangle\}$.

\end{prop}
By evaluating  at $w_i=0$, the quantum Chevalley formula (see \cite{fw}) is recovered.

Furthermore, the equivariant quantum Chevalley   formula completely determines the multiplication in $QH^*_T(G/B)$. That is   the following
          Mihalcea's criterion,  a special case $(P=B)$ of which is stated here only.

 \begin{prop}[Mihalcea's criterion; see Theorem 2 of \cite{mih}]\label{crit} $\mbox{}$

 \noindent Denote $\mathbb{Q}[\mathbf{\alpha}, \mathbf{q}]=\mathbb{Q}[\alpha_1,\cdots, \alpha_n, q_1,\cdots, q_n]$ and
          $\mathbb{Q}[\mathbf{\alpha}^{\pm1}, \mathbf{q}]=\mathbb{Q}[\alpha_1^{\pm1},  \cdots ,   \alpha_n^{\pm1},q_1,\cdots, q_n]$.
           Let $\mathcal{A}=\bigoplus_{u\in W}{\mathbb{Q}[\mathbf{\alpha}^{\pm1}, \mathbf{q}]}\sigma^u$ be any $\mathbb{Q}[\mathbf{\alpha}^{\pm1}, \mathbf{q}]$-algebra with the
            product written as $\sigma^u*\sigma^v=\sum_{w, \lambda}C_{u, v}^{w, \lambda}q_\lambda\sigma^w$   where $\lambda\succcurlyeq 0$.
        Suppose the structure coefficients $C_{u, v}^{w, \lambda}$ satisfy the following
    \begin{enumerate}
      \item (homogeneity) $C_{u, v}^{w, \lambda}\in \mathbb{Q}[\mathbf{\alpha}^{\pm1}]$ is a homogeneous rational function of degree
                                    $$\deg(C_{u, v}^{w, \lambda})=\ell(u)+\ell(v)-\ell(w)-\langle \lambda, 2\rho\rangle, \mbox{ whenever } C_{u, v}^{w, \lambda}\neq 0;$$
      \item (multiplication by unit)\quad  $C_{\scriptsize\mbox{id}, v}^{w, \lambda}=\begin{cases}1, &\mbox{ if } \lambda=0 \mbox{ and } w=v\\ 0, &\mbox{ otherwise}\end{cases}$;
      \item (commutativity)    \quad  $\sigma^u*\sigma^v=\sigma^v*\sigma^u$ for any $u, v\in W$;
      \item (associativity)  $(\sigma^{s_i} *\sigma^u )*\sigma^v=\sigma^{s_i}*(\sigma^u*\sigma^v)$ for any $u, v\in W$ and any simple reflection $s_i\in W$;

      \item (equivariant quantum Chevalley formula) For any  $u\in W$ and any simple reflection $s_i\in W$, the product of $\sigma^{s_i}*\sigma^u$ is given by
       $$\sigma^{s_i}* \sigma^u=(w_i-u(w_i))\sigma^u+ \sum_{\gamma\in \Gamma_1}\langle  \gamma^\vee, w_i\rangle \sigma^{u\sigma_\gamma}+
             \sum_{\gamma\in \Gamma_2}\langle \gamma^\vee, w_i\rangle q_{\gamma^\vee}\sigma^{u\sigma_\gamma}.
              $$

      \end{enumerate}
  Then for any $u, v, w, \lambda$, we have
                $$C_{u, v}^{w, \lambda}=\tilde N_{u, v}^{w, \lambda}.$$
    In particular,  $C_{u, v}^{w, \lambda}=0$ if $\deg(C_{u, v}^{w, \lambda})<0$,
          and $(\bigoplus\limits_{u\in W}{\mathbb{Q}[\mathbf{\alpha}, \mathbf{q}]}\sigma^u, *)$ is canonically isomorphic to
            $QH^*_T(G/B)$ as $\mathbb{Q}[\mathbf{\alpha}, \mathbf{q}]$-algebras.

 \end{prop}
 \begin{remark}
    As  shown in \cite{mi2}, the equivariant Schubert structure constants  are in fact   nonnegative combinations
         of monomials in the negative simple roots. Therefore, Mihalcea chose
                               the negative simple roots instead of the positive ones
                    for positivity reasons. For the same reason, we define the new algebra $(H^T_*(\Omega K), \bullet)$ below. As a consequence,
           the    canonical isomorphism   after localization  between $(H_*^T(\Omega K), \cdot)$ and $QH^*_T(G/B)$ 
           looks even more natural.
 \end{remark}

Define a new  product $\bullet$
           on $H^T_*(\Omega K)$ as follows.
    $$\mathfrak{S}_x\bullet\mathfrak{S}_y=\sum\nolimits_{z\in W_{\scriptsize\mbox{af}}^-}\tilde b_{x, y}^{\, z}\mathfrak{S}_z, \,\,\,\mbox{where }\,\,\, \tilde b_{x, y}^{\, z}=
                 (-1)^{\ell(z)-\ell(x)-\ell(y)}b_{x, y}^{\, z} .$$
    Note that $b_{x, y}^{\, z}$ is a homogeneous   polynomial of degree $\!\ell(z)-\!\ell(x)\!-\ell(y).$
   $\mbox{}$ Thus $(  H_*^T(\Omega K), \bullet)$ is canonically isomorphic to $(H_*^T(\Omega K), \cdot)$ as $S$-algebras. Immediately, it follows from
    the definition of $\bullet$ and Proposition \ref{prop22}
           that
          $\mathfrak{S}_x\bullet\mathfrak{S}_{t_\mu}=\mathfrak{S}_{xt_\mu}$ for any $x, t_\mu\in W_{\scriptsize\mbox{af}}^-$. As a consequence,
           $\{\mathfrak{S}_{t}~|~t\in \tilde Q^\vee\}$ is a multiplicatively  closed set without zero divisors.
          It is  easy to show that     the  following map  $\varphi$    is an $S$-module isomorphism.
            \begin{align*}\varphi:&\,\,\,  H^T_*(\Omega K)[\mathfrak{S}^{-1}_t~|~t\in\tilde Q^\vee]\longrightarrow QH^*_T(G/B)[\mathbf{q}^{-1}];\\
                                  &\qquad\qquad \mathfrak{S}_{wt_\lambda}\bullet\mathfrak{S}_{t_\mu}^{-1}\longmapsto q_{\lambda-\mu}\sigma^w,
            \end{align*}
     where  $QH^*_T(G/B)[\mathbf{q}^{-1}]= QH^*_T(G/B)[q_i^{-1}|i\in I]$.
      As a consequence, the algebra
         $H^T_*(\Omega K)[\mathfrak{S}^{-1}_t~|~t\in\tilde Q^\vee]$ has an $S$-basis
           $\{\varphi^{-1}(q_\lambda\sigma^w)~|~ \lambda\in Q^\vee, w\in W\}$. Therefore
                    for any $A, B\in   H^T_*(\Omega K)[\mathfrak{S}^{-1}_t~|~t\in\tilde Q^\vee], A\bullet B=\sum_{w, \lambda}C_{A, B}^{w, \lambda}\varphi^{-1}(q_\lambda\sigma^w)$.
      Thanks to Mihalcea's criterion, it becomes a routine to give a proof of Peterson's Theorem as below, as was done by Lam and Shimozono \cite{lamshi}.

  \begin{thm}\label{th11}
     \begin{enumerate}
        \item[\upshape (i)] 
        $\varphi:  H^T_*(\Omega K)[\mathfrak{S}^{-1}_t~|~t\in\tilde Q^\vee]\longrightarrow QH^*_T(G/B)[\mathbf{q}^{-1}]$ is an $S$-algebra isomorphism.
         \item[\upshape (ii)] Let $u, v, w\in W$ and $\lambda\in Q^\vee$. Take $\eta, \kappa, \mu\in \tilde Q^\vee$ such that
     $x=ut_\eta, y=vt_\kappa, z=wt_\mu$ lie in {\upshape
                $  W_{\scriptsize\mbox{af}}^- $}  and $\lambda=\mu-\eta-\kappa$.  Then we have
            $$\tilde N_{u, v}^{w, \lambda}=\tilde b_{x, y}^{\, z}. $$
     \end{enumerate}
  \end{thm}

\begin{remark}
   The assumption ``$\lambda\succcurlyeq 0$" in Mihalcea's criterion    becomes obvious by using Proposition \ref{lemm22}.
\end{remark}
 In section \ref{seccalcustr}, we have given  elementary proofs of    Proposition \ref{prop22} and Proposition \ref{prop33}.
 Since both of these two propositions  play  very important roles in the proof made by Lam and Shimozono
  besides Mihalcea's criterion, we give an alternative  proof of Peterson's Theorem  in this sense.

\subsection{Proof of Theorem \ref{th22}}\label{secproofofmain}
Denote by $v_it_{\lambda_i}=m_{[t_i]}\in W_{\scriptsize\mbox{af}}^-$ the
  minimal length representative  in the coset $[t_i]=t_iW$ as before, where $t_i\in Q^\vee, i=1, 2$. Note that for any $x\in W_{\scriptsize\mbox{af}}^-$ one has
  that  $c_{x, [t_i]}=c_{x, [v_it_{\lambda_i}]}$ by definition, and that $d_{x, t_1t_2}=d_{x, v_1t_{\lambda_1}v_1^{-1}v_2t_{\lambda_2}v_2^{-1}}
             = d_{x, v_2t_{v_2^{-1}v_1(\lambda_1)+\lambda_2}}$ following from Lemma \ref{coe99}. Therefore  for any {\upshape $x, y, z\in W_{\scriptsize\mbox{af}}^-$},

             $$b_{x,y}^{\,z}=\sum_{t_1, t_2\in Q^\vee}\!c_{x, [t_1]}c_{y, [t_2]}d_{z, [t_1t_2]}=
                            \sum_{v_1t_{\lambda_1}, v_2t_{\lambda_2}\in W_{\scriptsize\mbox{af}}^-}
                           \!\!\!\!\! \!c_{x, [v_1t_{\lambda_1}]}c_{y, [v_2t_{\lambda_2}]}d_{z, [v_2t_{v_2^{-1}v_1(\lambda_1)+\lambda_2}]} .$$

  The following lemma is contained in Lemma 13.2.A of \cite{hum}.
  \begin{lemma}\label{dom} For any $\lambda\in Q^\vee$, there exists a unique $\lambda'\in \tilde Q^\vee$ and some $w\in W$ such that $\lambda'=w(\lambda)$.
   Furthermore, $\lambda'\preccurlyeq \lambda$.
  \end{lemma}

 \begin{lemma}\label{lemm}
    Let  $\lambda_1, \lambda_2\in \tilde Q^\vee$ with $\langle \lambda_j, \alpha_i\rangle \leq  -2\ell(\omega_0)$ for all $j\in\{1, 2\}$ and   $i\in I$. Let
      $v_1, v_2\in W$. If $v_1\neq v_2$,  then
           $\ell(t_{v_2^{-1}v_1(\lambda_1)+\lambda_2})\leq\langle\lambda_1 +\lambda_2, -2\rho\rangle-4\ell(\omega_0)$.
 \end{lemma}
 \begin{proof}
   Take $w\in W$ such that $w(v_2^{-1}v_1(\lambda_1)+\lambda_2)\in\tilde Q^\vee$.
   \, By Lemma \ref{dom}, one has $wv_2^{-1}v_1(\lambda_1)=\lambda_1+\mu_1$ with $\mu_1\succcurlyeq 0$.
   If $w\neq 1$, then $w=[\sigma_{\beta_1}\cdots\sigma_{\beta_k}]_{\scriptsize\mbox{red}}$ with $k\geq 1$. Note that
     $w(\lambda_2)=\lambda_2-\sum_{j=1}^k\langle \lambda_2, \beta_j\rangle \gamma_j^\vee$, where $\gamma_j=\sigma_{\beta_1}\cdots\sigma_{\beta_{j-1}}(\beta_j)\in R^+$.
        Therefore, it follows from Lemma \ref{len22} that
       \begin{align*}\ell(t_{v_2^{-1}v_1(\lambda_1)+\lambda_2})
                        &=\langle \lambda_1+\mu_1+w(\lambda_2), -2\rho\rangle\\
                        &=\langle\lambda_1 +\lambda_2+\mu_1, -2\rho\rangle+\sum\nolimits_{j=1}^k\langle\lambda_2, \beta_j\rangle\langle\gamma_j^\vee, 2\rho\rangle\\
                        &= \langle\lambda_1 +\lambda_2, -2\rho\rangle-\langle\mu_1, 2\rho\rangle+\langle\lambda_2, \beta_1\rangle\cdot2+\sum_{j=2}^k\langle\lambda_2, \beta_j\rangle\langle\gamma_j^\vee, 2\rho\rangle\\
                        &\leq \langle\lambda_1 +\lambda_2, -2\rho\rangle-0 -2\ell(\omega_0)\cdot 2+0.
        \end{align*}
     If $w=1$, then $w(v_2^{-1}v_1(\lambda_1)+\lambda_2)=\lambda_2+w'(\lambda_1)$ with $w'=v^{-1}_2v_1\neq1$. With the same argument as above, one has
            $\ell(t_{v_2^{-1}v_1(\lambda_1)+\lambda_2})\leq \langle\lambda_1 +\lambda_2, -2\rho\rangle-4\ell(\omega_0).$
 \end{proof}

\begin{prop}\label{lemm22} Let {\upshape $x, y\in W_{\scriptsize\mbox{af}}^-$ }  with
     $x=ut_\eta, y=vt_\kappa$.   If $\langle \eta, \alpha_i\rangle \leq -5\ell(\omega_0)$ and
           $\langle \kappa, \alpha_i\rangle \leq -5\ell(\omega_0)$  for each $i\in I$, then we have
     {\upshape $$\mathfrak{S}_x\mathfrak{S}_y=\sum_{wt_{\mu}\in W_{\scriptsize\mbox{af}}^-\atop \ell(wt_\mu)\geq \ell(x)+\ell(y)}
                         \sum_{v_1\in W, \lambda_1, \lambda_2\in\tilde Q^\vee\atop \lambda_1\succcurlyeq \eta, \lambda_2\succcurlyeq \kappa, \lambda_1+\lambda_2\preccurlyeq \mu}
                           \!\!\!\!\! \!c_{x, [v_1t_{\lambda_1}]}c_{y, [v_1t_{\lambda_2}]}d_{wt_\mu, [v_1t_{\lambda_1+\lambda_2}]}  \mathfrak{S}_{wt_\mu}.$$
                  }
\end{prop}

\begin{proof}
 Note that $\mathfrak{S}_x\mathfrak{S}_y=\sum_{z\in W_{\scriptsize\mbox{af}}^-; \ell(z)\geq \ell(x)+\ell(y)}b_{x, y}^{\, z}\mathfrak{S}_z$ by  Corollary \ref{coe}.
    Now let $z=wt_\mu\in  W_{\scriptsize\mbox{af}}^-$ with $\ell(z)\geq \ell(x)+\ell(y)=\langle \eta+\kappa,-2\rho\rangle-\ell(u)-\ell(v)$.

    Note that $c_{x, [v_1t_{\lambda_1}]}\neq 0$ only if $v_1t_{\lambda_1}\preccurlyeq ut_\eta\preccurlyeq t_\eta$, which implies $\lambda_1\succcurlyeq \eta$;
              $c_{y, [v_2t_{\lambda_2}]}\neq 0$ only if $v_2t_{\lambda_2}\preccurlyeq vt_\kappa\preccurlyeq t_\kappa$, which implies $\lambda_2\succcurlyeq \kappa$.
              Hence,  $\lambda_1= \eta+\lambda_3=\eta+\sum_{i\in I}a_i\alpha_i^\vee$,  $\lambda_2=\kappa+\lambda_4=\kappa+\sum_{i\in I}b_i\alpha_i^\vee$ with $a_i, b_i\geq 0$
               for each $i\in I$.
    Note that $d_{z, [v_2t_{v_2^{-1}v_1(\lambda_1)+\lambda_2}]}\neq 0$ only if $z \preccurlyeq v_2t_{v_2^{-1}v_1(\lambda_1)+\lambda_2}$, in particular only if
              $\ell(v_2t_{v_2^{-1}v_1(\lambda_1)+\lambda_2})\geq \ell(z)
              \geq\ell(t_{\eta+\kappa})-\ell(u)-\ell(v)\geq\ell(t_{\eta+\kappa})-2\ell(\omega_0)$.
              Furthermore,
           \begin{align*}
                    \ell(v_2t_{v_2^{-1}v_1(\lambda_1)+\lambda_2})&=\ell(v_1t_{\lambda_1}v_1^{-1}v_2t_{\lambda_2})\\
                     &\leq\ell(v_1t_{\lambda_1})+\ell(v_1^{-1}v_2)+\ell(t_{\lambda_2})\\
                             &=\langle \lambda_1, -2\rho\rangle-\ell(v_1)+\ell(v_1^{-1}v_2)+\langle\lambda_2, -2\rho\rangle\\
                        &=\langle \eta+\kappa, -2\rho\rangle-2\sum\nolimits_{i\in I}(a_i+b_i)\langle \alpha_i^\vee, \rho\rangle-\ell(v_1)+\ell(v_1^{-1}v_2)\\
                         &\leq\ell(t_{\eta+\kappa})-2\sum\nolimits_{i\in I}(a_i+b_i)-0+\ell(\omega_0).
                    \end{align*}
    Hence, $2\sum\nolimits_{i\in I}(a_i+b_i)\leq 3\ell(\omega_0)$. In particular, $0\leq2a_i\leq 3\ell(\omega_0), 0\leq2 b_i\leq 3\ell(\omega_0)$,
           $\langle\lambda_1, \alpha_i\rangle =\langle \eta, \alpha_i\rangle+a_i\langle \alpha_i^\vee, \alpha_i\rangle+\sum_{j\neq i}a_j\langle \alpha_j^\vee, \alpha_i\rangle
             \leq-5\ell(\omega_0)+3\ell(\omega_0)+0= -2\ell(\omega_0)$ and  $\langle\lambda_2, \alpha_i\rangle\leq -2\ell(\omega_0)$ for each $i\in I$.
   Therefore, $v_1=v_2$; since if $v_1\neq v_2$, then a contradiction comes out following from Lemma \ref{lemm}:

        $\ell(v_2t_{v_2^{-1}v_1(\lambda_1)+\lambda_2})\leq \ell(v_2)+  \ell(t_{\lambda_1+\lambda_2})-4\ell(\omega_0)\leq
           \ell(\omega_0)+  \ell(t_{\eta+\kappa})-4\ell(\omega_0)$.

      So far, we have shown that the effective summation for $z=wt_\mu$ runs over those elements $v_1t_{\lambda_1}, v_1t_{\lambda_2}\in W_{\scriptsize\mbox{af}}^-$
       with $\lambda_1\succcurlyeq \eta$ and $\lambda_2\succcurlyeq \kappa$. Note that $\lambda_1, \lambda_2\in \tilde Q^\vee$ are regular,
          then $v_1t_{\lambda_i}\in W_{\scriptsize\mbox{af}}^-$ for any $v_1\in W$. Note
          $d_{z, [v_1t_{ \lambda_1+\lambda_2}]}\neq 0$ only if $wt_\mu \preccurlyeq v_1t_{ \lambda_1 +\lambda_2}\preccurlyeq t_{ \lambda_1 +\lambda_2}$,
          which implies $ { \lambda_1 +\lambda_2}\preccurlyeq \mu$. Thus the statement follows.
\end{proof}

\bigskip

Due to Peterson's Theorem (Theorem \ref{th11}) and Theorem \ref{str}, in fact,
       we obtain  an explicit formula   for all the equivariant quantum Schubert structure constants $\tilde N_{u, v}^{w, \lambda}$'s.
     Namely we just need to find $\eta, \kappa, \mu\in \tilde Q^\vee$ such that $x=ut_\eta, y=vt_\kappa,   z=wt_\mu$ lie in {\upshape
                $  W_{\scriptsize\mbox{af}}^- $}  and $\lambda=\mu-\eta-\kappa$, and then compute
                  $\tilde N_{u, v}^{w, \lambda}=(-1)^{\ell(z)-\ell(x)-\ell(y)}b_{x, y}^{\, z} .$
       In particular, we obtain a formula for quantum Schubert structure constants $N_{u, v}^{w, \lambda}$ by taking the non-equivariant limit
       $(\alpha_1, \cdots, \alpha_n)\to (0, \cdots, 0)$ at the origin.
       Although $c_{x, [t_1]}, c_{y, [t_2]}$ are rational functions, the summation
            $b_{x, y}^{\,z}=\sum_{t_1, t_2} c_{x, [t_1]}, c_{y, [t_2]}d_{z, [t_1t_2]}$ turns out to be a polynomial
            in $\alpha_i$'s so that the non-equivariant limit does exist, which will  equal  0 if the degree does not match.
       Furthermore, if the degree matches, then $b_{x, y}^{\, z}$ turns out to be a constant function (Corollary \ref{coe}) so that we can
       take a non-equivariant limit at any point. In order to compute  $N_{u, v}^{w, \lambda}$'s  in practise (by hand or by computer), we would like to
       choose $\eta, \kappa$ such that $\ell(x), \ell(y)$ are as small as possible. However, in order to give a neat formula, we would like
        to choose $A=\eta=\kappa$ that satisfies the assumption of \ref{lemm22}. For instance, we have chosen one such $A$ in Theorem \ref{th22}, in which
        the remark ``which is in fact a regular and anti-dominant element in $Q^\vee$" ensures that  $ut_A, vt_A, wt_{2A+\lambda}$ do lie in $
                  W_{\scriptsize\mbox{af}}^- $.

\begin{proof}[Proof of Theorem \ref{th22}]
  Note that
         $N_{u, v}^{w, \lambda}=\tilde N_{u, v}^{w, \lambda}\big|_{\alpha_1= \cdots=\alpha_n=0}$, the evaluation of the equivariant quantum Schubert structure constant
                  $\tilde N_{u, v}^{w, \lambda}\in\mathbb{Q}[\alpha_1, \cdots, \alpha_n]$ at the origin $(\alpha_1, \cdots \alpha_n)=(0, \cdots, 0)$.

    It follows from Theorem \ref{th11} that $\tilde N_{u, v}^{w, \lambda}=(-1)^{\ell(z)-\ell(x)-\ell(y)}b_{x, y}^{\, z}$ for
            $x, y, z\in W_{\scriptsize\mbox{af}}^-$   with  $x=ut_\eta, y=vt_\kappa, z=wt_\mu$ and $\lambda=\mu-\eta-\kappa$.
         Note  that $b_{x, y}^{\, z}$ is a homogeneous (rational) polynomial of degree
                  $\ell(z)-\ell(x)-\ell(y)=\ell(t_\mu)-\ell(w)-(\ell(t_\eta)-\ell(u))-(\ell(t_\kappa)-\ell(v))=\ell(u)+\ell(v)-\ell(w)-\langle \lambda, 2\rho\rangle$.
  \begin{enumerate}
     \item If $\langle \lambda, 2\rho\rangle>\ell(u)+\ell(v)-\ell(w)$, then it follows form Corollary \ref{coe} that $b_{x, y}^{\, z}=0$ and therefore
                       $N_{u, v}^{w, \lambda}=0$.
            If $\langle \lambda, 2\rho\rangle<\ell(u)+\ell(v)-\ell(w)$, then  $b_{x, y}^{\, z}$ is a homogeneous polynomial of positive degree $\ell(z)-\ell(x)-\ell(y)>0$.
                 The evaluation of $b_{x, y}^{\, z}$ at the origin is 0,  and therefore
                       $N_{u, v}^{w, \lambda}=0$.
     \item If $\langle \lambda, 2\rho\rangle=\ell(u)+\ell(v)-\ell(w)$, then  $b_{x, y}^{\, z}$ is a constant polynomial. In particular,
          $$N_{u, v}^{w, \lambda}=\tilde N_{u, v}^{w, \lambda}\big|_{\alpha_1=\cdots=\alpha_n=0}= b_{x, y}^{\, z}\big|_{\alpha_1=\cdots=\alpha_n=0}=b_{x, y}^{\, z}\big|_{\alpha_1=\cdots=\alpha_n=1}.$$

         Take $\eta, \kappa, \mu\in \tilde Q^\vee$ such that $x, y, z\in W_{\scriptsize\mbox{af}}^-$ with $x=ut_\eta, y=vt_\kappa, z=wt_\mu$ and $\lambda=\mu-\eta-\kappa$.
         This can be done as follows.

         The possible determinant of the Cartan matrix $\big(\langle \alpha_i^\vee,\alpha_j\rangle\big)$ is $1, 2, 3, 4$ and $n+1$ (see e.g. section 13 of \cite{hum}).
         As a consequence, the element $A=-12n(n+1)\sum_{i\in I}w_i^\vee$ is  in the coroot lattice $Q^\vee$. Furthermore, $A\in \tilde Q^\vee$
         and $\langle A, \alpha_i\rangle =-12n(n+1)<-5|R^+|=-5\ell(\omega_0)$ (see e.g. \cite{hum}). Note that
         $\lambda=\sum_{i\in I}a_i\alpha_i^\vee\succcurlyeq 0$ and
             $\langle \lambda, \alpha_i\rangle \leq  2a_i     \leq 
              \langle \lambda, 2\rho\rangle=\ell(u)+\ell(v)-\ell(w)\leq 2\ell(\omega_0)$. Thus  $\langle 2A+\lambda, \alpha_i\rangle <0$
            for each $i\in I$.
          Let $x=ut_A, y=vt_A$ and $z=wt_{2A+\lambda}$. Then $x, y, z\in W_{\scriptsize\mbox{af}}^-$ and   $\lambda=(2A+\lambda)-A-A$.

          Hence, the first formula
           follows from Theorem \ref{str} and  Theorem \ref{th11},  and the second formula follows
                 from Proposition \ref{lemm22} immediately.
  \end{enumerate}
\end{proof}

 The following vanishing criterion is also a  consequence of Peterson's Theorem.

\begin{prop}\label{pro55}
   For any $u, v, w\in W$ and $\lambda\in Q^\vee$ with $\lambda\succcurlyeq 0$,
             we take $\eta, \kappa\in\tilde Q^\vee$ such that
                $x=ut_\eta$ and $y=vt_\kappa$ lie in  {\upshape $W_{\scriptsize\mbox{af}}^-$} and we denote $\mu=\eta+\kappa+\lambda$. If {\upshape $wt_\mu\not\in W_{\scriptsize\mbox{af}}^-$},
                  then the equivariant quantum Schubert structure constant   $\tilde N_{u, v}^{w, \lambda}$ 
                  vanishes.
\end{prop}

\begin{proof}
   Denote $d=\ell(u)+\ell(v)-\ell(w)-\langle \lambda, 2\rho\rangle$.
   Take $M\in \mathbb{N}$ with $12(n+1)|M$ and $M\gg 0$  such that $B=-M\sum_{i\in I}w_i^\vee\in \tilde Q^\vee$, which does exist (following the proof of Theorem \ref{th22}),
   and  $\mu+2B\in \tilde Q^\vee$ is regular. Then $xt_B, yt_B, wt_{\mu+2B}\in W_{\scriptsize\mbox{af}}^-$. Therefore it follows from Theorem \ref{th11} that
   $\tilde N_{u, v}^{w, \lambda}=(-1)^db_{xt_B, yt_B}^{\, wt_{\mu+2B}}$. On the other hand,
     it follows from Proposition \ref{prop22} that
       $$\mathfrak{S}_{xt_B}\mathfrak{S}_{yt_B}=\mathfrak{S}_{t_{2B}}\mathfrak{S}_x\mathfrak{S}_y=\mathfrak{S}_{t_{2B}}\sum\nolimits_{z\in W_{\scriptsize\mbox{af}}^-} b_{x, y}^{\,z}
        \mathfrak{S}_z= \sum\nolimits_{z\in W_{\scriptsize\mbox{af}}^-} b_{x, y}^{\,z}\mathfrak{S}_{zt_{2B}}.$$
       Therefore for $zt_{2B}\in W_{\scriptsize\mbox{af}}^-$, $b_{xt_B, yt_B}^{\, zt_{2B}}\neq0$ only if $z\in W_{\scriptsize\mbox{af}}^-$.
       Hence,   $\tilde N_{u, v}^{w, \lambda}=0$.
\end{proof}
 As we will see  in section \ref{ex11}, Proposition \ref{pro55} is useful when we need to compute the quantum Schubert structure constants
  for $G/B$ by hand when the rank of $G$ is not too big.

\section{Examples }
In this section, we give two examples to demonstrate the effectiveness of our formula. To make the procedure   precise,  the first example is
simple and includes some more explanations.

\subsection{Type $A_2$} \label{ex11}

  $G=SL(3, \mathbb{C})$;  $B\subset G$ consists of  upper triangular matrices in $G$.
    In this case, $X=G/B=\{V_1\leqslant V_2\leqslant  \mathbb{C}^3~|~ \dim_{\mathbb{C}}V_i=i, i=1, 2\}$. 

 $\Delta=\{\alpha_1 , \alpha_2 \},  R^+=\{\alpha_1, \alpha_2,  \theta=\alpha_1+\alpha_2\}.$
 Denote $s_i=\sigma_{\alpha_i}$, then one has $W=\{1, s_1, s_2,s_1s_2, s_2s_1, s_1s_2s_1\}\cong  S_3$.
    $\sigma^{s_1s_2}\star\sigma^{s_1s_2}$, $\sigma^{s_1s_2}\star\sigma^{s_2s_1}$,
    $\sigma^{s_2s_1}\star\sigma^{s_2s_1}$, $\sigma^{s_1s_2}\star\sigma^{s_1s_2s_1}$,
    $\sigma^{s_2s_1}\star\sigma^{s_1s_2s_1}$ and $\sigma^{s_1s_2s_1}\star\sigma^{s_1s_2s_1}\in QH^*(X)$
       are the only products  that are not given by the quantum Chevalley formula directly. As an application of our theorems,
       we compute one of them in details as follows.

   \bigskip

{\itshape   \noindent \textbf{General discussion}:      For  $u, v\in W$ with $\ell(u)\geq 2$ and $\ell(v)\geq 2$,
      $\sigma^u\star\sigma^v=\sum_{w, \lambda}N^{w, \lambda}_{u, v}q_\lambda\sigma^w$.
          Note that $N^{w, \lambda}_{u, v}=0$ unless  $\lambda=a_1\alpha_1^\vee+a_2\alpha^\vee_2\succcurlyeq 0$ and
           $2(a_1+a_2)=\langle\lambda, 2\rho\rangle=\ell(u)+\ell(v)-\ell(w)\geq 4-\ell(w)\geq 1$, in which case $q_\lambda=q_1^{a_1}q^{a_2}_2$. Note that
            $-\theta^{\vee}=-\alpha_1^\vee-\alpha_2^\vee\in\tilde Q^\vee$ is regular. Therefore,} {\upshape $x=ut_{-\theta^\vee}, y=vt_{-\theta^\vee}\in W_{\scriptsize\mbox{af}}^-$.}
 {\itshape          Let $z=wt_{-2\theta^\vee+\lambda}$. By Proposition \ref{pro55},
            Theorem \ref{th11} and Theorem \ref{th22},  we have the following
}
          \begin{enumerate}
              \item[(i)] \textit{If}   $z\not\in W_{\scriptsize\mbox{af}}^-$  \textit{or} $\ell(z)\neq \ell(x)+\ell(y)$, \textit{then } $N^{w, \lambda}_{u, v}=0$.
              \item[(ii)]   \textit{If } {\upshape $z\in W_{\scriptsize\mbox{af}}^-$},  \textit{then}
                 {\upshape  
                 $$N^{w, \lambda}_{u, v}=b_{x, y}^{\, z}=\sum_{t_1, t_2\in Q^\vee} c_{x, [t_1]}c_{y, [t_2]}d_{z, [t_1t_2]}
                 =\sum
                                   c_{x, [v_1t_{\lambda_1}]}c_{y, [v_2t_{\lambda_2}]}
                                   d_{z, [v_2t_{v_2^{-1}v_1(\lambda_1)+\lambda_2}]},$$
                          }
         \end{enumerate}
{\itshape  \noindent  where the  effective summation runs over those  } $ v_it_{\lambda_i}=m_{[t_i]}\in W_{\scriptsize\mbox{af}}^-$
{\itshape  satisfying         $ v_1t_{\lambda_1}\preccurlyeq x$ and $v_2t_{\lambda_2}\preccurlyeq y$.}
 {\itshape           Furthermore if $x \neq 1$, then $d_{z, [t_2]}=0$ as $\ell(z)>\ell(y)\geq\ell([t_2])$. In particular if $x, y\neq 1$,
             then we do not need to consider the case $v_it_{\lambda_i}=1$.}

\bigskip
\noindent \textbf{Calculation for the case  \, $u=s_1s_2$ and $v=s_1s_2s_1$.}
\bigskip

               \noindent  In this case, $x=s_1s_2t_{-\theta^\vee}=s_2s_0$ and $y=s_1s_2s_1t_{-\theta^\vee}=s_0$.
             $\lambda=a_1\alpha_1^\vee+a_2\alpha_2^\vee$ with $a_1, a_2\geq 0$ and
                          $2(a_1+a_2)=\ell(u)+\ell(v)-\ell(w)=5-\ell(w)$. Hence, $(a_1, a_2)=(1, 1), (2, 0), (0, 2), (1, 0)$ or $(0, 1)$.

                   If $(a_1, a_2)=(2, 0)$, then  $z\not\in W_{\scriptsize\mbox{af}}^-$ by noting $-2\theta^\vee+\lambda=2\alpha_2^\vee\not\in \tilde Q^\vee$.
                     If $(a_1, a_2)=(1, 0)$, then $\lambda=\alpha_1$,
                     $\ell(w)=3$ and $w=s_1s_2s_1$. Since $\langle-2\theta^\vee+\alpha_1, \alpha_1\rangle=0$ while $w(\alpha_1)=-\alpha_2\not\in R^+$,
                      $z\not \in W_{\scriptsize\mbox{af}}^-$. Similarly, we can show  $z\not \in W_{\scriptsize\mbox{af}}^-$ if
                         $(a_1, a_2)=  (0, 2) $ or $(0, 1)$. Hence,  $N_{u, v}^{w, \lambda}=0$ unless  $(a_1, a_2)=(1, 1)$.

                   If $(a_1, a_2)=(1, 1)$, then  $\ell(w)=5-2(1+1)=1$ and therefore $w=s_1$ or $s_2$.
                       Hence,  $\sigma^{u}\star\sigma^{v}=C_1q_1q_2\sigma^{s_1}+C_2q_1q_2\sigma^{s_2}$ for some real numbers $C_1$ and $C_2$.

                    Note that
                   $1\neq v_1t_{\lambda_1}\preccurlyeq x=s_2s_0$  with $v_1t_{\lambda_1}\in W_{\scriptsize\mbox{af}}^-$ implies
                    that $v_1t_{\lambda_1}=s_0 
                     $ or $s_2s_0$.
                   \, $1\neq v_2t_{\lambda_2}\preccurlyeq y=s_0$  with $v_2t_{\lambda_2}\in W_{\scriptsize\mbox{af}}^-$ implies
                    that $v_2t_{\lambda_2}=s_0=y.$  Hence,
                   \begin{align*} \mbox{}   b_{x, y}^{\, z}&= c_{x, [s_0]}c_{y, [y]}
                                                         d_{z, [\sigma_\theta t_{-\theta^\vee-\theta^\vee}]}+c_{x, [s_2s_0]}c_{y, [y]}
                                                    d_{z, [\sigma_\theta t_{\sigma_\theta s_1s_2(-\theta^\vee)-\theta^\vee}]}\\
                                                 &= c_{s_2s_0, s_0 }' c_{s_0, s_0}'
                                                      d_{z, [\sigma_\theta t_{-2\theta^\vee}]}+c_{s_2s_0, s_2s_0}' c_{s_0, s_0}'
                                                    d_{z, [\sigma_\theta t_{-\alpha_2^\vee-\theta^\vee}]}.
                   \end{align*}

                \begin{align*}\mbox{ Note that}\qquad c_{s_0, s_0}'&=(-1)^1\cdot {1\over s_0(\alpha_0)}\big|_{\alpha_0=-\theta}=-{1\over \theta};\\
                              c_{s_2s_0, s_0}'&=(-1)^2\cdot{1\over  \alpha_2  s_0(\alpha_0)}\big|_{\alpha_0=-\theta}={1\over \alpha_2 \theta};\\
                                c_{s_2s_0, s_2s_0}'&=(-1)^2{1\over s_2(\alpha_2)s_2s_0(\alpha_0)}\big|_{\alpha_0=-\theta}
                                                      ={1\over \alpha_2s_2(\alpha_0)}\big|_{\alpha_0=-\theta}=-{1\over \alpha_2\alpha_1}.
              \end{align*}

                 Note that  $ \sigma_\theta t_{-2\theta^\vee}=[s_0s_2s_1s_2s_0]_{\scriptsize\mbox{red}}$
                   and    $ \sigma_\theta t_{-\alpha_2^\vee-\theta^\vee}=[s_0s_2s_1s_0s_1]_{\scriptsize\mbox{red}}$.

                       Now for $w=s_1$ and $\lambda=\theta^\vee$, we have $z=s_1t_{-\theta^\vee}=[s_2s_1s_0]_{\scriptsize\mbox{red}}$. Therefore
                        \begin{align*}
                          d_{z, [\sigma_\theta t_{-2\theta^\vee}]}&=d_{s_2s_1s_0,  [s_0s_2s_1s_2s_0]}
                                                                  =s_0(\alpha_2)s_0s_2(\alpha_1)s_0s_2s_1s_2(\alpha_0)\big|_{\alpha_0=
                                                                     -\theta}=-\alpha_1\theta^2;\\
                           d_{z, [\sigma_\theta t_{-\alpha_2^\vee-\theta^\vee}]}
                                       &=d_{s_2s_1s_0, [s_0s_2s_1s_0s_1]}=s_0(\alpha_2)s_0s_2(\alpha_1)s_0s_2s_1(\alpha_0)\big|_{\alpha_0=-\theta}
                                                   = - \alpha_1^2\theta.
                        \end{align*}
        \vspace{-0.5cm}

              $$\mbox{ Hence,}\quad C_1= b_{x, y}^{\, z}={1\over \alpha_2 \theta}\cdot {-1\over \theta}\cdot({  -\alpha_1\theta^2})
                                             +   {-1\over \alpha_2\alpha_1} {-1\over \theta}\cdot( - \alpha_1^2\theta)={\alpha_1\over \alpha_2 }+(-{\alpha_1\over \alpha_2}) = 0.
                                                      \qquad\quad\mbox{}       $$

                         Now for $w=s_2$ and $\lambda=\theta^\vee$, we have $z=s_2t_{-\theta^\vee}=[s_1s_2s_0]_{\scriptsize\mbox{red}}$.
                               Note that
                           $d_{z, [\sigma_\theta t_{-2\theta^\vee}]}=d_{s_1s_2s_0, [s_0s_2s_1s_2s_0]}=s_0s_2(\alpha_1)s_0s_2s_1(\alpha_2)s_0s_2s_1s_2(\alpha_0)\big|_{\alpha_0=-\theta}
                            =-\alpha_2\theta^2$ and that
                            $   d_{z, [\sigma_\theta t_{-\alpha_2^\vee-\theta^\vee}]}= d_{s_1s_2s_0, [s_0s_2s_1s_0s_1]}=0$ as $s_1s_2s_0\not\preccurlyeq s_0s_2s_1s_0s_1$.
                      Therefore,  $C_2=b_{x, y}^{\, z}={1\over \alpha_2 \theta}\cdot {-1\over \theta}\cdot(-\alpha_2\theta^2)=1 $.

                Hence,
                       $$\quad \sigma^{s_1s_2}\star\sigma^{s_1s_2s_1}=q_1q_2\sigma^{s_2}.\hfill$$
                  Similarly, we can compute  quantum products for the remaining   cases.

\subsection{Type $B_3$}\label{ex22}
   $G=Spin(7, \mathbb{C})$; 
         $X=G/B=\{V_1\leqslant V_2\leqslant V_3\leqslant \mathbb{C}^7~|~ \dim_{\mathbb{C}}V_i=i,$ $  (V_i, V_{i})=0, i=1,2,3\},
                     $ where $(\cdot, \cdot)$ is a quadratic form  on $\mathbb{C}^7$.
                     See e.g. \cite{hum} for 
\upshape{
       $$\typeb\,\,;\quad\theta=\alpha_1+2\alpha_2+2\alpha_3=s_2s_3s_2(\alpha_1),\quad \theta^\vee=\alpha_1^\vee+2\alpha_2^\vee+\alpha_3^\vee,\quad |R^+|=9, $$
\vspace{-0.8cm}  \\ \noindent $\mbox{}$\hspace{-0.5cm} $\cordin$

\noindent $W$ is generated by simple reflections $\{s_1, s_2, s_3\}$, $|W|=48$, $\sigma_\theta=[s_2s_3s_2s_1s_2s_3s_2]_{\scriptsize\mbox{red}}$.
\bigskip

\noindent\textbf{ Calculation for}  $\sigma^u\star\sigma^v$,  where   $u=s_1s_2s_3s_1s_2$ and $v=s_3s_1s_2s_3s_1s_2$.
\bigskip

Note that $\langle -\theta^\vee, \alpha_2\rangle=-1<0$, $\langle -\theta^\vee, \alpha_i\rangle=0$ while $u(\alpha_i)\succcurlyeq0$
and $v(\alpha_i)\succcurlyeq0$ for $i=1, 3$. Hence, it suffices to take $A=-\theta^\vee$.  Then one has
         $x=ut_{-\theta^\vee}=[s_3s_2s_0]_{\scriptsize\mbox{red}}$ and $ y=vt_{-\theta^\vee}=[s_2s_0]_{\scriptsize\mbox{red}}$.
     Note that $u(-\theta^\vee)=\alpha_1^\vee+\alpha_2^\vee$ and $v(-\theta^\vee)=\alpha_1^\vee+\alpha_2^\vee+\alpha^\vee_3$.
           Denote $m_t$ the minimal length representative  in the coset $tW$ for $t\in Q^\vee$, then one has
          $$b_{x, y}^{\, z}=\sum 
                                   c_{x, [v_1t_{\lambda_1}]}c_{y, [v_2t_{\lambda_2}]}
                                   d_{z, [m_{t_1t_2}]},
               $$
        where $ v_it_{\lambda_i}=m_{t_i}\in W_{\scriptsize\mbox{af}}^-$ and the  effective summation runs over those satisfying
         $1\neq v_1t_{\lambda_1}\preccurlyeq x$ and $1\neq v_2t_{\lambda_2}\preccurlyeq y$. Explicitly, the possible nonzero terms are listed in the following table,
         where we denote
             $\eta=-2\alpha_1^\vee-3\alpha_2^\vee-2\alpha_3^\vee$, $\kappa=-2\alpha_1^\vee-2\alpha_2^\vee-\alpha_3^\vee.$

\begin{center}
\begin{tabular}{|r|l|c|c|c|c|}
  \hline
 $([v_1t_{\lambda_1}], [v_2t_{\lambda_2}])$ &  $ t_1\cdot t_2$  &  $m_{t_1t_2}$  &   $[m_{t_1t_2}]_{\scriptsize\mbox{red}}$    \\ \hline\hline
$([s_0], [s_0])$   & $t_{\theta^\vee+\theta^\vee}$   & $\sigma_\theta t_{-2\theta^\vee}$   &     $s_0s_2s_3s_2s_1s_2s_3s_2s_0$     \\ \hline
$([s_0], [s_2s_0])$   & $t_{\theta^\vee+v(-\theta^\vee)}$  & $s_1s_2s_3s_2s_1s_3s_2s_3t_{\eta}$  &  $s_0s_2s_3s_1s_2s_0$       \\ \hline
$([s_2s_0], [s_0])$   & $t_{v(-\theta^\vee)+\theta^\vee}$  & $s_1s_2s_3s_2s_1s_3s_2s_3t_{\eta}$  &  $s_0s_2s_3s_1s_2s_0$       \\ \hline
$([s_2s_0], [s_2s_0])$   & $t_{v(-\theta^\vee)+v(-\theta^\vee)}$  & $ vt_{-2\theta^\vee}$ &  $s_2s_0s_2s_3s_2s_1s_2s_3s_2s_0$          \\ \hline
$([s_2s_1s_0], [s_0])$   &  $t_{u(-\theta^\vee)+\theta^\vee}$ & $s_1s_2s_1s_3s_2s_1s_3t_{\eta}$  &  $s_0s_3s_2s_3s_1s_2s_0$         \\ \hline
$([s_2s_1s_0], [s_2s_0])$   & $t_{u(-\theta^\vee)+v(-\theta^\vee)}$  & $s_1s_2s_3s_2s_1t_{\kappa}$    &   $s_0s_2s_3s_2s_0$       \\ \hline
\end{tabular}
\end{center}
 Note that $s_3s_0=s_0s_3\in s_0W$. By definition,
    \begin{align*}
        c_{x, [s_0]}&=c_{s_3s_2s_0, s_0}'+c_{s_3s_2s_0, s_3s_0}'=
                            ({(-1)^3\over \alpha_3\alpha_2s_0(\alpha_0)}+{(-1)^3\over s_3(\alpha_3)s_3(\alpha_2)s_3s_0(\alpha_0)}\big)\big|_{\alpha_0=-\theta}\\
                    &\hspace{3.85cm}={-2\over \alpha_2\theta(\alpha_2+2\alpha_3)};\\
         c_{x, [s_2s_0]}&=c_{s_3s_2s_0, s_2s_0}'=
                            {(-1)^3\over \alpha_3s_2(\alpha_2)s_2s_0(\alpha_0)}\big|_{\alpha_0=-\theta}={1\over \alpha_2\alpha_3(\alpha_1+\alpha_2+2\alpha_3)};\\
         c_{x, [s_3s_2s_0]}&=c_{s_3s_2s_0, s_3s_2s_0}'=
                            {(-1)^3\over s_3(\alpha_3)s_3s_2(\alpha_2)s_3s_2s_0(\alpha_0)}\big|_{\alpha_0=-\theta}={-1\over \alpha_3(\alpha_2+2\alpha_3)(\alpha_1+\alpha_2)};\\ 
          c_{y, [s_0]}&= c_{s_2s_0, s_0}'=
                            {(-1)^2\over  \alpha_2 s_0(\alpha_0)}\big|_{\alpha_0=-\theta}={1\over \alpha_2\theta};\\
  c_{y, [s_2s_0]}&= c_{s_2s_0, s_2s_0}'=
                            {(-1)^2\over  s_2(\alpha_2)s_2 s_0(\alpha_0)}\big|_{\alpha_0=-\theta}={-1\over \alpha_2(\alpha_1+\alpha_2+2\alpha_3)}.
       \end{align*}

 Let $z=wt_{-2\theta^\vee+\lambda}$. Then $N_{u, v}^{w, \lambda}\neq0$  only if $z\in W_{\scriptsize\mbox{af}}^-$ and $\ell(z)=\ell(x)+\ell(y)=5$.
   Note that the only possibilities are $z=s_2s_3s_1s_2s_0$, $s_1s_2s_3s_2s_0$ or $s_0s_2s_3s_2s_0$.

 For $z=s_1s_2s_3s_2s_0=wt_{-2\theta^\vee+\lambda}$, we have $w=s_2s_3s_2$ and $\lambda=\alpha_1^\vee+2\alpha_2^\vee+\alpha_3^\vee$.
 Note that  $d_{z, [{t_1t_2}]}\neq 0$ only if $z\preccurlyeq m_{t_1t_2}$. Thus only $d_{z, [\sigma_\theta t_{-2\theta^\vee}]}$ and $d_{z, [vt_{-2\theta^\vee}]}$
 are nonzero.
 Furthermore, we have
  $$d_{z, [\sigma_\theta t_{-2\theta^\vee}]}=-\alpha_2\theta^2(\alpha_2+\alpha_3)(\alpha_2+2\alpha_3),\quad
   d_{z, [vt_{-2\theta^\vee}]}=\alpha_2\alpha_3(\alpha_2+2\alpha_3)(\alpha_1+\alpha_2+2\alpha_3)^2.$$
  Hence,
    \begin{align*} b_{x, y}^{\, z}&=c_{x, [s_0]}c_{y, [s_0]}d_{z, [\sigma_\theta t_{-2\theta^\vee}]}+c_{x, [s_2s_0]}c_{y, [s_2s_0]}d_{z, [vt_{-2\theta^\vee}]}\\
                 &={2   \alpha_2\theta^2(\alpha_2+\alpha_3)(\alpha_2+2\alpha_3)\over \alpha_2\theta(\alpha_2+2\alpha_3)\cdot \alpha_2\theta}-
                  { \alpha_2\alpha_3(\alpha_2+2\alpha_3)(\alpha_1+\alpha_2+2\alpha_3)^2 \over \alpha_2\alpha_3(\alpha_1+\alpha_2+2\alpha_3)\cdot \alpha_2(\alpha_1+\alpha_2+2\alpha_3)}\\
                  &={2   ( \alpha_2+\alpha_3) \over \alpha_2 }-
                  { \alpha_2+2\alpha_3 \over \alpha_2 }=1.
       \end{align*}
For $z=s_2s_3s_1s_2s_0=wt_{-2\theta^\vee+\lambda}$, we have $w=s_3s_1s_2$ and $\lambda=\alpha_1^\vee+2\alpha_2^\vee+\alpha_3^\vee$.
\begin{align*}
     d_{z, [t_{2\theta^\vee}]}&=-2\theta^3(\alpha_2+\alpha_3)(\alpha_1+\alpha_2+\alpha_3) ;\\
     d_{z, [t_{\theta^\vee+v(-\theta^\vee)}]}&=-\theta(\alpha_2+\alpha_3)(\alpha_1+\alpha_2+\alpha_3)(\alpha_1+\alpha_2+2\alpha_3)^2 ;\\
     d_{z, [t_{v(-2\theta^\vee)}]}&=-\alpha_3(\alpha_1+\alpha_2+2\alpha_3)^2(2\alpha_1^2+2\alpha_1\alpha_2+\alpha_2^2
           +6\alpha_1\alpha_3+4\alpha_2\alpha_3+4\alpha_3^2) ;\\
     d_{z, [t_{u(-\theta^\vee)+\theta^\vee}]}&=-\alpha_2\theta(\alpha_1+\alpha_2)^2(\alpha_1+\alpha_2+\alpha_3) ;\\
     d_{z, [t_{u(-\theta^\vee)+v(-\theta^\vee)}]}&=0.
 \end{align*}
Substituting them in the summation for $b_{x, y}^{\,z}$ and  simplifying, we obtain   $b_{x, y}^{\,z}=1$.

For $z=s_0s_2s_3s_2s_0=wt_{-2\theta^\vee+\lambda}$, we have $w=s_1s_2s_3s_2s_1$ and $\lambda=2\alpha_2^\vee+\alpha_3^\vee$.
\begin{align*}
     d_{z, [t_{2\theta^\vee}]}&=-\theta^3(\alpha_1^2+3\alpha_1\alpha_2+3\alpha_2^2
           +3\alpha_1\alpha_3+6\alpha_2\alpha_3+2\alpha_3^2);\\
     d_{z, [t_{\theta^\vee+v(-\theta^\vee)}]}&=-\theta^2(\alpha_1+\alpha_2+\alpha_3)(\alpha_1+\alpha_2+2\alpha_3)^2;\\
     d_{z, [t_{v(-2\theta^\vee)}]}&=-(\alpha_1+\alpha_2+2\alpha_3)^2(\alpha_1^2+2\alpha_1\alpha_2
           +3\alpha_1\alpha_3+\alpha_2\alpha_3+2\alpha_3^2);\\
     d_{z, [t_{u(-\theta^\vee)+\theta^\vee}]}&=-\theta^2(\alpha_1+\alpha_2)^2(\alpha_1+\alpha_2+\alpha_3);\\
     d_{z, [t_{u(-\theta^\vee)+v(-\theta^\vee)}]}&=-\alpha_1\theta(\alpha_1+\alpha_2)(\alpha_1+\alpha_2+\alpha_3)(\alpha_1+\alpha_2+2\alpha_3).
 \end{align*}
 Substituting them in the summation for $b_{x, y}^{\,z}$ and  simplifying, we obtain   $b_{x, y}^{\,z}=1$.

Hence, we obtain the following

$$\sigma^{s_1s_2s_3s_1s_2}\star\sigma^{s_3s_1s_2s_3s_1s_2}= q_1q_2^2q_3(\sigma^{s_2s_3s_2}+\sigma^{s_3s_1s_2})+q_2^2q_3\sigma^{ s_1s_2s_3s_2s_1}.$$

\section{Appendix}

\subsection{Proofs of   lemmas in section \ref{seccalcustr}  and  Corollary \ref{sum1}}\label{appelemmaproof} In this subsection, we first prove all the lemmas in section \ref{seccalcustr}
 after reviewing some basic facts on
  affine Weyl group (as a Coxeter group). Then we give the proofs of all the lemmas in section \ref{seccalcustr} as well as   Corollary \ref{sum1}.

Recall that $\mathcal{S}=\{\sigma_i~|~i\in I_{\scriptsize\mbox{af}}\}$.  Denote $\mathcal{T}=\{x\sigma_ix^{-1}~|~ x\in W_{\scriptsize\mbox{af}}, \sigma_i\in\mathcal{S}\}=\{\sigma_\gamma~|~\gamma \in R_{\scriptsize\mbox{re}}^+ \}$.
    Let $x, x'\in W_{\scriptsize\mbox{af}}$. We say $x$ covers $x'$, denoted by $x'\rightarrow  x$ or $x'\overset{\sigma_\gamma}{\rightarrow}  x$,  if  there exists some  $\sigma_\gamma\in\mathcal{T}$ such that
     $x=\sigma_\gamma x'$ and $\ell(x)=\ell(x')+1$.
      We say      $x'\preccurlyeq x$ with respect to the  Bruhat order  $(W_{\scriptsize\mbox{af}}, \preccurlyeq)$,
   if  there exists a chain  $ x'=x_1\rightarrow x_2\rightarrow\cdots \rightarrow x_k=x.$ We list some well-known facts (from \cite{hill} and \cite{hump})
     for the Coxeter system
    $(W_{\scriptsize\mbox{af}},\mathcal{S})$ as follows.
 \begin{lemma}\label{rword}  Let   {\upshape $x, y\in W_{\scriptsize\mbox{af}}$} with  {\upshape $x=[\sigma_{\beta_1}\cdots\sigma_{\beta_r}]_{\scriptsize\mbox{red}}$}.
                      Denote $\gamma_k=\sigma_{\beta_1}\cdots\sigma_{\beta_{k-1}}(\beta_{k})$.
     \begin{enumerate}
       \item[{\upshape (a)}]  If $y\, {\rightarrow} \,x$, then there exists a unique $j, 1\leq j\leq r$, such that $x=\sigma_{\gamma_j} y$ and
              {\upshape  $y=[\sigma_{\beta_1}\cdots\sigma_{\beta_{j-1}}\sigma_{\beta_{j+1}}\cdots \sigma_{\beta_r}]_{\scriptsize\mbox{red}}.$}
     \item[{\upshape (b)}]  If $y\preccurlyeq x$, 
                      then {\upshape $y= [\sigma_{\beta_{k_1}}\cdots\sigma_{\beta_{k_s}}]_{\scriptsize\mbox{red}}$} for some subsequence  $(k_1,\cdots, k_s)$, which we call an induced
                      reduced decomposition of $y$ from {\upshape $x=[\sigma_{\beta_1}\cdots\sigma_{\beta_r}]_{\scriptsize\mbox{red}}$}.
                       In particular, if $y\preccurlyeq x$ and $\ell(x)=\ell(y)+1$, then $y\overset{\sigma_{\gamma_j}}{\to}x$ for a unique $j$.
     \item[{\upshape (c)}]{\upshape {(Lifting Property)}}
          Let  $\sigma_i\in \mathcal{S}$. Suppose $\ell(\sigma_ix)>\ell(x)$ and $\ell(\sigma_iy)<\ell(y)$, then the following are equivalent:
   {\upshape $$\mbox{(i)}\,\, \sigma_ix\preccurlyeq y; \qquad \mbox{(ii)}\,\, x\preccurlyeq y;\qquad \mbox{(iii)}\,\, x\preccurlyeq \sigma_iy.$$}
     \item[{\upshape (d)}] {\upshape $\{\gamma_1, \cdots, \gamma_r\} = \{ \gamma\in R^+_{\scriptsize\mbox{re}}~|~ x^{-1}(\gamma)\in -R^+_{\scriptsize\mbox{re}}\}; \qquad\ell(xy)\leq \ell(x)+\ell(y)$.   }
      \item[{\upshape (e)}]     Let {\upshape $\gamma\in R^+_{\scriptsize\mbox{re}}$}.   {\upshape $ \ell(\sigma_\gamma x)\leq\ell(x) \Longleftrightarrow x^{-1}(\gamma)\in  -R^+_{\scriptsize\mbox{re}}$}.

  \end{enumerate}
\end{lemma}

\begin{lemma}[see  e.g.  \cite{mare}]\label{len}
    For any $\gamma\in R^+$,\,\,\,   $\ell(\sigma_\gamma)\leq \langle \gamma^\vee, 2\rho\rangle-1$.
\end{lemma}

The following lemma is on the property of  the longest element $\omega_0$ in $W$.

\begin{lemma}\label{lemm55}
  $\omega_0(-\theta)=\theta$.
\end{lemma}
\begin{proof}
   (We learned the proof from Victor Reiner.)
   The highest root $\theta_0\in R^+$ is characterized among all the positive roots by the property   that $\langle \theta_0, \alpha_i^\vee\rangle\geq 0$ for all
   $\alpha_i$'s.
   Note that $\omega_0(-\theta)$ is a positive roots. Furthermore for any simple root $\alpha_i$, $\ell(\sigma_{\omega_0(\alpha_i)})=\ell(\omega_0\sigma_i\omega_0)=\ell(\omega_0)-\ell(\sigma_i\omega_0)=
               \ell(\omega_0)-(\ell(\omega_0)-\ell(\sigma_i))=1$, which implies that $\omega_0(\alpha_i)=-\alpha_j$ for some $j\in I$.
           Note that $\omega_0=\omega_0^{-1}$.    Hence for $\theta_0=\omega_0(-\theta)$, $           \langle\theta_0, \alpha_i^\vee \rangle=\langle\omega_0(-\theta), \alpha_i^\vee\rangle
            =\langle\theta,-\omega_0(\alpha_i)^\vee\rangle=   \langle\theta , \alpha_j^\vee\rangle\geq0$.
           Thus $\theta_0=\theta$.
\end{proof}

The next lemma is a consequence  of Lemma \ref{len22}
 \begin{lemma}\label{cor1}
 Suppose $\lambda\in Q^\vee$ is anti-dominant and regular.  Then
 for any $w, u\in W$, we have\,\, {\upshape (i)} $wt_\lambda\preccurlyeq t_\lambda$; {\upshape (ii)} $wt_\lambda u\preccurlyeq t_\lambda  $ implies $u=1$.
\end{lemma}
\begin{proof}
  Write $w=[\sigma_{\beta_1}\cdots\sigma_{\beta_r}]_{\scriptsize\mbox{red}}$. Denote $x_j=\sigma_{\beta_j}\cdots\sigma_{\beta_r}t_\lambda$ for each $1\leq j\leq r$.
  Since $\lambda\in \tilde Q^\vee$ is regular, $x_j\in W_{\scriptsize\mbox{af}}^-$ for each $1\leq j\leq r+1$, where we denote $x_{r+1}=t_\lambda$.
     \noindent(i)   $\ell(x_{j+1})=\ell(t_\lambda)-\ell(\sigma_{\beta_{j+1}}\cdots\sigma_{\beta_r})=\ell(t_\lambda)-(r-j)$. Note that $x_{j+1}=\sigma_{\beta_j}x_j$ and
   $\ell(x_{j+1})=\ell(x_j)+1$. Thus $x_j\preccurlyeq x_{j+1}$ for  $1\leq j\leq r$.  Hence, $wt_\lambda=x_1\preccurlyeq x_{r+1}=t_\lambda$.

      \noindent(ii) Note that   $x_j\in W_{\scriptsize\mbox{af}}^-$ for   $1\leq j\leq r+1$. Hence,  $\ell(x_ju)=\ell(x_j)+\ell(u)$,
       $\ell(\sigma_{\beta_j}x_ju)=  \ell(x_{j+1}u) > \ell(x_ju)$ and   $\ell(\sigma_{\beta_j}x_{r+1})=\ell(x_{r+1})-1<\ell(x_{r+1})$.
       Therefore if $x_{j}u\preccurlyeq x_{r+1}$, then $x_{j+1}u=\sigma_{\beta_j}x_ju\preccurlyeq x_{r+1}$  by Lemma \ref{rword}. Hence,
       $x_{1}u=wt_\lambda u\preccurlyeq t_\lambda=x_{r+1}$ implies $t_\lambda u=x_{r+1}u\preccurlyeq x_{r+1}$, by induction on $j$. Thus $\ell(u)=0$ and $u=1$.
\end{proof}

\bigskip

 \begin{proof}[Proof of Lemma \ref{sum3}]
  By Lemma \ref{lemm55},  $t_{\lambda}=t_{-\theta^\vee}t_{\lambda+\theta^\vee}=\omega_0\sigma_\theta t_{-\theta^\vee}\sigma_\theta \omega_0t_{\lambda+\theta^\vee}$.
   Note that for each $i\in I$,
    $\langle \lambda+ \theta^\vee, \alpha_i\rangle\leq -2+\langle\theta^\vee, \alpha_i\rangle\leq-2+1=-1$. Hence, $\lambda+\theta^\vee\in \tilde Q^\vee$ is regular
     and we have $\sigma_\theta \omega_0t_{\lambda+\theta^\vee}\in W_{\scriptsize\mbox{af}}^-$.
   Hence, any reduced decomposition of  $\sigma_\theta \omega_0t_{\lambda+\theta^\vee}$ must be of the form $u_1\sigma_0\cdots u_r\sigma_0$.
  Furthermore,
     \begin{align*} \ell(t_\lambda)&\leq \ell(\omega_0\sigma_\theta t_{-\theta^\vee})+\ell(\sigma_\theta \omega_0t_{\lambda+\theta^\vee}) \\
                                   &= \ell(\omega_0\sigma_0)+\ell(t_{\lambda+\theta^\vee})-\ell(\sigma_\theta \omega_0)\\
                                   &= \ell(\omega_0)+1 + \langle{\lambda+\theta^\vee}, -2\rho\rangle-(\ell(\omega_0)-\ell(\sigma_\theta))\\
                                   &=1+\ell(t_\lambda)-\langle{\theta^\vee}, 2\rho\rangle+\ell(\sigma_\theta)\\
                                   &\leq 1+\ell(t_\lambda)-\langle{\theta^\vee}, 2\rho\rangle+ \langle \theta^\vee, 2\rho\rangle-1
                                    =\ell(t_\lambda).
     \end{align*}
  Hence, all inequalities are indeed equalities. Thus
              $t_\lambda$ admits a reduced decomposition of the form $\omega_0\sigma_0u_1\sigma_0\cdots u_r\sigma_0$.

  For any $w\in W$, we have $wt_\lambda\preccurlyeq t_\lambda$ by Lemma \ref{cor1}.
  Thus there exist a subsequence $(i_1, \cdots, i_k)$ as required. Furthermore any such a sequence gives an expression of $wt_\lambda$
       of the form $wt_\lambda=u_0'\sigma_0u_1'\sigma_0\cdots u_r'\sigma_0$ with $u_0'\preccurlyeq \omega_0$ and
   $u_j'\preccurlyeq u_j$ for each $1\leq j\leq r$.  If $u_j'\neq u_j$ for some $1\leq j\leq r$, then $\ell(u_j')\leq \ell(u_j)-1$.
   Note that $t_\lambda=\omega_0\sigma_0u_1\sigma_0\cdots u_r\sigma_0=w^{-1}u_0'\sigma_0u_1'\sigma_0\cdots u_r'\sigma_0.$ Thus
      \begin{align*}
                     \ell(\omega_0)+ r+1+\sum\nolimits_{1\leq k\leq r}\ell(u_k)&=  \ell (\omega_0\sigma_0u_1\sigma_0\cdots u_{r}\sigma_0)\\
                             &=    \ell( w^{-1} u_0'\sigma_0u_1'\sigma_0\cdots u_{r}'\sigma_0)\\
                               &\leq \ell(w^{-1} u_0')+r+1 +  \sum\nolimits_{1\leq k\leq r}\ell(u_{k}')\\
                             &\leq \ell(\omega_0)+r+1 + \ell(u_j)-1+\sum\nolimits_{k\neq j}\ell(u_{k})\\
                             &=  \ell(\omega_0)+r +\sum\nolimits_{1\leq k\leq r}\ell(u_k).
            \end{align*}
   This is a contradiction. Hence, the statement follows. That is, induced  decomposition(s) of $wt_\lambda$ must be of the form
    $u_0\sigma_0u_1\sigma_0\cdots u_r\sigma_0$ (in which $u_0=w\omega_0=\sigma_{i_1}\cdots \sigma_{i_a}$).
\end{proof}

\bigskip

\begin{proof}[Proof of Lemma \ref{sum88}]
   Write $m_{[z]}=ut_\mu\in W_{\scriptsize\mbox{af}}^-$, then $z=ut_\mu w$ for some $w\in W$.
              Let $w=[\sigma_{\beta_1}\cdots\sigma_{\beta_k}]_{\scriptsize\mbox{red}}$ and denote $\tilde z=ut_\mu\sigma_{\beta_1}\cdots\sigma_{\beta_{k-1}}$.
              Note that
                    $\ell(\sigma_{\beta_k}x^{-1})=
                           \ell(x^{-1})+1>\ell(x^{-1})$
                    and  $\ell(\sigma_{\beta_k}z^{-1})
                 =\ell( z^{-1})-1<\ell(z^{-1})$.
                 Since $x\preccurlyeq z$, $x^{-1}\preccurlyeq z^{-1}$ and therefore
                    $x^{-1}\preccurlyeq  \sigma_{\beta_k}z^{-1}=\tilde z^{-1}$ by (c) of Lemma \ref{rword}. Hence, $x \preccurlyeq  \tilde z$ with $\ell(\tilde z)=\ell(z)-1$.
           Hence,  the first half of the  statement holds  by induction on $\ell(w)$.

  To prove the second half, we recall that $c_{x, [y]}=\sum_{\tilde y\in yW}c_{x, \tilde y}'$. Given $\tilde y\in yW$, we have $\tilde y=ut_\lambda v$ for some $v\in W$.
                   Note that  $c_{x, \tilde y}'\neq 0$ only if $\tilde y\preccurlyeq x$.
                Suppose $\lambda=\mu$ and $\lambda$ is regular, then $ut_\lambda v \preccurlyeq wt_\lambda\preccurlyeq t_\lambda$
                         implies $v=1$ by Lemma \ref{cor1}.
                Suppose $\ell(x)=\ell(y)+1$.  If  $\tilde y\neq y$, then $\ell(\tilde y)>\ell(y)$ and therefore $\tilde y=x$ follows from $\tilde y\preccurlyeq x$,
                        which contradicts to the uniqueness of the minimal length  representative  in each coset.
                        Hence,  $c_{x, [y]}=c_{x, y}'$ if
                 either of the two assumptions holds.
\end{proof}

\bigskip
\begin{proof}[Proof of Lemma \ref{coe99}]
Let $y=[\sigma_{\beta_1}\cdots\sigma_{\beta_k}]_{\scriptsize\mbox{red}}$ and $w=[\sigma_{\beta_{k+1}}\cdots\sigma_{\beta_{k+s}}]_{\scriptsize\mbox{red}}$.
   Then $\beta_{k+1},\cdots, \beta_{k+s}\in \Delta$, $\beta_k=\alpha_0$ and $yw=[\sigma_{\beta_1}\cdots\sigma_{\beta_{k+s}}]_{\scriptsize\mbox{red}}$.
     For any subsequence $J=(j_1, \cdots, j_a)$ of $(1, \cdots, k+s)$, we denote $\sigma_{J}=\sigma_{\beta_{j_1}}\cdots \sigma_{\beta_{j_a}}$.
        Suppose $x=[\sigma_J]_{\scriptsize\mbox{red}}$, then $J$ must be a subsequence of $(1, \cdots, k)$. Therefore, $d_{x, yw}=d_{x, y}$ by definition.
        Indeed, if $J=J_1\bigsqcup J_2$ with
                  $J_2\subset(k+1,\cdots, k+s)$   nonempty and   $J_1\subset (1, \cdots, k)$, then $\sigma_{J_2}\in W$ and
                               $\ell(x(\sigma_{J_2})^{-1})=\ell(\sigma_{J_1}\sigma_{J_2}(\sigma_{J_2})^{-1})=\ell(\sigma_{J_1})\leq |J_1|<|J|=\ell(x)$, which
             contradicts to the fact that $x$ is of minimal length  in the coset $x W$.

\end{proof}

\begin{proof}[Proof of Corollary \ref{sum1}]
  We can assume $\lambda, \mu$ to be regular. (Otherwise, we take any regular $\tau\in \tilde Q^\vee$ and consider $\lambda+\tau, \mu+\tau$.)
           We  claim that  $wt_\mu\preccurlyeq t_\lambda\Longleftrightarrow t_\mu\preccurlyeq t_\lambda$. Hence, the
             statement  follows from Proposition \ref{stem1} immediately.

         Indeed, we have $wt_\mu\preccurlyeq t_\mu$ by Lemma \ref{cor1}. Suppose   $t_\mu\preccurlyeq t_\lambda$, then
           we have  $wt_\mu\preccurlyeq t_\lambda$. 
           Suppose $wt_\mu\preccurlyeq t_\lambda$. Write $w=[\sigma_{\beta_1}\cdots\sigma_{\beta_k}]_{\scriptsize\mbox{red}}$.
          Note that $\ell(\sigma_{\beta_1}wt_\mu)=\ell(t_\mu)-\ell(w)+1>\ell(wt_\mu)$ and $\ell(\sigma_{\beta_1}t_\lambda)=\ell(t_\lambda)-1<\ell(t_\lambda)$.
          Hence,  $\tilde wt_\mu\preccurlyeq t_\lambda$ holds by Lemma \ref{rword}, where $\tilde w=\sigma_{\beta_1}w$ with $\ell(\tilde w)=\ell(w)-1$.
          Thus we can deduce $t_\mu\preccurlyeq t_\lambda$  by induction on $\ell(w)$.
\end{proof}

\subsection{Proofs of Proposition  \ref{coe44} and Proposition \ref{coe88}}\label{appeproposproof}

\begin{lemma}\label{sum66}
 Let {\upshape $x, y, z\in W_{\scriptsize\mbox{af}}^-$} with $x=\sigma_it_\lambda$ and $y=ut_\mu$ where $i\in I$. Let $t_j\in Q^\vee$   and denote $v_jt_{\lambda_j}=m_{[t_j]}$, j=1, 2.
   Suppose $x\succcurlyeq [t_1], y\succcurlyeq [t_2], z\preccurlyeq [t_1t_2]$ and $\ell(z)\geq\ell(x)+\ell(y)$. 
    Then only the following two possibilities can happen,

       {\upshape\textbf{Case A:}}\quad $\ell([t_1t_2])=\ell(x)+\ell(y)+1$;\qquad
            {\upshape\textbf{Case B:}}\quad  $\ell([t_1t_2])=\ell(z)=\ell(x)+\ell(y)$.
\end{lemma}
\begin{proof} Note that $z\preccurlyeq [t_1t_2]$ implies $\ell(z)\leq \ell([t_1t_2])$. Therefore,
                   \begin{align*}
                       \ell(x)+\ell(y)\leq \ell([t_1t_2])
                                                             &=\ell([v_1t_{\lambda_1}v_1^{-1}v_2t_{\lambda_2}])\\
                                                             &\leq \ell(v_1t_{\lambda_1}v_1^{-1}v_2t_{\lambda_2})\\
                                                             &\leq \ell(v_1t_{\lambda_1})+\ell(v_1^{-1})+\ell(v_2t_{\lambda_2})\\
                                                             &= \ell(t_{\lambda_1}) +\ell(v_2t_{\lambda_2})\\
                                                             &\leq \ell(t_\lambda)+ \ell(y)
                                                              =\ell(x)+1+ \ell(y).
                      \end{align*}
                  Hence, only two cases ({\upshape \textbf{Case A}} or
                         {\upshape\textbf{ Case B}}) are possible.
\end{proof}

 \begin{lemma}\label{pos1}
   Under the same  assumptions as in Lemma \ref{sum66}, we assume $\lambda$ is regular. If {\upshape \textbf{Case A}} occurs,  then $v_1t_{\lambda_1}=ut_\lambda$ and
      $v_2t_{\lambda_2}=ut_\mu$. Furthermore, only one of the following three possibilities can happen,
         \begin{enumerate}
           \item[{\upshape a)}] $z=ut_{\lambda+\mu}$;
           \item[{\upshape b)}] there exists $\gamma\in \Gamma_1$ such that $z=u\sigma_\gamma t_{\lambda+\mu}$;
           \item[{\upshape c)}] there exists  $\gamma\in \Gamma_2$ such that $z=u\sigma_\gamma t_{\lambda+\mu+\gamma^\vee}$. 
         \end{enumerate}
 \end{lemma}
 \begin{proof} Since \textbf{Case A} holds, it follows from the proof of Lemma \ref{sum66} that $ \lambda_1=\lambda$ and $v_2t_{\lambda_2}=ut_\mu.$ Furthermore,
      \begin{align*}
                       \ell(x)+\ell(y)+1= \ell([t_1t_2])= \ell([t_2t_1])
                                                             &=\ell([ut_{\mu}u^{-1}v_1t_{\lambda}])\\
                                                             &\leq \ell(ut_{\mu}u^{-1}v_1t_{\lambda})\\
                                                             &\leq \ell(ut_{\mu})+\ell(u^{-1}v_1t_{\lambda})\\
                                                             &= \ell(y) +\ell(t_{\lambda})-\ell(u^{-1}v_1)\\
                                                             &=\ell(y)+\ell(x)+1-\ell(u^{-1}v_1).
                      \end{align*}
        Hence,  $\ell(u^{-1}v_1)=0$ and therefore $v_1=u$.
        Hence, $[t_1t_2]=[ut_{\lambda}u^{-1}ut_{\mu}]=[ut_{\lambda+\mu}]$.

        Note that   $\ell(x)+\ell(y)\leq \ell(z)\leq  \ell(x)+\ell(y)+1=\ell([ut_{\lambda+\mu}])=\ell(ut_{\lambda+\mu})$.

         If $\ell(z)=\ell(x)+\ell(y)+1=\ell(ut_{\lambda+\mu})$, then the condition  $z\preccurlyeq[ut_{\lambda+\mu}]$ implies that  $z=ut_{\lambda+\mu}$. This is just case a).

         If $\ell(z)=\ell(x)+\ell(y)=\ell(ut_{\lambda+\mu})-1$, then the condition  $z\preccurlyeq[ut_{\lambda+\mu}]$ implies that
           $z\overset{\sigma_{\gamma+m\delta}}{\longrightarrow }ut_{\lambda+\mu}$ for some $\gamma+m\delta \in   R^+_{\scriptsize\mbox{re}}.$
        Note that $m\geq 0$ and that $z={\sigma_{\gamma+m\delta}} ut_{\lambda+\mu}=\sigma_\gamma ut_{m u^{-1}(\gamma)^\vee+\lambda+\mu}\in W_{\scriptsize\mbox{af}}^-.$
               Since $\ell({\sigma_{\gamma+m\delta}} ut_{\lambda+\mu})=\ell(z)<\ell(ut_{\lambda+\mu})$, it follows from Lemma \ref{rword} that
                           $(ut_{\lambda+\mu})^{-1}(\gamma+m\delta)=u^{-1}(\gamma)+\big(m+\langle\lambda+\mu, u^{-1}(\gamma)\rangle\big)\delta\in -R^+_{\scriptsize\mbox{re}}.$
          Hence, $m+\langle\lambda+\mu, u^{-1}(\gamma)\rangle\leq 0$.  Since $\lambda+\mu\in \tilde Q^\vee$   and $m\geq 0$, we must have
           $u^{-1}(\gamma)\in R^+$.
         Therefore, 
             \begin{align*} \ell(t_{\lambda+\mu})-\ell(u)-1&=\ell(z)\\
                                                           &=\langle m u^{-1}(\gamma)^\vee+\lambda+\mu, -2\rho\rangle -\ell(\sigma_\gamma u)\\
                                                               &=\ell(t_{\lambda+\mu})-m\langle  u^{-1}(\gamma)^\vee, 2\rho\rangle -\ell(u\sigma_{u^{-1}\gamma})\\
                                                               &\leq\ell(t_{\lambda+\mu})-m\langle  u^{-1}(\gamma)^\vee, 2\rho\rangle -\ell(u)+\ell(\sigma_{u^{-1}(\gamma)}) \\
                                                               &\leq\ell(t_{\lambda+\mu})-m\langle  u^{-1}(\gamma)^\vee, 2\rho\rangle -\ell(u)
                                                                               +\langle {u^{-1}(\gamma)^\vee}, 2\rho\rangle-1.
             \end{align*}
        Hence,  $(m-1)\langle {u^{-1}(\gamma)^\vee}, 2\rho\rangle\leq 0.$ Since $u^{-1}(\gamma)\in R^+$, $\langle {u^{-1}(\gamma)^\vee}, 2\rho\rangle>0$.
    Therefore, $0\leq m\leq 1$; that is, $m=0$ or 1.     Denote $\tilde \gamma=u^{-1}(\gamma)$. Note that $\tilde\gamma\in R^+$.

   If $m=0$, then $\tilde \gamma \in R^+, z=u\sigma_{\tilde \gamma}t_{\lambda+\mu}$ and $\ell(u\sigma_{\tilde \gamma})=\ell(u)+1$. This is just case b).

   If $m=1$, then $\tilde \gamma\in R^+, z=u\sigma_{\tilde \gamma}t_{\lambda+\mu+\tilde \gamma^\vee}$ and $\ell(u\sigma_{\tilde \gamma})
               =\ell(u)+1-\langle  \tilde \gamma^\vee, 2\rho\rangle$.
    This is just case c).
 \end{proof}

  \begin{lemma}\label{pos2}
   Under the same assumptions  as in Lemma \ref{sum66}, we assume that  $\langle \lambda, \alpha_j\rangle< -\ell(\omega_0)$  and $\langle \mu, \alpha_j\rangle< -\ell(\omega_0)$  for all $j\in  I $.
 If {\upshape\textbf{Case B}} occurs, 
     then only one of the following two possibilities can happen,
        \\ {\,\upshape a)}\, there exists $\gamma\in \Gamma_1$ such that  
             $v_1t_{\lambda_1}=u\sigma_\gamma t_\lambda$,  $v_2t_{\lambda_2}=u\sigma_\gamma t_\mu$,  $z=u\sigma_\gamma t_{\lambda+\mu}$; 
         \\{\,\upshape b)}\, there exists $\gamma\in \Gamma_2$ such that
                      $v_1t_{\lambda_1}=u\sigma_\gamma t_\lambda$, $v_2t_{\lambda_2}=u\sigma_\gamma t_{\mu+\gamma^\vee}$,
                    $z=u\sigma_\gamma t_{\lambda+\mu+\gamma^\vee}.$ 
 \end{lemma}
\begin{proof}
  Note that we have $z=[t_1t_2]$ in this case. Since $v_1t_{\lambda_1}\preccurlyeq \sigma_it_\lambda\preccurlyeq t_\lambda$, $\lambda\preccurlyeq \lambda_1$ by Corollary \ref{sum1}.
           Hence, $\ell(t_{\lambda_1})=\ell(t_\lambda)-2M$ for some $M\geq 0$. Therefore,
            \begin{align*}
                       \ell(x)+\ell(y)= \ell([t_1t_2])
                                      &=\ell([v_1t_{\lambda_1}v_1^{-1}v_2t_{\lambda_2}])\\
                                                             &\leq \ell(v_1t_{\lambda_1}v_1^{-1}v_2t_{\lambda_2})\\
                                                             &\leq \ell(v_1t_{\lambda_1})+\ell(v_1^{-1})+\ell(v_2t_{\lambda_2})\\
                                                             &= \ell(t_{\lambda_1}) +\ell(v_2t_{\lambda_2})\\
                                                             &\leq \ell(t_{\lambda_1})+ \ell(y)
                                                              =\ell(x)+1-2M+ \ell(y).
              \end{align*}
     Hence, $M=0, \,\, \lambda_1=\lambda$ and \,\,\,
          $\ell(y)\geq \ell(v_2t_{\lambda_2})\geq \ell(x)+\ell(y)-\ell(t_{\lambda_1})=\ell(y)-1. $
          \\     Hence, there are only the following two possibilities.
           \\ Case \,(\,\!i\,\!):\quad $\ell(v_2t_{\lambda_2})=\ell(y)$, which implies that $v_2t_{\lambda_2}=y=ut_\mu$.
           \\ Case (ii):\quad $\ell(v_2t_{\lambda_2})=\ell(y)-1$.

        Due to Lemma \ref{pos3} as below, Case (i) is impossible. It remains to discuss Case (ii).
         In this case,
           \begin{align*}
                       \ell(x)+\ell(y)= \ell([t_2t_1])
                                      &=\ell([v_2t_{\lambda_2}v_2^{-1}v_1t_{\lambda}])\\
                                                             &\leq \ell(v_2t_{\lambda_2}v_2^{-1}v_1t_{\lambda})\\
                                                             &\leq \ell(v_2t_{\lambda_2})+\ell(v_2^{-1} v_1t_{\lambda})\\
                                                             &= \ell(y)-1+\ell(t_{\lambda})- \ell(v_2^{-1} v_1)
                                                              =\ell(y)+\ell(x)- \ell(v_2^{-1} v_1).
              \end{align*}
          Hence, $\ell(v_2^{-1} v_1)=0$ and therefore $v_1=v_2$.

           Since $\ell(v_2t_{\lambda_2})=\ell(y)-1$ and $v_2t_{\lambda_2}\preccurlyeq y$, there exists
              $  \gamma+m\delta \in   R^+_{\scriptsize\mbox{re}}$ such that $v_2t_{\lambda_2}=\sigma_{\gamma+m\delta}ut_{\mu}.$ With the same
                discussion as in the proof of  Lemma \ref{pos1},  we have $\tilde \gamma=u^{-1}(\gamma)\in R^+$ and
                $ v_2t_{\lambda_2}=u\sigma_{\tilde \gamma}t_{\mu+m\tilde\gamma^\vee}\in W^-_{\scriptsize\mbox{af}}.$
                 Hence, $v_1=v_2=u\sigma_{\tilde\gamma}$ and $z=u\sigma_{\tilde\gamma}t_{\lambda+\mu+m\tilde \gamma^\vee}\in W^-_{\scriptsize\mbox{af}}$.
        With the same argument as in the proof of Lemma \ref{pos1} again,
         we can deduce that either $m=0$ or $m=1$.

         If $m=0$, then we have $\tilde\gamma\in R^+$ such that case a) holds; that is,
          $$v_1t_{\lambda_1}=u\sigma_{\tilde\gamma} t_\lambda;\quad v_2t_{\lambda_2}=u\sigma_{\tilde\gamma} t_\mu;
                    \quad  z=u\sigma_{\tilde\gamma} t_{\lambda+\mu};\qquad \ell(u\sigma_{\tilde\gamma})=\ell(u)+1.$$

        If $m=1$, then we have $\tilde\gamma\in R^+$ such that
          $v_1t_{\lambda_1}=u\sigma_{\tilde\gamma} t_\lambda,\,\, v_2t_{\lambda_2}=u\sigma_{\tilde\gamma} t_{\mu+\tilde \gamma^\vee}$,
                    \,\,  $z=u\sigma_{\tilde\gamma} t_{\lambda+\mu+\tilde \gamma^\vee}$ and  $\ell(u\sigma_{\tilde\gamma})=\ell(u)+1-\langle  \tilde\gamma^\vee, 2\rho\rangle;$
          that is,  case b) holds.
\end{proof}

 \begin{remark}\label{rem55}
  The condition  ``$\langle \mu, \alpha_j\rangle  <-\ell(\omega_0)$ for all $j\in I$" does imply that
   {\upshape $u\sigma_\gamma t_{\mu+\gamma^\vee}, u\sigma_\gamma t_{\lambda+\mu+\gamma^\vee}\in W^-_{\scriptsize\mbox{af}}$,} whenever $\gamma\in \Gamma_2$ and $\lambda\in\tilde Q^\vee$.

 Indeed, the statement can be checked directly for the case $|I|=n=1, 2$.  For any $\gamma\in R^+$,
       write $\gamma^\vee=\sum_ia_i\alpha_i^\vee$. Note that $\ell(\omega_0)=|R^+|\geq 9$ and $a_i\leq 4$ if $n=3, 4$, and that $\ell(\omega_0)>12$ and $a_i\leq 6$
          if $n\geq 5$ (see e.g. page 66 of \cite{hum}). Hence,  $\mu+\gamma^\vee\in \tilde Q^\vee$ is regular if $n\geq 3$.
  In particular, the statement holds.
 \end{remark}

 \begin{lemma}\label{pos3}
         {\upshape Case (i)} in the proof of Lemma \ref{pos2} can never occur.
 \end{lemma}
 \begin{proof}
         Assume Case (i) holds, then we have $\lambda_1=\lambda, v_2t_{\lambda_2}=ut_\mu$ and
           \begin{align*}
                      \ell(x)+\ell(y) =  \ell([t_2t_1])
                                                             &=\ell([ut_{\mu}u^{-1}v_1t_{\lambda}])\\
                                                             &\leq \ell(ut_{\mu}u^{-1}v_1t_{\lambda})\\
                                                             &\leq \ell(ut_{\mu})+\ell(u^{-1}v_1t_{\lambda})\\
                                                             &= \ell(y) +\ell(t_{\lambda})-\ell(u^{-1}v_1)
                                                             =\ell(y)+\ell(x)+1-\ell(u^{-1}v_1).
              \end{align*}
    Therefore, we have either $\ell(u^{-1}v_1)=0$ or $\ell(u^{-1}v_1)=1$.

  For the former case, we have $v_1=u$, and therefore $\ell(x)+\ell(y)=\ell([t_1t_2])=\ell(ut_{\lambda+\mu})=\ell(x)+\ell(y)+1$.
  This is a contradiction.

  For the latter case, $v_1=u\sigma_j$ for some $j$. 
    If $\langle \lambda, \alpha_j\rangle \leq \langle \mu, \alpha_j\rangle$, then $\lambda+\mu-\langle \mu, \alpha_j\rangle\alpha_j^\vee\in\tilde Q^\vee$. 
   Note that the integer  $\langle \mu, \alpha_j\rangle  <-\ell(\omega_0)$.  Therefore
     \begin{align*} \ell(x)+\ell(y) =  \ell([t_2t_1])
                                                             &=\ell([ut_{\mu}u^{-1}u\sigma_jt_{\lambda}])\\
                                                             &\leq \ell(ut_{\mu}\sigma_jt_{\lambda})\\
                                                             &=\ell(u\sigma_jt_{\mu-\langle \mu, \alpha_j\rangle\alpha_j^\vee+\lambda})\\
                                                             &\leq\ell(u\sigma_j)+ \ell(t_{\mu-\langle \mu, \alpha_j\rangle\alpha_j^\vee+\lambda})\\
                                                             &=\ell(u\sigma_j)+ \langle \mu-\langle \mu, \alpha_j\rangle\alpha_j^\vee+\lambda, -2\rho\rangle\\
                                                             &=\ell(u\sigma_j)+ \ell(t_{\lambda+ \mu})+2\langle \mu, \alpha_j\rangle\\
                                                             &\leq\ell(\omega_0)+\ell(x)+1+\ell(y)+\ell(u)-2\ell(\omega_0)-2
                                                              <\ell(x)+\ell(y).
     \end{align*}
  If  $\langle \lambda, \alpha_j\rangle > \langle \mu, \alpha_j\rangle$, then $\lambda+\mu-\langle \lambda, \alpha_j\rangle\alpha_j^\vee\in\tilde Q^\vee$.
   Therefore,
        \begin{align*} \ell(x)+\ell(y) =  \ell([t_1t_2])
                                                             &=\ell([u\sigma_jt_{\lambda}\sigma_ju^{-1}ut_{\mu}])\\
                                                             &\leq \ell(u\sigma_jt_{\lambda}\sigma_jt_{\mu})\\
                                                             &=\ell(ut_{\lambda-\langle \lambda, \alpha_j\rangle\alpha_j^\vee+\mu})\\
                                                             &\leq\ell(u)+ \ell(t_{\lambda-\langle \lambda, \alpha_j\rangle\alpha_j^\vee+\mu})\\
                                                             &=\ell(u)+ \langle \lambda-\langle \lambda, \alpha_j\rangle\alpha_j^\vee+\mu, -2\rho\rangle\\
                                                             &=\ell(u)+ \ell(t_{\lambda+ \mu})+2\langle \lambda, \alpha_j\rangle\\
                                                             &\leq\ell(\omega_0)+\ell(x)+1+\ell(y)+\ell(u)-2\ell(\omega_0)-2
                                                              <\ell(x)+\ell(y).
     \end{align*}
    Both cases deduce contradictions.
   Hence, Case (i) is impossible.
  \end{proof}

\bigskip

\begin{proof} [Proof of Proposition \ref{coe44}] Note that $\mathfrak{S}_x\mathfrak{S}_y
          = \sum_{z\in W_{\scriptsize\mbox{af}}^-}\sum_{t_1, t_2}c_{x, [t_1]}c_{y, [t_2]}d_{z, [t_1t_2]}\mathfrak{S}_z$, where the only
  nonzero terms are those satisfying $x\succcurlyeq [t_1], y\succcurlyeq [t_2], z\preccurlyeq [t_1t_2]$ and $\ell(z)\geq\ell(x)+\ell(y)$.
   Therefore,  our result follows from  Lemma \ref{pos1}, Lemma \ref{pos2} and Lemma \ref{sum88} immediately.
\end{proof}

 \bigskip

\begin{proof}[Proof of Proposition \ref{coe88}]

  $d_{\sigma_i, u}=w_i-u(w_i)$ holds by expanding the right side (with respect to a reduced expression of $u$) and comparing both sides.
  It follows  from the definition, Lemma  \ref{rword} and Lemma \ref{sum3} that
        $d_{w, w}=\prod_{\gamma\in R^+\atop w^{-1}(\gamma)\in -R^+}\gamma$ and
        $d_{wt_\lambda, wt_\lambda}=d_{w\omega_0, w\omega_0} \cdot\prod\limits_{j=r+1}^m w(H_j)
                        = \Big(\prod_{\gamma\in R^+\atop \omega_0 w^{-1}(\gamma)\in -R^+}\gamma\Big)\cdot\prod\limits_{j=r+1}^m w(H_j).$
    Note that  $\omega_0$ is an involution that maps $-R^+$ to $R^+$, that
      $\{\gamma\in R^+~|~ w(\gamma)\in R^+\}$ is $w$-invariant  and that
                      $\prod_{i=1}^r H_i= \prod_{\beta\in R^+}\beta$. Hence,
                    \begin{align*}
                            w(d_{w^{-1}, w^{-1}})\cdot d_{wt_{\lambda+\mu}, wt_{\lambda+\mu} }
                                   &= w(\prod_{\gamma\in R^+\atop w(\gamma)\in -R^+}\gamma)
                                             \Big(\prod_{\gamma\in R^+\atop w^{-1}(\gamma)\in R^+}\gamma\Big)\cdot\prod_{j=r+1}^{m+p}w(H_j) \\
                                  &=w\big(\prod_{\gamma\in R^+\atop w(\gamma)\in -R^+} \gamma\big)
                                         \cdot w\big(\prod_{\gamma\in R^+\atop w(\gamma)\in R^+}\gamma\big)\cdot\prod_{j=r+1}^{m+p}w(  H_j) \\ 
                                  &= w(\prod_{\gamma\in R^+}\gamma) \cdot\prod_{j=r+1}^{m+p}w( H_j) \\
                                 &= w(\prod_{j=1}^{r}  H_j )\cdot\prod_{j=r+1}^{m+p}w(  H_j)
                                  =\prod_{j=1}^{m+p}w(H_j).\hspace{1.6cm}
                    \end{align*}

   Let $ut_{\mu}=[\sigma_{\beta_{i_1}}\cdots \sigma_{\beta_{i_k}}]_{\scriptsize\mbox{red}}$,
      $ut_{\lambda+\mu}=[\sigma_{\beta_{i_1}}\cdots \sigma_{\beta_{i_s}}]_{\scriptsize\mbox{red}}$ ($k<s$) and denote $\gamma_j=\sigma_{\beta_{i_1}}\cdots \sigma_{\beta_{i_{j-1}}}(\beta_{i_j})$.
  Note that  $u\sigma_\gamma t_{\mu+\gamma^\vee}, u\sigma_\gamma t_{\lambda+\mu+\gamma^\vee}\in W^-_{\scriptsize\mbox{af}}$ whenever $\gamma\in \Gamma_2$, by Remark \ref{rem55}.
   Therefore for $\gamma\in \Gamma_2$, we have $\ell(u\sigma_\gamma t_{\lambda+\mu+\gamma^\vee})=\ell(ut_{\lambda+\mu})-1$ and $\ell(u\sigma_\gamma t_{\mu+\gamma^\vee})=\ell(ut_{\mu})-1$.
  Note that  $u(\gamma+\delta)= u(\gamma)+\delta\in R^+_{\scriptsize\mbox{re}}$ and
   $u\sigma_\gamma t_{\lambda+\mu+\gamma^\vee}=\sigma_{u(\gamma+\delta)}ut_{\lambda+\mu}$. 
   Hence,  there is a unique $1\leq j\leq s$
        such that $u\sigma_\gamma t_{\lambda+\mu+\gamma^\vee} =[\sigma_{\beta_{i_1}}\cdots\widehat{\sigma_{\beta_{i_j}}}\cdots \sigma_{\beta_{i_s}}]_{\scriptsize\mbox{red}}$ and
        $ {\gamma_j}=u(\gamma +\delta) $ by Lemma \ref{rword}. As a consequence, $d_{u\sigma_\gamma t_{\lambda+\mu+\gamma^\vee}, ut_{\lambda+\mu}}={1\over \gamma_j}\prod_{a=1}^s\gamma_a
                ={1\over u(\gamma+\delta)}d_{u  t_{\lambda+\mu}, ut_{\lambda+\mu}}$.
                Hence,  (2)   holds. Similarly, (1) also holds.

       With the same argument as above,
          there is a unique $1\leq j\leq k$ such that $\gamma_j=u(\gamma+\delta)$ and
           $u\sigma_\gamma t_{\mu+\gamma^\vee} =\sigma_{\beta_{i_1}}\cdots\widehat{\sigma_{\beta_{i_j}}}\cdots \sigma_{\beta_{i_k}}$.
         Denote $(a_1, \cdots, a_{k-1})=(i_1, \cdots, \hat i_j, \cdots, i_k)$ and denote $\tilde \gamma_b=\sigma_{\beta_{a_1}}\cdots \sigma_{\beta_{a_{b-1}}}(\beta_{a_b})$.
         Immediately, we have 
           $c_{ut_\mu, u\sigma_\gamma t_{\mu+\gamma^\vee}}=-(\gamma_j\prod_{b=1}^{k-1}\tilde \gamma_b)^{-1}$ and
           $d_{u\sigma_\gamma t_{\mu+\gamma^\vee}, u\sigma_\gamma t_{\mu+\gamma^\vee}}=\prod_{b=1}^{k-1}\tilde \gamma_b$
           by definition.  Therefore, (3) also holds by the following observation \\
        $\mbox{}$ \qquad $d_{u\sigma_\gamma t_{\lambda+\mu+\gamma^\vee}, u\sigma_\gamma t_{\lambda+\mu+\gamma^\vee}}=
               d_{u\sigma_\gamma t_{\mu+\gamma^\vee}, u\sigma_\gamma t_{\mu+\gamma^\vee}}\cdot \prod_{j=1}^m u\sigma_\gamma t_{\mu+\gamma^\vee}(H_j)$.
\end{proof}

\subsection{Equivariant quantum cohomology of $X=G/B$} \label{appendequiquantum}

 The Lie group $G$ possesses a so-called   Bruhat decomposition $G= \bigsqcup_{w\in W} {B}\dot w {B}$,  labelled by elements in the Weyl group $W$.
  It induces a decomposition
    of $X:=G/B$ into Schubert cells: $X=\bigsqcup_{w\in W}B\dot wB/B$, in which $B\dot w B/B \cong \mathbb{C}^{\ell(w)}$.
     The closures $X_{w}:=\overline{B\dot w B/B}$ are called Schubert varieties in $X$. Let
      $\sigma_w$ denote the image of the fundamental class $[X_w]$ under the canonical map
      $H_*(X_w, \mathbb{Z})\rightarrow H_*(X, \mathbb{Z})$. Then $\sigma_w\in H_{2\ell(w)}(X, \mathbb{Z})$ and
      these Schubert homology classes $\sigma_w$'s form an additive basis of
      $H_*(X, \mathbb{Z})$. The cohomology group $H^*(X, \mathbb{Z})$ also  has an additive basis of Schubert cohomology classes
       $\sigma^w$'s such that  $\langle\sigma_u, \sigma^v\rangle=\delta_{u, v}$ for any $u, v\in W$. If we write $g^{u, v}=\int_{[X]}\sigma^u\cup \sigma^v$,
  then the matrix  $\big(g^{u, v}\big)$ is invertible with its inverse denoted as
    $\big(g_{u, v}\big)=\big(g^{u, v}\big)^{-1}$.

 For each $i\in I$, we denote $s_i=\sigma_{\alpha_i}$ and introduce a formal variable $q_i$.
   Identify      $ H_2(X, \mathbb{Z})=\bigoplus_{i\in I}\mathbb{Z}\sigma_{s_{i}}$ with $Q^\vee$ via
          $\beta=\sum_id_i\sigma_{s_i}\mapsto\lambda_\beta=\sum_id_i\alpha_i^\vee$. Denote $q_{\lambda_\beta}=q^\beta=\prod_{i\in I}q_i^{d_i}$.

 Let $\overline{\mathcal{M}}_{0, m}(X, \beta)$ be the Kontsevich's moduli space of stable maps of degree $\beta$ of $m$-pointed genus 0 curves into $X$ (see \cite{fupa}).
 Let $\mbox{ev}_i$ denote the $i$-th canonical evaluation map $\mbox{ev}_i: \overline{\mathcal{M}}_{0, m}(X, \beta)\to X$ given by
    $\mbox{ev}_i([f: C\to X; p_1, \cdots, p_m])=f(p_i)$. The genus zero Gromov-Witten invariant for $\gamma_1,\cdots, \gamma_m\in H^*(X)=H^*(X, \mathbb{Q})$ is defined as
         $$I_{0, m, \beta}(\gamma_1, \cdots, \gamma_m)=\int_{\overline{\mathcal{M}}_{0, m}(X, \beta)}\mbox{ev}_1^* (\gamma_1)\cup\cdots
                            \cup \mbox{ev}_m^*(\gamma_m). $$

  The
   (small) quantum product  for $a, b\in H^*(X)$ is a deformation of the cup product defined as follows.
   $$a\star b= \sum_{u, v\in W; \beta\in H_2(X, \mathbb{Z})} I_{0, 3, \beta}  (a, b, \sigma^u)g_{u, v}\sigma^v q^{\beta}  .$$

 The  $\mathbb{Q}[\mathbf{q}]$-module
         $H^*(X)[\mathbf{q}]:=H^*(X)\otimes\mathbb{Q}[\mathbf{q}]$  equipped with $\star$ is called the
      small quantum cohomology ring of $X$ and denoted as  $QH^*(X)$. So the same Schubert classes $\sigma^u=\sigma^u\otimes 1$ form
      a basis for $QH^*(X)$ over $\mathbb{Q}[\mathbf{q}]$ and we write
                $$\sigma^u\star \sigma^v = \sum_{w\in W, \lambda\in Q^\vee}    N_{u,v}^{w, \lambda}q_{\lambda}\sigma^w.$$
  The coefficients $N_{u, v}^{w, \lambda}$'s are called the \textit{quantum Schubert structure constants.}
  In fact, $\sum_{v_1\in W}g_{v_1, w}\sigma^{v_1}=\sigma^{\omega_0w}$ (see e.g. \cite{fw}). Compared with the original definition of quantum product,
   the quantum Schubert structure constant $N_{u, v}^{w, \lambda}$ is exactly equal to the (3-pointed genus zero) Gromov-Witten invariant
    $I_{0, 3, \lambda}(\sigma^u, \sigma^v, \sigma^{\omega_0w})$.   When $\lambda=0$, they give the classical Schubert structure constants  for $H^*(X)$.
 The $T$-action on $X$ induces an action on the
  moduli space   $\overline{\mathcal{M}}_{0,3}(X, \beta)$ given by:
 $t\cdot (f: C\rightarrow X; p_1, p_2,  p_3) = ( f_t:C\rightarrow X; p_1,p_2, p_3)$
 where $ f_t(x) := t\cdot f(x)$.  The evaluation maps $\mbox{ev}_i$'s are
  $T$-equivariant.
   We use the same notation $\sigma^u$ to denote the equivariant Schubert class in $H^*_T(X)$. The
     equivariant Gromov-Witten invariant is defined as $I^T_{0, 3, \beta}(\sigma^u, \sigma^v, \sigma^w)=\pi_*^T(\mbox{ev}_1^T(\sigma^u)\cdot \mbox{ev}_2^T(\sigma^v)\cdot
               \mbox{ev}_3^T(\sigma^w))$, where $\pi_*^T$ is the equivariant Gysin push forward. As a consequence, the equivariant (small) quantum product
                 $\star_T$ is defined (see e.g. \cite{mih}).
  The equivariant   quantum cohomology ring  $QH^*_T(X)=(H^*(X)[\mathbf{\alpha},\mathbf{q}], \star_T)$ is commutative and associative, which has an $S[\mathbf{q}]$-basis of Schubert classes
     with $S[\mathbf{q}]=\mathbb{Q}[\alpha_1, \cdots, \alpha_n, q_1, \cdots, q_n]$. Furthermore,
      $$\sigma^u\star_T \sigma^v = \sum_{w\in W, \lambda\in Q^\vee}     \tilde N_{u,v}^{w, \lambda} q_{\lambda}\sigma^w,\quad \mbox{where }    \tilde N_{u,v}^{w, \lambda}= \tilde N_{u,v}^{w, \lambda}(\alpha)\in S=\mathbb{Q}[\alpha_1,\cdots,\alpha_n].$$

\subsection{Equivariant cohomology of $\Omega K$}\label{equi}

  The affine Kac-Moody group $\mathcal{G}$  possesses a Bruhat decomposition $ \bigsqcup_{x\in W_{\scriptsize \mbox{af}}}\mathcal{B}x\mathcal{B}$,
       where the canonical identification $W_{\scriptsize \mbox{af}}\cong N(\hat T_\mathbb{C})/\hat T_\mathbb{C}$ is used.
         Here  $\hat T_\mathbb{C}=\mbox{Hom}_\mathbb{Z}(\hat{\mathfrak{h}}_\mathbb{Z}^*, \mathbb{C}^*)$ denotes the standard maximal torus of $\mathcal{G}$,
       in which $\hat{\mathfrak{h}}_\mathbb{Z}$ is the integral form of $\hat{\mathfrak{h}}=\mathfrak{h}_{\scriptsize \mbox{af}}$ (see e.g. chapter 6 of \cite{kumar}).
     The Bruhat decomposition of $\mathcal{G}$ induces a decomposition
    of $\mathcal{G}/\mathcal{P}_Y$ into Schubert cells: $\mathcal{G}/\mathcal{P}_Y=\bigsqcup_{x\in W_{\scriptsize \mbox{af}}^{Y}}\mathcal{B}x\mathcal{P}_Y/\mathcal{P}_Y.$
      Schubert varieties  are the closures of $\mathcal{B}x\mathcal{P}_Y/\mathcal{P}_Y$'s in $\mathcal{G}/\mathcal{P}_Y$. Let
      $\mathfrak{S}_x^Y$ denote the image of the fundamental class $[\overline{\mathcal{B}x\mathcal{P}_Y/\mathcal{P}_Y}]$ under the canonical map
      $H_*(\overline{\mathcal{B}x\mathcal{P}_Y/\mathcal{P}_Y})\rightarrow H_*(\mathcal{G}/\mathcal{P}_Y)$. Then
      $H_*(\mathcal{G}/\mathcal{P}_Y, \mathbb{Z})$ has an additive basis of Schubert homology classes $\{\mathfrak{S}_x^Y~|~x\in {W}_{\scriptsize\mbox{af}}^Y\}$. We denote $\mathfrak{S}_x=\mathfrak{S}^Y_x$
      wherever there is no confusion. 
      Similarly, the cohomology group $H^*(\mathcal{G}/\mathcal{P}_Y, \mathbb{Z})$ has an additive basis of Schubert cohomology classes $\{\mathfrak{S}^x~|~x\in W_{\scriptsize \mbox{af}}^Y\}$,
      where $\langle \mathfrak{S}_x, \mathfrak{S}^y\rangle=\delta_{x, y}$ with respect to the natural pairing.

      The standard maximal torus $\hat T_\mathbb{C}$ of $\mathcal{G}$
        has complex dimension $n+2$ with maximal compact sub-torus $\hat T= \mbox{Hom}_\mathbb{Z}(\hat{\mathfrak{h}}_\mathbb{Z}^*, \mathbb{S}^1)$.
       With respect to the natural action of   $\hat T_\mathbb{C}$  on $\mathcal{G}/\mathcal{B}$, we consider the equivariant cohomology $H^*_{\hat T}(\mathcal{G}/\mathcal{B})$,
       which is an $\hat S$-module with $\hat S=S[\hat{\mathfrak{h}}_\mathbb{Z}^*]=H^*_{\hat T}(\mbox{pt})$.
       Note that the 1-dimensional sub-torus $\mathbb{C}^*$, which comes from the central extension, acts on $\mathcal{G}/\mathcal{B}$ trivially.
        As a consequence, 
         the equivariant Schubert structure constants
         are polynomials in $\mathbb{Q}[\delta, \alpha_1, \cdots, \alpha_n]\subset \hat S$ only. 
        Since we are concerned with the non-trivial part of the $\hat{T}_{\mathbb{C}}$-action only, we denote
        $\hat S= \mathbb{Q}[\delta, \alpha_1, \cdots, \alpha_n]= \mathbb{Q}[\alpha_0, \alpha_1, \cdots, \alpha_n]$ by abusing of notations.
         $H^*_{\hat T}(\mathcal{G}/\mathcal{B})$ is an $\hat S$-module spanned by the basis of equivariant Schubert classes (see e.g. \cite{kumar} for concrete definitions), which we also denote as
         $\{\mathfrak{S}^{x}~|~ x\in W_{\scriptsize \mbox{af}}\}$ simply.    Via the embedding $\pi^*: H^*_{\hat T}(\mathcal{G}/\mathcal{P}_Y) \hookrightarrow  H^*_{\hat T}(\mathcal{G}/\mathcal{B})$ induced by
    the natural projection $\pi: \mathcal{G}/{\mathcal{B}} \rightarrow\mathcal{G}/{\mathcal{P}_Y}$,
     the equivariant cohomology   $H^*_{\hat T}(\mathcal{G}/\mathcal{P}_Y)$ is also an $\hat S$-module spanned by
       the  basis of equivariant Schubert classes $\{\mathfrak{S}^{x}~|~x\in W_{\scriptsize \mbox{af}}^Y\}$. As a consequence,
        equivariant Schubert structure constants  for $\mathcal{G}/\mathcal{P}_Y$ are covered by
        equivariant Schubert structure constants  for $\mathcal{G}/\mathcal{B}$.

 \begin{remark}\label{rema}
      For  $\mathfrak{S}^{x}, \mathfrak{S}^{y}\in H^*_{\hat T}(\mathcal{G}/\mathcal{B})$,
                {\upshape $\mathfrak{S}^{x} \mathfrak{S}^{y} =\sum_{z\in W_{\scriptsize\mbox{af}}}p_{x, y}^{\,z}\mathfrak{S}^{z} $}.
         The    equivariant Schubert structure constant  $p_{x, y}^{\,z}$ is a polynomial in $\hat S$. In terms of combination of rational functions,
                     one has {\upshape $p_{x, y}^{\,z}=\sum_{v\in W_{\scriptsize\mbox{af}}}d_{x, v}d_{y, v}c_{z, v}$}
                     (see e.g. chapter 11 of \cite{kumar}).
 \end{remark}

   Let  $L_{\scriptsize\mbox{an}}K=\{f\in \mathcal{G}~|~ f(\mathbb{S}^1)\subset K\}$ and
        $\Omega_{\scriptsize\mbox{an}}K=\{f\in L_{\scriptsize\mbox{an}}K~|~ f(1_{\mathbb{S}^1})= 1_K\}$.
     Note that each $f\in \mathcal{G}$ can be written as $f(t)=f_K(t)\cdot f_P(t)$ for some unique $f_K\in \Omega_{\scriptsize\mbox{an}}K$ and $f_P\in\mathcal{P}_0$.
     Therefore we can realize $\mathcal{G}/\mathcal{P}_0$ as $\Omega_{\scriptsize\mbox{an}}K$, which is homotopy-equivalent to $\Omega K$, via the ($L_{\scriptsize\mbox{an}}K$-equivariant) homeomorphism $ \mathcal{G}/\mathcal{P}_0\rightarrow \Omega_{\scriptsize\mbox{an}}K$ (see \cite{press} and references therein
      for more details). Since we are concerned with properties at the level of (co)homology only, we do not distinguish between $\Omega_{\scriptsize\mbox{an}}K$ and $\Omega K$.
      The Bruhat decomposition of $\mathcal{G}/\mathcal{P}_0$ readily gives a Bruhat decomposition of $\Omega K$. As
       a consequence and by abusing notations, we know that $H_*(\Omega K, \mathbb{Z})$ (resp. $H^*(\Omega K, \mathbb{Z})$) has an additive $\mathbb{Z}$-basis of
       Schubert (co)homology classes $\{\mathfrak{S}_x (\mbox{resp. } \mathfrak{S}^x)\,~|~ x\in  W_{\scriptsize\mbox{af}}^-\}$.

        The (non-trivial part of the) $\hat T$-action on $\mathcal{G}/\mathcal{P}_0$ corresponds to the natural $\mathbb{S}^1\times T$ action on $\Omega K$, which
         consists of the rotation action of $\mathbb{S}^1$ on $\Omega K$ and the action of $T$ on $\Omega K$ by pointwise conjugation.
          By considering the $T$-action only, we obtain the   evaluation maps $\mbox{ev}: H^*_{\hat T}(\mathcal{G}/\mathcal{P}_0)\rightarrow H^*_T(\mathcal{G}/\mathcal{P}_0)$ and
           $ {ev}: \hat S=H^*_{\hat T}(\mbox{pt})\rightarrow H^*_T(\mbox{pt})=S$,
           where the $T$-equivariant cohomology $H^*_T(\mathcal{G}/\mathcal{P}_0)$ is an $S$-module with $S=\mathbb{Q}[\alpha_1, \cdots, \alpha_n]$.
                The image of the null root $\delta=\alpha_0+\theta$
          in $S$ is 0.  More precisely, we have $H^*_{\hat T}(\mathcal{G}/\mathcal{P}_0)=\mbox{Span}_{\hat S}
                 \{\hat{\mathfrak{S}}^{x}~|~ x\in W^-_{\scriptsize\mbox{af}}\}$ and
                  $H^*_T(\mathcal{G}/\mathcal{P}_0)=\mbox{Span}_S\{\mathfrak{S}^{x}~|~ x\in W^-_{\scriptsize\mbox{af}}\}$. Let $f=f(\alpha_0, \alpha_1,\cdots, \alpha_n)\in \hat S$, then
                 we have $  {ev}(f)=f(-\theta, \alpha_1, \cdots, \alpha_n)\in S$ and $\mbox{ev}(f\hat{\mathfrak{S}}^{x})=ev(f)\mathfrak{S}^{x}$.

    \begin{remark}\label{rema22}
                {\upshape $\mathfrak{S}^{x} \mathfrak{S}^{y} =\sum_{z\in W_{\scriptsize\mbox{af}}^-}\tilde p_{x, y}^{\,z}\mathfrak{S}^{z} $}.
         The    $T$-equivariant Schubert structure constant    $\tilde p_{x, y}^{\,z}$   is a polynomial in $S$. It follows from Remark \ref{rema}
         and Lemma \ref{coe99} that
                  {\upshape $\tilde p_{x, y}^{\,z}=\sum_{v\in W_{\scriptsize\mbox{af}}^-}d_{x, [v]}d_{y, [v]}c_{z, [v]}$} as combination of rational functions.
    \end{remark}

   \begin{remark}\label{rema33}
              The    $T$-equivariant Schubert structure constant    $p_{u, v}^w$ for $ G/B $ can   also be expressed
              in terms of $c_{u, v}$ and $d_{u, v}$. The polynomial $p_{u, v}^w$ is given by
                  {\upshape $  p_{u, v}^{\,w}=\sum_{v_1\in W }d_{u, v_1}d_{v, v_1}c_{w, v_1}$} as combination of rational functions  (see e.g. \cite{kumar}).
    \end{remark}

  \begin{remark}
       Because of the natural $S$-module isomorphism $H_T^*(\Omega K\times \Omega K)\cong  H_T^*(\Omega K)\otimes_S H_T^*(\Omega K)$, the Pontryagin  product $\Omega K\times \Omega K\rightarrow \Omega K$,  which is
                   associative and  $T$-equivariant,
     induces a coassociative coproduct  $H_T^*(\Omega K)\rightarrow  H_T^*(\Omega K)\otimes H_T^*(\Omega K)$.
  In particular, this gives an alternative definition
of the $T$-equivariant homology of $\mathcal{G}/\mathcal{P}_0$, which doesn't use Borel-Moore
homology. Indeed, define $H^T_*(\mathcal{G}/\mathcal{P}_0)$ to be the submodule  of $\mbox{Hom}_S(H^*_T(\mathcal{G}/\mathcal{P}_0), S)$ spanned by
      those  $ \mathfrak{S}_{x}\in \mbox{Hom}_S(H^*_T(\mathcal{G}/\mathcal{P}_0), S)$
       which for any {\upshape $x, y\in W_{\scriptsize\mbox{af}}^-$} satisfy $\langle \mathfrak{S}_{x}, \mathfrak{S}^{y}\rangle= \delta_{x, y}$  with respect to the natural pairing.
     Then the product of $H^T_*(\mathcal{G}/\mathcal{P}_0)$,   induced from the coproduct of   $H_T^*(\mathcal{G}/\mathcal{P}_0)=H_T^*(\Omega K)$,
        makes $H^T_*(\mathcal{G}/\mathcal{P}_0)$ an $S$-algebra.
       Note that the elements $\mathfrak{S}_x$ coincide with
the integration operators $\mathcal{L}_w$ defined in \cite{ara} Prop. 2.5.1, so Arabia's localization formula can be applied.
        Thus the whole proof of the formula for structure coefficients $b_{x, y}^{\, z}$'s still goes through.
  \end{remark}

\section*{Acknowledgements}
 The first author is supported in part by a RGC research grant from the Hong
Kong Government. Both authors thank Alberto Arabia, Haibao Duan,
Sergey Fomin, Thomas Lam, Peter Magyar, Victor Reiner, John R.
Stembridge, Xiaowei Wang and Yongchang Zhu for useful discussions. We also thank the referee for valuable comments and suggestions.
The second author is particularly grateful to Thomas Lam for
considerable help.
\bibliographystyle{amsplain}

\begin{thebibliography}{99}


\bibitem{ara}A. Arabia,\,{\it Cohomologie $T$-\'equivariante de la varit\'et\'e de drapeaux d'un groupe de Ka\v{c}-Moody}, Bull.  Soc. Math. France 117 (1989), no. 2, 129-165.
\bibitem{borel}A. Borel,\,{\it Sur la cohomologie des espaces fibr\'es principaux et des espaces homog\'enes de groupes de Lie compacts},  Ann. of Math. (2) 57, (1953), 115--207.

\bibitem{bu0}A.S. Buch,\,{\it Quantum cohomology of Grassmannians}, Compo. Math. 137(2003), 227--235.

\bibitem{bu}A.S. Buch,\,{\it Quantum cohomology of partial flag manifolds}, Trans. Amer. Math. Soc. 357 (2005), no. 2, 443--458.



\bibitem{cmn} P.E. Chaput, L. Manivel, N. Perrin,\,{\it Quantum cohomology of minuscule homogeneous spaces}, Transform. Groups 13 (2008), no. 1, 47--89.
\bibitem{duan}H. Duan,\,{\it Multiplicative rule of Schubert class}, Invent. Math. 159 (2005), no. 2, 407--436.
\bibitem{fomin}S. Fomin,\,{\it Lecture notes on quantum cohomology of the flag manifold}, Geometric combinatorics (Kotor, 1998), Publ. Inst. Math. (Beograd) (N.S.) 66 (80) (1999), 91--100.
\bibitem{fominGP}S. Fomin, S. Gelfand, A. Postnikov,\,{\it Quantum Schubert polynomials}, J. Amer. Math. Soc. 10 (1997), no. 3, 565--596.
\bibitem{fu11}W. Fulton,\,{\it On the quantum cohomology of homogeneous varieties}, The legacy of Niels Henrik Abel, 729--736, Springer, Berlin, 2004.
\bibitem{fu22}W. Fulton,\,{\it Young tableaux: with applications to representation theory and geometry},  Cambridge University Press, Cambridge, 1997.

\bibitem{fupa}W. Fulton, R. Pandharipande,\,{\it Notes on stable maps and quantum cohomology},
                   Proc. Sympos. Pure Math. 62, Part 2, Amer. Math. Soc., Providence, RI, 1997.
\bibitem{fw} W. Fulton, C. Woodward,\,{\it On the quantum product of Schubert classes}, J. Algebraic Geometry  13 (2004), no. 4, 641-661.


\bibitem{grah}W. Graham,\,{\it Positivity in equivariant Schubert calculus}, Duke Math. J. 109, no. 3 (2001), 599-614.


\bibitem{hill}H. Hiller,\,{\it  The geometry of Coxeter groups},  Boston : Pitman Pub., c1982.
\bibitem{hum} J.E. Humphreys,\,{\it Introduction to Lie algebras and representation theory}, Graduate Texts in Mathematics 9, Springer-Verlag, New York-Berlin,   1980.
 \bibitem{hump} J.E. Humphreys,\,{\it Reflection groups and Coxeter groups}, Cambridge University Press, Cambridge, UK, 1990.
 \bibitem{kac} V.G. Kac,\,{\it Infinite-dimensional Lie algebras},Cambridge University Press, Cambridge, 1990.
\bibitem{kim} B. Kim,\,{\it Quantum cohomology of flag manifolds $G/B$ and quantum Toda lattices}, Ann. of Math. (2) 149 (1999), no. 1, 129--148.
\bibitem{kntao} A. Knutson, T. Tao,\,{\it The honeycomb model of ${\rm GL}\sb n(\Bbb C)$ tensor products. I. Proof of the saturation conjecture},  J. Amer. Math. Soc. 12 (1999), no. 4, 1055--1090.
\bibitem{kntao22} A. Knutson, T. Tao, C. Woodward,\,{\it The honeycomb model of ${\rm GL}\sb n(\Bbb C)$ tensor products. II. Puzzles determine facets of the Littlewood-Richardson cone}, J. Amer. Math. Soc. 17 (2004), no. 1, 19--48.
\bibitem{koku} B. Kostant, S. Kumar,\,{\it The nil Hecke ring and the cohomology of $G/P$ for a Kac-Moody group $G$}, Adv. in Math. 62 (1986), 187-237.
\bibitem{kt1}A. Kresch, H. Tamvakis,\,{\it Quantum cohomology of orthogonal Grassmannians}, Compos. Math. 140 (2004), no. 2, 482--500.
\bibitem{kt2}A. Kresch, H. Tamvakis,\,{\it  Quantum cohomology of the Lagrangian Grassmannian}, J. Algebraic Geometry 12 (2003), no. 4, 777--810.

\bibitem{kumar} S. Kumar,\,{\it Kac-Moody groups, their flag varieties and representation theory}, Progress in Mathematics 204, Birh\"auser Boston, Inc., Boston, MA, 2002.
\bibitem{lam} T. Lam,\,{\it Schubert polynomials for the affine Grassmannian}, J. Amer. Math. Soc. 21 (2008), no. 1, 259--281.
\bibitem{lamshi} T. Lam, M. Shimozono,\,{\it Quantum cohomology of $G/P$ and homology of affine Grassmannian}, Acta Math. 204 (2010), no. 1, 49--90. 

\bibitem{lamlmsh}T. Lam, L. Lapointe, J. Morse, M. Shimozono,\,{\it Affine insertion and Pieri rules for the affine Grassmannian},
                       Memoirs of the AMS, to appear; arxiv: math.CO/0609110.
\bibitem{lamschshi}T. Lam, A. Schilling, M. Shimozono,\,{\it Schubert polynomials for the affine Grassmannian of the symplectic group},
                        arxiv: math. AG/0710.2720.
\bibitem{czli}C. Li,\,{\it Quantum cohomology of homogeneous varieties}, Ph.D. theis, Chinese University of Hong Kong, 2009.

\bibitem{lus}G. Lusztig,\,{\it Singularities, character formulas, and a q-analog of weight multiplicities}, in ``Analysis and topology on singular spaces II-III", Ast\'erisque 101-102(1983), 208-229.
\bibitem{mag}P. Magyar,\,{\it Notes on Schubert classes of a loop group},  arXiv: math. RT/0705.3826.
\bibitem{mare}A.-L. Mare,\,{\it Polynomial representatives of Schubert classes in $QH^*(G/B)$}, Math. Res. Lett. 9 (2002), no. 5-6, 757--769.
\bibitem{mih} L.C. Mihalcea,\,{\it Equivariant quantum cohomology of homogeneous spaces},  Duke Math. J. 140 (2007), no. 2, 321--350.
\bibitem{mi2}L.C. Mihalcea,\,{\it Positivity in equivariant quantum Schubert calculus}, Amer. J. Math. 128 (2006), no. 3, 787--803.
\bibitem{peterson} D. Peterson,\,{\it Quantum cohomology of $G/P$}, Lecture notes at MIT, 1997 (notes by J. Lu and K. Rietsch).
\bibitem{press} A. Pressley, G. Segal,\,{\it Loop groups},  Clarendon Press, Oxford, 1986.


\bibitem{stem} J.R. Stembridge,\,{\it Tight quotients and double quotients in the Bruhat order}, Electron. J. Combin. 11 (2004/06), no. 2, Research Paper 14-41.
\bibitem{vafa} C. Vafa,\, {\it Topological mirrors and quantum rings}, in: Essays on mirror manifolds (ed. S.T. Yau), International Press 1992, 96-119.

\bibitem{wo}C.T. Woodward,\,{\it On D. Peterson's comparison formula for Gromov-Witten invariants of $G/P$},
                                                       Proc. Amer. Math. Soc. 133 (2005), no. 6,   1601--1609.

 \end{thebibliography}

\end{document}